\documentclass[11pt,reqno]{amsart}
\usepackage{amsmath,amsfonts,amssymb}
\usepackage[paperwidth=20.5cm,paperheight=29.7cm,width=17cm,height=25cm,top=24mm,left=19mm,headsep=11mm]{geometry}

\theoremstyle{theorem}
\newtheorem{theorem}{Theorem}[section]
\newtheorem{proposition}[theorem]{Proposition}
\newtheorem{lemma}[theorem]{Lemma}
\newtheorem{condition}[theorem]{Condition}
\newtheorem{remark}[theorem]{Remark}
\newtheorem{corollary}[theorem]{Corollary}
\newtheorem{example}[theorem]{Example}
\numberwithin{equation}{section}

\theoremstyle{plain}

\usepackage[T2A]{fontenc}
\usepackage[cp1251]{inputenc}
\usepackage{amsfonts}
\usepackage[english]{babel}
\usepackage{amsmath,amsthm,amssymb}

\def\Ker{\mathrm{Ker}\,}
\def\rank{\mathrm{rank}\,}
\def\Dom{\mathrm{Dom}\,}
\def\clos{\mathrm{clos}\,}
\def\sgn{\mathrm{sign}\,}
\def\mes{\mathrm{meas}\,}
\def\diag{\mathrm{diag}\,}
\def\esssup{\mathrm{ess}\text{-}\mathrm{sup}\,}

\def\le{\leqslant}

\def\ge{\geqslant}

\begin{document}

\title[Spectral approach to homogenization of hyperbolic equations]{Spectral approach to homogenization \\of hyperbolic equations with periodic coefficients}

\author{M.~A.~Dorodnyi and T.~A.~Suslina}


\address{St. Petersburg State University, Universitetskaya nab. 7/9, St.~Petersburg, 199034, Russia}

\email{mdorodni@yandex.ru}

\email{t.suslina@spbu.ru}


\begin{abstract}
In $L_2(\mathbb{R}^d;\mathbb{C}^n)$, we consider selfadjoint strongly elliptic second order differential operators ${\mathcal A}_\varepsilon$
with periodic coefficients depending on ${\mathbf x}/ \varepsilon$, $\varepsilon>0$.
We study the behavior of the operators $\cos( {\mathcal A}^{1/2}_\varepsilon \tau)$ and
${\mathcal A}^{-1/2}_\varepsilon \sin( {\mathcal A}^{1/2}_\varepsilon \tau)$, $\tau \in \mathbb{R}$, for small $\varepsilon$. Approximations for these
operators in the $(H^s\to L_2)$-operator norm with a suitable $s$ are obtained. The results are used to study the behavior of
the solution ${\mathbf v}_\varepsilon$ of the Cauchy problem for the hyperbolic equation
$\partial^2_\tau {\mathbf v}_\varepsilon = - \mathcal{A}_\varepsilon {\mathbf v}_\varepsilon +\mathbf{F}$.
General results are applied to the acoustics equation and the system of elasticity theory.
\end{abstract}

\thanks{Supported by Russian Foundation for Basic Research (grant no.~16-01-00087)}

\maketitle

\section*{Introduction\label{Intr}}

The paper concerns  homogenization for periodic differential operators (DOs). A broad literature is devoted to homogenization problems;
first, we mention the books \cite{BeLP, BaPa, ZhKO}.
For homogenization problems in  ${\mathbb R}^d$, one of the methods is the spectral approach based on the Floquet-Bloch theory; see, e.~g.,
\cite[Chapter~4]{BeLP}, \cite[Chapter~2]{ZhKO}, \cite{Se}, \cite{COrVa}.

\subsection{The class of operators\label{sec0.1}}
We consider selfadjoint elliptic second order DOs in $L_2(\mathbb{R}^d;\mathbb{C}^n)$ admitting a factorization of the form
\begin{equation}
\label{intro_A}
\mathcal{A} = f(\mathbf{x})^* b(\mathbf{D})^* g(\mathbf{x}) b(\mathbf{D}) f(\mathbf{x}).
\end{equation}
Here $b(\mathbf{D})=\sum_{l=1}^d b_l D_l$ is the $(m\times n)$-matrix first order DO.
We assume that $m\ge n$ and that the symbol $b(\boldsymbol{\xi})$ has maximal rank.
The matrix-valued functions  $g(\mathbf{x})$ (of size $m\times m$) and $f(\mathbf{x})$
(of size $n\times n$) are assumed to be periodic with respect to some lattice $\Gamma$ and such that
$g(\mathbf{x}) >0$; $g,g^{-1} \in L_\infty$; $f,f^{-1} \in L_\infty$.
It is convenient to start with a simpler class of operators
\begin{equation}
\label{intro_hatA}
\widehat{\mathcal{A}} = b(\mathbf{D})^* g(\mathbf{x}) b(\mathbf{D}).
\end{equation}
Many operators of mathematical physics can be represented in the form \eqref{intro_A} or~\eqref{intro_hatA}; see \cite{BSu1}.
The simplest example is the acoustics operator $\widehat{\mathcal{A}} = - \hbox{\rm div}\, g(\mathbf{x}) \nabla =
\mathbf{D}^* g(\mathbf{x})\mathbf{D}$.

Now we introduce the small parameter  $\varepsilon > 0$ and
denote $\varphi^\varepsilon(\mathbf{x}):= \varphi(\varepsilon^{-1}\mathbf{x})$ for any $\Gamma$-periodic function $\varphi(\mathbf{x})$.
Consider the operators
\begin{gather}
\label{intro_A_eps}
\mathcal{A}_{\varepsilon} = f^{\varepsilon}(\mathbf{x})^* b(\mathbf{D})^* g^{\varepsilon}(\mathbf{x}) b(\mathbf{D}) f^{\varepsilon}(\mathbf{x}),\\
\label{intro_hatA_eps}
\widehat{\mathcal{A}}_{\varepsilon} = b(\mathbf{D})^* g^{\varepsilon}(\mathbf{x}) b(\mathbf{D}).
\end{gather}

\subsection{Operator error estimates for elliptic and parabolic problems in $\mathbb{R}^d$\label{sec0.2}}
 In a series of papers \cite{BSu1,BSu2,BSu3,BSu4} by Birman and Suslina, an operator-theoretic approach to homogenization of elliptic equations in
$\mathbb{R}^d$  was suggested and developed. This approach is based on the scaling transformation, the Floquet-Bloch theory, and the analytic perturbation theory.

Let us talk about the simpler operator \eqref{intro_hatA_eps}.  In \cite{BSu1}, it was proved that
\begin{equation}
\label{intro_resolv_est}
\| (\widehat{\mathcal{A}}_{\varepsilon} + I )^{-1} - (\widehat{\mathcal{A}}^0 + I )^{-1} \|_{L_2(\mathbb{R}^d) \to L_2(\mathbb{R}^d)} \le C \varepsilon.
\end{equation}
Here $\widehat{\mathcal{A}}^0 = b(\mathbf{D})^* g^0 b(\mathbf{D})$ is the \textit{effective operator} with the constant \textit{effective matrix} $g^0$. Approximations for $(\widehat{\mathcal{A}}_{\varepsilon} +I)^{-1}$
in the $(L_2\to L_2)$-norm with an error $O(\varepsilon^2)$ and in the $(L_2\to H^1)$-norm with an error $O(\varepsilon)$
(with correctors taken into account) were obtained in \cite{BSu2, BSu3} and \cite{BSu4}, respectively.

The operator-theoretic approach was applied to parabolic problems in \cite{Su1, Su2, Su3, V, VSu}.
In \cite{Su1,Su2}, it was proved that
\begin{equation}
\label{intro_parab_exp_est}
\| e^{-\tau \widehat{\mathcal{A}}_{\varepsilon}} - e^{-\tau \widehat{\mathcal{A}}^0} \|_{L_2(\mathbb{R}^d) \to L_2(\mathbb{R}^d)} \le C \varepsilon (\tau + \varepsilon^2)^{-1/2},\quad \tau >0.
\end{equation}
Approximations for $e^{-\tau \widehat{\mathcal{A}}_{\varepsilon}}$
in the $(L_2\to L_2)$-norm with an error $O(\varepsilon^2)$ and in the $(L_2\to H^1)$-norm with an error $O(\varepsilon)$
 were found in \cite{V} and \cite{Su3}, respectively.
Even more accurate approximations for the semigroup and the resolvent of $\widehat{\mathcal{A}}_{\varepsilon}$
were found in \cite{VSu}.

Estimates of the form \eqref{intro_resolv_est}, \eqref{intro_parab_exp_est} are called \textit{operator error estimates} in homogenization theory.
They are order-sharp. A different approach to operator error estimates
was suggested by Zhikov; see \cite{Zh, ZhPas1, ZhPas2} and a recent survey \cite{ZhPas3}.

\subsection{Operator error estimates for nonstationary Schr\"odinger-type and hyperbolic equations\label{intro_BSu5_Su5_rewiew_section}}
The situation with  homogenization of nonstationary  Schr\"odinger-type and hyperbolic equations
is different from the case of elliptic and parabolic problems. The paper \cite{BSu5} is devoted to such problems. Again, we dwell on the results for the simpler operator \eqref{intro_hatA_eps}.
In operator terms, the behavior of  the operators $e^{-i \tau \widehat{\mathcal{A}}_{\varepsilon}}$ and
$\cos(\tau \widehat{\mathcal{A}}_{\varepsilon}^{1/2})$  (where $\tau \in \mathbb{R}$) for small $\varepsilon$  is studied.
For these operators it is impossible to obtain approximations in the $(L_2 \to L_2)$-norm,
and we are forced to consider the norm of operators acting from the Sobolev space $H^s (\mathbb{R}^d; \mathbb{C}^n)$ (with appropriate $s$)
to $L_2 (\mathbb{R}^d; \mathbb{C}^n)$. In \cite{BSu5}, the following estimates were proved:
\begin{gather}
\label{intro_exp_est}
\| e^{-i \tau \widehat{\mathcal{A}}_{\varepsilon}} - e^{-i \tau \widehat{\mathcal{A}}^0} \|_{H^3 (\mathbb{R}^d) \to L_2 (\mathbb{R}^d)} \le \widetilde{C}(1 + |\tau|)\varepsilon, \\
\label{intro_cos_est}
\| \cos(\tau \widehat{\mathcal{A}}_{\varepsilon}^{1/2}) - \cos(\tau (\widehat{\mathcal{A}}^0)^{1/2}) \|_{H^2 (\mathbb{R}^d) \to L_2 (\mathbb{R}^d)} \le C(1 + |\tau|)\varepsilon.
\end{gather}
A similar result for the operator  $\widehat{\mathcal A}^{-1/2}_\varepsilon \sin(\tau \widehat{\mathcal A}^{1/2}_\varepsilon )$
was recently obtained by Meshkova \cite{M}:
\begin{equation}
\label{intro_sin_est}
\| \widehat{\mathcal{A}}_{\varepsilon}^{-1/2} \sin(\tau \widehat{\mathcal{A}}_{\varepsilon}^{1/2}) - (\widehat{\mathcal{A}}^0)^{-1/2} \sin(\tau (\widehat{\mathcal{A}}^0)^{1/2}) \|_{H^1 (\mathbb{R}^d) \to L_2 (\mathbb{R}^d)} \le \check{C}(1 + |\tau|)\varepsilon.
\end{equation}
By interpolation, the operator in \eqref{intro_exp_est} in the $(H^s \to L_2)$-norm is of order
$O(\varepsilon^{s/3})$ (where $0\le s \le 3$), for
the operator in \eqref{intro_cos_est} the $(H^s \to L_2)$-norm is of order $O(\varepsilon^{s/2})$ (where $0\le s \le 2$),
and for
the operator in \eqref{intro_sin_est} the $(H^s \to L_2)$-norm is of order $O(\varepsilon^{s})$ (where $0\le s \le 1$).

Let us explain the method; we comment on the proof of estimate \eqref{intro_cos_est}.
 Denote ${\mathcal H}_0 := -\Delta$.
Clearly, \eqref{intro_cos_est} is equivalent to
\begin{equation}
\label{intro_est2}
\| \bigl( \cos(\tau \widehat{\mathcal{A}}_{\varepsilon}^{1/2}) - \cos(\tau (\widehat{\mathcal{A}}^0)^{1/2}) \bigr) (\mathcal{H}_0 + I )^{-1} \|_{L_2(\mathbb{R}^d) \to L_2(\mathbb{R}^d)} \le C(1+ |\tau|)\varepsilon.
\end{equation}
By the scaling transformation, \eqref{intro_est2} is equivalent to
\begin{equation}
\label{intro_est3}
\| \bigl( \cos(\tau \varepsilon^{-1} \widehat{\mathcal{A}}^{1/2}) - \cos(\tau \varepsilon^{-1} (\widehat{\mathcal{A}}^0)^{1/2})\bigr)
 \varepsilon^{2} (\mathcal{H}_0 + \varepsilon^{2} I )^{-1} \|_{L_2(\mathbb{R}^d) \to L_2(\mathbb{R}^d)} \le C (1+ |\tau|) \varepsilon.
\end{equation}

Next, $\widehat{\mathcal{A}}$ expands in the direct integral of the operators $\widehat{\mathcal{A}}(\mathbf{k})$ acting in $L_2 (\Omega; \mathbb{C}^n)$
(where $\Omega$ is the cell of the lattice $\Gamma$) and given by
$b(\mathbf{D} + \mathbf{k})^* g(\mathbf{x}) b(\mathbf{D} + \mathbf{k})$ with periodic boundary conditions.
The spectrum of  $\widehat{\mathcal{A}}(\mathbf{k})$ is discrete. The family of operators $\widehat{\mathcal{A}}(\mathbf{k})$ is studied by means of the analytic perturbation theory
(with respect to the onedimensional parameter $t = |\mathbf{k}|$). For the operators $\widehat{\mathcal{A}}(\mathbf{k})$
the analog of estimate \eqref{intro_est3} is proved with
the constant independent of  $\mathbf{k}$. This yields \eqref{intro_est3}.

Recently, in~\cite{Su4} (see also~\cite{Su5}) it was shown that estimate~\eqref{intro_exp_est} is sharp in the following sense:
there exist operators for which estimate
$\| e^{-i \tau \widehat{\mathcal{A}}_{\varepsilon}} - e^{-i \tau \widehat{\mathcal{A}}^0} \|_{H^s \to L_2}  \le C(\tau) \varepsilon$
is false if $s < 3$. On the other hand, under some additional assumptions the result can be improved:
the analog of \eqref{intro_exp_est} in the $(H^2 \to L_2)$-norm holds.

\subsection{Main results of the paper\label{intro_main_results_section}}
We study the behavior of the operators  $\cos(\tau \widehat{\mathcal{A}}_{\varepsilon}^{1/2})$ and
$\widehat{\mathcal{A}}_{\varepsilon}^{-1/2} \sin(\tau \widehat{\mathcal{A}}_{\varepsilon}^{1/2})$ for small $\varepsilon$.
On the one hand, we confirm the sharpness of estimates \eqref{intro_cos_est} and \eqref{intro_sin_est} in the following sense.
We find a condition under which the norm in \eqref{intro_cos_est} cannot be replaced by the $(H^s \to L_2)$-norm with some $s<2$
and the norm in \eqref{intro_sin_est} cannot be replaced by the $(H^r \to L_2)$-norm with some $r<1$.
This condition is formulated  in the spectral terms. Consider the operator family
$\widehat{\mathcal{A}}(\mathbf{k})$ and put
 $\mathbf{k} = t \boldsymbol{\theta}$, $t = |\mathbf{k}|$, $\boldsymbol{\theta} \in \mathbb{S}^{d-1}$.
This family is analytic with respect to the parameter~$t$. For $t=0$ the point $\lambda_0=0$ is an eigenvalue of multiplicity $n$ of the ``unperturbed''
operator $\widehat{\mathcal{A}}(0)$. Then for small $t$ there exist the real-analytic branches of the eigenvalues and the eigenvectors of
$\widehat{\mathcal{A}}(t \boldsymbol{\theta})$.
For small $t$ the eigenvalues $\lambda_l (t, \boldsymbol{\theta})$, $l=1,\dots,n,$ admit the convergent power series expansions
\begin{equation}
\label{intro_eigenvalues_series}
\lambda_l(t, \boldsymbol{\theta}) = \gamma_l (\boldsymbol{\theta}) t^2 + \mu_l (\boldsymbol{\theta}) t^3 + \ldots,
\quad \; \quad   l = 1, \ldots, n,
\end{equation}
where $\gamma_l(\boldsymbol{\theta})>0$ and $\mu_l(\boldsymbol{\theta})\in \mathbb{R}$.
The condition is that $\mu_l (\boldsymbol{\theta}_0) \ne 0$ for some $l$ and some  $\boldsymbol{\theta}_0 \in \mathbb{S}^{d-1}$.
Examples of the operators satisfying this condition are provided.

On the other hand, under some additional assumptions we improve the results and obtain the analog of
\eqref{intro_cos_est} in the $(H^{3/2}\to L_2)$-norm and the analog of \eqref{intro_sin_est} in the $(H^{1/2}\to L_2)$-norm.
If  $n=1$, for this it suffices that $\mu (\boldsymbol{\theta}) = \mu_1 (\boldsymbol{\theta})=0$, $\boldsymbol{\theta} \in \mathbb{S}^{d-1}$.
This is the case for the operator $- \operatorname{div} g^\varepsilon (\mathbf{x}) \nabla$, where $g(\mathbf{x})$
is symmetric matrix with real entries.
If $n\ge 2$, besides the condition that all the coefficients $\mu_l (\boldsymbol{\theta})$
in \eqref{intro_eigenvalues_series}  are equal to zero, we impose one more condition in terms of the coefficients
$\gamma_l (\boldsymbol{\theta})$, $l=1,\dots,n$. The simplest version of this condition is that the branches
$\gamma_l (\boldsymbol{\theta})$ do not intersect.

For more general operator \eqref{intro_A_eps},
we obtain analogs of the results described above
for the  operators $f^\varepsilon \cos(\tau \mathcal{A}_{\varepsilon}^{1/2}) (f^\varepsilon)^{-1}$
and $f^\varepsilon \mathcal{A}_{\varepsilon}^{-1/2} \sin(\tau \mathcal{A}_{\varepsilon}^{1/2}) (f^\varepsilon)^*$.

Next, we apply the results given in operator terms to study the behavior of the solution of the Cauchy
problem for hyperbolic equations. In particular, we consider the nonstationary acoustics equation and the system of elasticity theory.

\subsection{Method\label{method}}
The results are obtained by further development of the operator-theoretic approach.
We follow the plan described above in Subsection~\ref{intro_BSu5_Su5_rewiew_section}.
Considerations are based on the abstract operator-theoretic scheme.
We study the family of operators $A(t)=X(t)^* X(t)$ acting
in some Hilbert space $\mathfrak{H}$. Here $X(t)=X_0 + tX_1$. (The family $A(t)$ models the operator family $\mathcal{A}(\mathbf{k}) = \mathcal{A}(t \boldsymbol{\theta})$, but the parameter $\boldsymbol{\theta}$ is absent in the abstract setting.) It is assumed that
the point $\lambda_0=0$ is an eigenvalue of $A(0)$ of finite multiplicity $n$.
Then for $|t|\le t^0$ the perturbed operator $A(t)$ has exactly $n$ eigenvalues on the interval $[0,\delta]$
(here $\delta$ and $t^0$ are controlled explicitly).
These eigenvalues and the corresponding eigenvectors are real-analytic functions of $t$.
The coefficients of the corresponding power series expansions
are called \textit{threshold characteristics} of the operator $A(t)$.
We distinguish a finite rank operator $S$ (the \textit{spectral germ} of  $A(t)$) which acts in the space
$\mathfrak{N} = \Ker A(0)$.
The germ is related to the threshold characteristics of principal order.

 In terms of the spectral germ, it is possible to approximate the operators $\cos(\varepsilon^{-1} \tau A(t)^{1/2})$ and
$A(t)^{-1/2} \sin(\varepsilon^{-1} \tau A(t)^{1/2})$.
  Application of the abstract results leads to the required estimates for DOs.
  However, at this step additional difficulties arise. They concern the improvement of the results in the case where
  $\mu_l (\boldsymbol{\theta})=0$, $l=1,\dots,n$. In the general case,
we are not always able to make our constructions and estimates uniform in $\boldsymbol{\theta}$, and we are forced to
  impose the additional assumption of isolation of the branches $\gamma_l (\boldsymbol{\theta})$, $l=1,\dots,n$.

\subsection{Plan of the paper}
The paper consists of three chapters. Chapter~1 (Sections 1--4) contains the necessary operator-theoretic material;
here main results in the abstract setting are obtained.
In Chapter~2 (Sections~5--10), the periodic DOs of the form~\eqref{intro_A}, \eqref{intro_hatA} are studied.
In Section~5 we describe the class of operators and the direct integral expansion for operators of the form~\eqref{intro_A};
the corresponding family of operators ${\mathcal A}(\mathbf{k})$ is incorporated in the framework of the abstract scheme.
In Section~6 we describe the effective characteristics for the operator~\eqref{intro_hatA}.
In Section~7, using the abstract results, we obtain the required approximations for the operator functions of
$\widehat{\mathcal A}(\mathbf{k})$. The operator ${\mathcal A}(\mathbf{k})$ is considered in Section~8.
In Section~9, we find approximations for the operator functions of
 ${\mathcal A}(\mathbf{k})$. Next, in Section~10
from the results of Sections~7 and~9 we deduce the required approximations for the operator functions of
 the operators $\widehat{\mathcal A}$ and ${\mathcal A}$. Chapter~3 (Sections 11--14) is devoted to homogenization problems.
In Section~11, by the scaling transformation, we deduce main results (approximations
for the operator functions of $\widehat{\mathcal A}_\varepsilon$ and ${\mathcal A}_\varepsilon$) from the results of Chapter~2.
  In Section~12, the results are applied to the Cauchy problem for
hyperbolic equations. Sections~13 and~14 are devoted to applications of the general results to the particular equations.

\subsection{Notation}
Let $\mathfrak{H}$ and $\mathfrak{H}_{*}$ be complex separable Hilbert spaces.
The symbols $(\cdot, \cdot)_{\mathfrak{H}}$ and $ \| \cdot \|_{\mathfrak{H}}$ stand for the inner product and the norm in $\mathfrak{H}$, respectively;
the symbol $\| \cdot \|_{\mathfrak{H} \to \mathfrak{H}_*}$ stands for the norm of a linear continuous operator from $\mathfrak{H}$ to $\mathfrak{H}_{*}$. Sometimes we omit the indices. By $I = I_{\mathfrak{H}}$ we denote the identity operator in $\mathfrak{H}$.
If $\mathfrak{N}$ is a subspace of $\mathfrak{H}$, then $\mathfrak{N}^{\perp} : = \mathfrak{H} \ominus \mathfrak{N}$.
If $P$ is the orthogonal projection of $\mathfrak{H}$ onto $\mathfrak{N}$, then $P^\perp$ is the orthogonal projection onto
 $\mathfrak{N}^{\perp}$.
If $A: \mathfrak{H} \to \mathfrak{H}_*$ is a linear operator, then $\Dom A$ and $\Ker A$ stand for its domain and its kernel.

The symbols $\langle\cdot, \cdot \rangle$ and $|\cdot|$ denote the inner product and the norm in $\mathbb{C}^n$;
$\boldsymbol{1}_n$ is the unit $(n \times n)$-matrix. If $a$ is an $(n\times n)$-matrix, then the symbol $|a|$
denotes the norm of $a$ viewed as a linear operator in $\mathbb{C}^n$.
 Next, we use the notation $\mathbf{x} = (x_1,\dots, x_d) \in \mathbb{R}^d$, $i D_j = \partial_j = \partial / \partial x_j$,
$j=1,\dots,d$, $\mathbf{D} = -i \nabla = (D_1,\dots, D_d)$.
The $L_p$-classes (where $1 \le p \le \infty$) and the Sobolev classes of $\mathbb{C}^n$-valued functions in a domain
$\mathcal{O} \subset \mathbb{R}^d$ are denoted by $L_p (\mathcal{O}; \mathbb{C}^n)$ and
$H^s (\mathcal{O}; \mathbb{C}^n)$, respectively.
Sometimes we write simply $L_p({\mathcal O})$, $H^s({\mathcal O})$.

\smallskip
Part of the results of the present paper was announced in \cite{DSu}.

\section*{Chapter~1. Abstract operator-theoretic scheme}

\section{Quadratic operator pencils\label{abstr_section_1}}

The material of this section is borrowed from \cite{BSu1}, \cite{BSu2}, and \cite{Su4}.

\subsection{The operators $X(t)$ and $A(t)$\label{abstr_X_A_section}}
Let $\mathfrak{H}$ and $\mathfrak{H}_{*}$ be complex separable Hilbert spaces. Suppose that
$X_0: \mathfrak{H} \to \mathfrak{H}_{*}$ is a densely defined and closed operator, and
$X_1: \mathfrak{H} \to \mathfrak{H}_{*}$ is a bounded operator.
Then the operator $X(t) := X_0 + tX_1$, $t \in \mathbb{R}$, is closed on the domain $\Dom X(t) = \Dom X_0$.
Consider the family of selfadjoint (and nonnegative) operators
$A(t) := X(t)^* X(t)$ in $\mathfrak{H}$.
The operator $A(t)$ is generated by the closed quadratic form $\| X(t) u \|^{2}_{\mathfrak{H}_*}$,
$u \in \Dom X_0$. Denote $A(0) = X_0^*X_0 =: A_0$ and
$\mathfrak{N} := \Ker  A_0 = \Ker X_0$.
We impose the following condition.

\begin{condition}\label{cond1}
The point $\lambda_0 = 0$ is an isolated point of the spectrum of $A_0$, and $0 < n := {\rm{dim}}\, \mathfrak{N} < \infty$.
\end{condition}

Let $d^0$ be the \textit{distance from the point $\lambda_0 = 0$ to the rest of the spectrum of} $A_0$.
We put $\mathfrak{N}_*:= \Ker X_0^*$, $n_* := \dim \mathfrak{N}_*$, and \textit{assume that} $n \le n_* \le \infty$.
Let $P$ and $P_*$  be the orthogonal projections of $\mathfrak{H}$ onto $\mathfrak{N}$ and of
$\mathfrak{H}_*$ onto  $\mathfrak{N}_*$, respectively.
Denote by $F(t;[a,b])$ the spectral projection of $A(t)$ for the interval $[a,b]$, and put
${\mathfrak F}(t;[a,b]):=F(t; [a,b]) \mathfrak{H}$. \textit{We fix a number $\delta>0$ such that} $8 \delta < d^0$.
We write $F(t)$ in place of $F(t;[0,\delta])$ and ${\mathfrak F}(t)$ in place of ${\mathfrak F}(t;[0,\delta])$.
Next, we choose a number $t^0 >0$ such that
\begin{equation}
\label{abstr_t0_fixation}
t^0 \le \delta^{1/2} \|X_1\|^{-1}.
\end{equation}
According to \cite[Chapter~1, (1.3)]{BSu1},
$F(t; [0,\delta]) = F(t;[0, 3 \delta])$ and $\rank F(t; [0,\delta]) = n$ for $|t| \le t^0$.

\subsection{The operators $Z$, $R$, and $S$\label{abstr_Z_R_S_op_section}}
Now we introduce some operators appearing in the analytic perturbation theory considerations; see \cite[Chapter~1, Section~1]{BSu1},
\cite[Section~1]{BSu2}.

Let $\omega \in \mathfrak{N}$, and let $\psi = \psi(\omega) \in \Dom X_0 \cap \mathfrak{N}^{\perp}$
be a (weak) solution of the equation \hbox{$X^*_0 (X_0 \psi + X_1 \omega) = 0$}.
We define a bounded operator $Z : \mathfrak{H} \to \mathfrak{H}$ by the relation $Zu = \psi (P u)$, $u \in \mathfrak{H}$.
Note that $Z$  takes $\mathfrak{N}$ to $\mathfrak{N}^\perp$ and $\mathfrak{N}^\perp$ to $\{0\}$.
Let $R : \mathfrak{N} \to \mathfrak{N}_*$ be the operator defined by
$R \omega =  (X_0 Z + X_1)\omega$, $\omega \in \mathfrak{N}$.
Another representation for $R$ is given by $R = P_* X_1\vert_{\mathfrak{N}}$.

According to \cite[Chapter~1]{BSu1}, the operator  $S:=R^*R: \mathfrak{N} \to \mathfrak{N}$
is called the \textit{spectral germ of the operator family $A(t)$ at} $t=0$.
The germ $S$ can be represented as $S= P X_1^* P_* X_1\vert_{\mathfrak{N}}.$
The spectral germ is said to be \textit{nondegenerate} if ${\rm Ker}\, S = \{ 0\}$.
We have
\begin{equation}
\label{abstr_Z_R_S_est}
\| Z \| \le (8 \delta)^{-1/2} \| X_1 \|, \quad \| R \| \le \| X_1 \|, \quad   \| S \| \le \| X_1 \|^2.
\end{equation}

\subsection{The analytic branches of eigenvalues and eigenvectors  of $A(t)$\label{branches}}

According to the general analytic perturbation theory (see \cite{Ka}), for $|t|\le t^0$ there exist real-analytic functions $\lambda_l(t)$
and real-analytic $\mathfrak{H}$-valued functions $\varphi_l(t)$ such that
$A(t) \varphi_l(t) = \lambda_l (t) \varphi_l(t)$, $l = 1, \dots, n$,
and the $\varphi_l(t)$, $l=1,\dots,n$, form an \textit{orthonormal basis} in ${\mathfrak F}(t)$.
Moreover, for $|t|\le t_*$, where $0< t_* \le t^0$ \textit{is sufficiently small}, we have the following convergent power series expansions:
\begin{align}
\label{abstr_A(t)_eigenvalues_series}
\lambda_l(t) &= \gamma_l t^2 + \mu_l t^3 + \dots, \quad \gamma_l \ge 0, \quad \; \mu_l \in \mathbb{R}, \quad   l = 1, \dots, n, \\
\label{abstr_A(t)_eigenvectors_series}
\varphi_l (t) &= \omega_l + t \psi_l^{(1)} + \dots, \quad	l = 1, \dots, n.
\end{align}
The elements $\omega_l= \varphi_l(0)$, $l=1,\dots,n,$ form an orthonormal basis in  $\mathfrak{N}$.

In \cite[Chapter~1, Section~1]{BSu1} and  \cite[Section~1]{BSu2} it was checked that
$\widetilde{\omega}_l := \psi_l^{(1)} - Z \omega_l \in \mathfrak{N}$,
\begin{equation}
\label{abstr_S_eigenvectors}
S \omega_l = \gamma_l \omega_l , \quad l = 1, \ldots, n,
\end{equation}
and $(\widetilde{\omega}_l , \omega_j) + (\omega_l, \widetilde{\omega}_j)=0$, $l,j =1,\dots,n$.
Thus, the \textit{numbers $\gamma_l$ and the elements $\omega_l$
are eigenvalues and eigenvectors of the germ} $S$.
We have
$P = \sum_{l=1}^{n} (\cdot, \omega_l) \omega_l$ and  $SP = \sum_{l=1}^{n} \gamma_l (\cdot, \omega_l) \omega_l$.

\subsection{Threshold approximations\label{threshold}}

The following statement was obtained in \cite[Chapter~1, Theorems~4.1~and~4.3]{BSu1}.
In what follows, \emph{we denote by $\beta_j$ various  absolute constants assuming that} $\beta_j \ge 1$.

\begin{proposition}[\cite{BSu1}]
Under the assumptions of Subsection~\emph{\ref{abstr_X_A_section}}, for $|t| \le t^0$ we have
    \begin{align}
    \label{abstr_F(t)_threshold_1}
    \| F(t) - P \| &\le C_1 |t|; \quad   C_1 = \beta_1 \delta^{-1/2} \| X_1 \|,
\\
\label{1.6a}
    \| A(t)F(t) - t^2 SP \| &\le C_2 |t|^3; \quad C_2 = \beta_2 \delta^{-1/2}\| X_1 \|^3.
    \end{align}
\end{proposition}

From \eqref{abstr_t0_fixation}, \eqref{abstr_Z_R_S_est}, and \eqref{1.6a} it follows that
\begin{equation}
\label{1.6b}
\| A(t) F(t)\| \le (1+\beta_2) \| X_1 \|^2 t^2,\quad |t| \le t^0.
\end{equation}

We also need a more precise approximation for the operator $A(t)F(t)$; see  \cite[Theorem~4.1]{BSu2}.

\begin{proposition}[\cite{BSu2}]
\label{th1.3}
Under the assumptions of Subsection~\emph{\ref{abstr_X_A_section}},
for \hbox{$|t|\le t^0$} we have
    \begin{align}
     \label{abstr_A(t)_threshold_2}
    A(t) F(t) &= t^2 SP + t^3 K + \Psi (t),
\\
    \label{abstr_C5}
    \| \Psi (t) \| &\le C_3 t^4; \quad C_3 = \beta_3 \delta^{-1}\| X_1 \|^4.
    \end{align}
The operator  $K$ is represented as $K = K_0 + N = K_0 + N_0 + N_*$,
where $K_0$ takes $\mathfrak{N}$ to $\mathfrak{N}^{\perp}$ and $\mathfrak{N}^{\perp}$ to $\mathfrak{N}$, while $ N_0$ and $ N_*$
take $\mathfrak{N}$ to itself and $\mathfrak{N}^{\perp}$ to $\{ 0 \}$.
 In terms of the power series coefficients, these operators are given by
$K_0 = \sum_{l=1}^{n} \gamma_l \left( (\cdot, Z \omega_l) \omega_l + (\cdot, \omega_l) Z \omega_l \right)$,
\begin{equation}
\label{abstr_N_0_N_*}
N_0 = \sum_{l=1}^{n} \mu_l (\cdot, \omega_l) \omega_l, \quad N_* = \sum_{l=1}^{n} \gamma_l \left( (\cdot, \widetilde{\omega}_l) \omega_l + (\cdot, \omega_l) \widetilde{\omega}_l\right) .
\end{equation}
In the invariant terms, we have $K_0 = Z S P + S P Z^*$ and $N = Z^*X_1^* R P + (RP)^* X_1 Z$.
The operators $N$ and $K$ satisfy the following estimates
\begin{equation}
\label{abstr_K_N_estimates}
\|N\| \le (2 \delta)^{-1/2} \| X_1 \|^3, \quad \|K\| \le 2(2 \delta)^{-1/2} \| X_1 \|^3.
\end{equation}
\end{proposition}

\begin{remark}
    \label{abstr_N_remark}
 $1^\circ$. In the basis $\{\omega_l\}$, the operators $N$, $N_0$, and $N_*$
 \emph{(}restricted to $\mathfrak{N}$\emph{)}
are represented by $(n\times n)$-matrices. The operator $N_0$ is diagonal{\rm :}
            $(N_0 \omega_j, \omega_k ) = \mu_j \delta_{jk}$, $j, k = 1, \ldots ,n$.
The matrix entries of $N_*$ are given by
         $(N_* \omega_j, \omega_k) = \gamma_k (\omega_j, \widetilde{\omega}_k) + \gamma_j (\widetilde{\omega}_j, \omega_k ) = ( \gamma_j - \gamma_k)(\widetilde{\omega}_j, \omega_k )$, $j, k = 1,\ldots, n$.
    Hence, the diagonal elements of $N_*$ are equal to zero. Moreover,
           $(N_* \omega_j, \omega_k) = 0$ if $\gamma_j = \gamma_k$.
 $2^\circ$.
If $n=1$, then $N_*=0$, i.~e., $N=N_0$.
\end{remark}

\subsection{The nondegeneracy condition\label{abstr_nondegenerated_section}}
Below we impose the following additional condition.

\begin{condition}\label{nondeg}
There exists a constant $c_*>0$ such that $A(t) \ge c_* t^2 I$ for $|t| \le t^0$.
\end{condition}

From Condition \ref{nondeg} it follows that $\lambda_l(t) \ge c_* t^2$, $l=1,\dots,n$, for $|t|\le t^0$.
By \eqref{abstr_A(t)_eigenvalues_series}, this implies
$\gamma_l \ge c_* > 0$, $l= 1, \ldots, n$, i.~e., the germ $S$ is nondegenerate:
\begin{equation}
\label{abstr_S_nondegenerated}
S \ge c_* I_{\mathfrak{N}}.
\end{equation}

\subsection{The clusters of eigenvalues of $A(t)$\label{abstr_cluster_section}}

The content of this subsection is borrowed from ~\cite[Section~2]{Su4} and concerns the case where $n \ge 2$.

Suppose that Condition~\ref{nondeg} is satisfied.
Now it is convenient to change the notation tracing the multiplicities of the eigenvalues of $S$.
Let $p$ be the number of different eigenvalues of the germ $S$.
We enumerate these eigenvalues in the increasing order
and denote them by $\gamma_j^\circ$, $j=1,\dots,p$.
Their multiplicities are denoted by $k_1,\dots, k_p$ (obviously, $k_1+\dots+k_p =n$).
Let $\mathfrak{N}_j := \Ker(S - \gamma^{\circ}_j I_\mathfrak{N})$, $j = 1 ,\ldots, p$. Then
$\mathfrak{N} = \sum_{j=1}^{p} \oplus \mathfrak{N}_j$.
Let $P_j$ be the orthogonal projection of $\mathfrak{H}$ onto $\mathfrak{N}_j$. Then
\begin{equation}
\label{abstr_P_Pj}
P = \sum_{j=1}^{p} P_j, \qquad P_j P_l = 0 \quad \text{for} \  j \ne l.
\end{equation}

We also change the notation for the eigenvectors of the germ
(which are the ``embryos'' in \eqref{abstr_A(t)_eigenvectors_series}) dividing them in $p$ parts so that
$\omega_1^{(j)},\dots, \omega^{(j)}_{k_j}$ correspond to the eigenvalue
 $\gamma_j^\circ$ and form an orthonormal basis in $\mathfrak{N}_j$.
We also change the notation for  the eigenvalues and the eigenvectors of $A(t)$.
The eigenvalue and the eigenvector whose expansions \eqref{abstr_A(t)_eigenvalues_series} and \eqref{abstr_A(t)_eigenvectors_series}
start with the terms $\gamma_j^\circ t^2$ and $\omega^{(j)}_q$
are denoted by $\lambda^{(j)}_q(t)$ and $\varphi^{(j)}_q(t)$.

So, for $|t|\le t^0$ the first $n$ eigenvalues of $A(t)$ are divided in $p$ clusters:
the $j$-th  cluster
consists of the eigenvalues $\lambda^{(j)}_q(t)$, $q=1,\dots,k_j$.
We have $\lambda^{(j)}_q(t)= \gamma_j^\circ t^2 + \mu_q^{(j)} t^3+\dots$,
 $q=1,\dots,k_j$.

For each pair of indices $(j,l)$,  $1\le j,  l \le p$, $j \ne l$, we denote
\begin{equation}
\label{abstr_c_circ_jl}
c^{\circ}_{jl} := \min \{c_*, n^{-1} |\gamma^{\circ}_l - \gamma^{\circ}_j|\}.
\end{equation}
Clearly, there exists a number $i_0 = i_0(j,l)$, where $j \le i_0 \le l-1$ if $j<l$
and $l \le i_0 \le j-1$ if $l<j$, such that $\gamma^\circ_{i_0+1} - \gamma_{i_0}^\circ \ge c^\circ_{jl}$.
It means that on the interval between $\gamma_{j}^\circ$ and $\gamma_{l}^\circ$
there is a gap in the spectrum of $S$ of length at least $c^\circ_{jl}$.
If such $i_0$ is not unique, we agree to take the minimal possible $i_0$ (for definiteness).
  Next, we choose a number $t^{00}_{jl}\le t^0$ such that
\begin{equation}
\label{abstr_t00_jl}
t^{00}_{jl} \le (4C_2)^{-1} c^{\circ}_{jl} = (4 \beta_2)^{-1} \delta^{1/2} \|X_1\|^{-3 } c^{\circ}_{jl}.
\end{equation}
Let $\Delta^{(1)}_{jl}:=[\gamma_1^\circ - c_{jl}^\circ/4, \gamma^\circ_{i_0} + c_{jl}^\circ/4]$ and
$\Delta_{jl}^{(2)}:=[\gamma_{i_0+1}^\circ - c_{jl}^\circ/4, \gamma^\circ_{p} + c_{jl}^\circ/4]$.
The distance between the  segments $\Delta^{(1)}_{jl}$ and $\Delta_{jl}^{(2)}$ is at least $c_{jl}^\circ/2$.
As was shown in \cite[Section~2]{Su4}, for
$|t| \le t^{00}_{jl}$ the operator $A(t)$
has exactly $k_1+\dots + k_{i_0}$ eigenvalues (counted with multiplicities)
in the segment $t^2 \Delta^{(1)}_{jl}$  and
exactly $k_{i_0+1}+\dots + k_{p}$ eigenvalues in the segment $t^2\Delta^{(2)}_{jl}$. Let $F_{jl}^{(r)}(t)$ be the spectral projection of $A(t)$
corresponding to the segment $t^2 \Delta^{(r)}_{jl}$, $r=1,2$.
Then
\begin{equation}
\label{abstr_F(t)_F(1)_F(2)}
F(t) = F^{(1)}_{jl} (t) + F^{(2)}_{jl} (t), \quad |t| \le t^{00}_{jl}.
\end{equation}

The following statement was proved in \cite[Proposition 2.1]{Su4}.

\begin{proposition}[\cite{Su4}]
\label{prop2.1}
Suppose that $t^{00}_{jl}$ is subject to~\emph{\eqref{abstr_c_circ_jl}}, \emph{\eqref{abstr_t00_jl}}.
Then for $|t|\le t^{00}_{jl}$ we have
    \begin{equation}
    \label{abstr_cluster_thresold}
    \| F^{(1)}_{jl} (t) - (P_1 + \cdots + P_{i_0}) \| \le C_{4,jl} |t|, \quad
    \| F^{(2)}_{jl} (t) - (P_{i_0+1} + \cdots + P_p) \| \le C_{4,jl} |t|.
    \end{equation}
The constant~$C_{4,jl}$ is given by $C_{4,jl} = \beta_4 \delta^{-1/2} \|X_1\|^5 (c^{\circ}_{jl})^{-2}$.
\end{proposition}

\section{Threshold approximations for the operators $\cos(\tau A(t)^{1/2})$ and $A(t)^{-1/2}\sin(\tau A(t)^{1/2})$}\label{sec2}

In Sections \ref{sec2} and \ref{abstr_aprox_thrm_section} we suppose that the assumptions of Subsection~{\ref{abstr_X_A_section}} and Condition~{\ref{nondeg}}
are satisfied.

\subsection{Approximation for $A(t)^{1/2} F(t)$}

The following result was proved in~\cite[Theorem~2.4]{BSu5}.

\begin{proposition}[\cite{BSu5}]
   We have
    \begin{equation}
    \label{abstr_A_sqrt_threshold_1}
    \| A(t)^{1/2} F(t) - (t^2 S)^{1/2} P \| \le C_5 t^2, \quad |t| \le t^0; \quad     C_5 = \beta_5 \delta^{-1/2} \bigl(\|X_1\|^2 +  c_*^{-1/2} \|X_1\|^3\bigr).
    \end{equation}
\end{proposition}

We need to find a more precise approximation for the operator $A(t)^{1/2} F(t)$.
For this, we use the following representation (see, e.~g.,~\cite[Chapter~3, Section~3, Subsection~4]{Kr})
\begin{equation}
\label{abstr_A_sqrt_repres}
A(t)^{1/2} F(t)  = \frac{1}{\pi} \int_{0}^{\infty} \zeta^{-1/2} (A(t) + \zeta I)^{-1} A(t) F(t) \, d\zeta, \quad 0 < |t| \le t^0.
\end{equation}
Also, we need the following approximation for the resolvent $(A(t) + \zeta I)^{-1}$ obtained in~\cite[(5.19)]{BSu2}:
\begin{equation}
\label{abstr_resolv_threshold}
(A(t) + \zeta I)^{-1} F(t) = \Xi(t,\zeta) + t (Z \Xi(t,\zeta) + \Xi(t,\zeta) Z^*) -
 t^3 \Xi(t,\zeta) N \Xi(t,\zeta) + \mathcal{J}(t,\zeta), \  |t| \le t^0, \; \zeta > 0.
\end{equation}
Here $\Xi(t,\zeta): = (t^2 SP + \zeta I)^{-1} P$, and the operator $\mathcal{J}(t,\zeta) $ satisfies the following estimate (see~\cite[Subsection~5.2]{BSu2})
\begin{align}
\label{abstr_J_est}
&\| \mathcal{J}(t,\zeta) \| \le C_6 t^4 (c_* t^2 + \zeta)^{-2} + C_7 t^2 (c_* t^2 + \zeta)^{-1},\quad |t|\le t^0,\ \zeta >0;
\\
\label{abstr_C8_C9}
&C_6 = \beta_6 \delta^{-1} \bigl(\|X_1\|^4 +  c_*^{-1} \|X_1\|^6 \bigr),\quad
C_7 = \beta_7 \delta^{-1} \bigl(\|X_1\|^2 +  c_*^{-1} \|X_1\|^4\bigr).
\end{align}

By~\eqref{abstr_A(t)_threshold_2} and~\eqref{abstr_resolv_threshold},
\begin{equation}
\label{abstr_sum}
(A(t) + \zeta I)^{-1} A(t) F(t) = t^2 \Xi(t,\zeta)  SP + t^3 Z \Xi(t,\zeta)   SP
+    t^3 \Xi(t,\zeta) K -  t^5 \Xi(t,\zeta) N \Xi(t,\zeta) SP + Y(t,\zeta),
\end{equation}
where
\begin{equation}
\label{2.6a}
\begin{aligned}
Y(t,\zeta)&:=
 \Xi(t,\zeta) \Psi(t) + t (Z \Xi(t,\zeta) + \Xi(t,\zeta) Z^*)(t^3 K + \Psi(t))
\\
& - t^3 \Xi(t,\zeta) N \Xi(t,\zeta) (t^3 K + \Psi(t)) + \mathcal{J} (t,\zeta) \left( t^2 SP + t^3 K + \Psi (t) \right).
\end{aligned}
\end{equation}
We have used that $Z^* P =0$.
Substituting \eqref{abstr_sum} in~\eqref{abstr_A_sqrt_repres} and recalling that $N = N_0 + N_*$, we obtain
\begin{align}
\label{abstr_1/2}
A(t)^{1/2} F(t) & = \sum_{j=1}^3 I_j(t) + I_0(t) + I_*(t) + \Phi(t), \quad |t| \le t^0,
\\
\label{I1}
I_1(t)& := \frac{1}{\pi} \int_{0}^{\infty} \zeta^{-1/2} t^2 \Xi(t,\zeta)  SP \, d\zeta,
\\
\nonumber
I_2(t) &:= \frac{1}{\pi} \int_{0}^{\infty} \zeta^{-1/2}  t^3  Z \Xi(t,\zeta)  SP\, d\zeta= tZ I_1(t),
\\
\nonumber
I_3(t) &:= \frac{1}{\pi} \int_{0}^{\infty} \zeta^{-1/2}  t^3 \Xi(t,\zeta) K \, d\zeta = t I_1(t) S^{-1}PK,
\\
\nonumber
I_0(t) &:= -\frac{1}{\pi} \int_{0}^{\infty} \zeta^{-1/2} t^5 \Xi (t,\zeta) N_0 \Xi (t,\zeta) \, SP \, d\zeta,
\\
\label{I*}
I_*(t) &:= -\frac{1}{\pi} \int_{0}^{\infty} \zeta^{-1/2}  t^5 \Xi (t,\zeta) N_* \Xi (t,\zeta) \, SP \, d\zeta,
\\
\label{Phi}
\Phi(t) &:= \frac{1}{\pi} \int_{0}^{\infty} \zeta^{-1/2}  Y(t,\zeta) \, d\zeta.
\end{align}
For $t=0$ we put $I_j(0):=0$, $j=1,2,3,$ and $I_0(0) = I_*(0) = \Phi(0):=0$.
Then \eqref{abstr_1/2} for $t=0$ is obvious.

Using representation of the form \eqref{abstr_A_sqrt_repres}
for the operator $(t^2SP)^{1/2}P=|t| S^{1/2}P$, we see that
\begin{equation}
\label{abstr_I1-3}
I_1(t)=  |t| S^{1/2} P,\quad I_2(t)= t |t| Z S^{1/2} P, \quad I_3(t) = t |t| S^{-1/2} P K.
\end{equation}
Let us calculate the operator $I_0(t)$  in the basis $\{\omega_l\}_{l=1}^n$.
Since $N_0 S = S N_0$ (see \eqref{abstr_S_eigenvectors} and \eqref{abstr_N_0_N_*}), then
\begin{equation*}
I_0 (t) \omega_l = -\frac{t^5}{\pi} \int_{0}^{\infty} \zeta^{-1/2} \frac{\gamma_l}{( \gamma_l t^2 + \zeta)^{2}} N_0 \omega_l \, d\zeta = - \frac{1}{2} t |t| \gamma_l^{-1/2} N_0 \omega_l,\quad l=1,\dots,n.
\end{equation*}
In the invariant terms, we have
\begin{equation}
\label{I0}
I_0(t) = - \frac{1}{2} t |t| N_0 S^{-1/2} P.
\end{equation}
Substituting \eqref{abstr_I1-3} and \eqref{I0} in \eqref{abstr_1/2}
and recalling that $K = Z S P + S P Z^* + N_0 + N_*$, we obtain
\begin{equation}
\label{repr}
A(t)^{1/2} F(t)  = |t| S^{1/2} P +  t |t| \bigl( Z S^{1/2} P +  S^{1/2} P Z^* \bigr)
+ \frac{1}{2} t |t| N_0 S^{-1/2} P
+ t |t| S^{-1/2} P N_* + I_*(t) + \Phi(t).
\end{equation}
We have taken into account that $PZ=0$.
Consider the operator
\begin{equation}
\label{II*}
{\mathcal I}_*(t) := I_*(t) - t |t| N_* S^{-1/2} P.
\end{equation}
Since $|t| S^{-1/2} P = S^{-1} I_1(t)$, it follows from \eqref{I1} and \eqref{I*} that
\begin{equation}
\label{III*}
{\mathcal I}_*(t)
= -\frac{t^3}{\pi} \int_{0}^{\infty} \zeta^{-1/2} \left( \Xi (t,\zeta) N_*  + N_* \Xi (t,\zeta) - \zeta \,\Xi (t,\zeta) N_* \Xi (t,\zeta)\right) \, d\zeta.
\end{equation}
From the last representation it is clear that the operator ${\mathcal I}_*(t)$
is selfadjoint. Is is easily seen that
${\mathcal I}_*(t)= t|t| {\mathcal I}_*(1)$.
Relations \eqref{abstr_S_nondegenerated} and \eqref{III*} imply the estimate
\begin{equation*}
\| \mathcal{I}_*(1) \| \le \frac{1}{\pi} \|N\|  \int_{0}^{\infty} \zeta^{-1/2} \left( 2(\zeta+ c_*)^{-1} + \zeta (\zeta+ c_*)^{-2}\right)\,  d \zeta.
\end{equation*}
Combining this with \eqref{abstr_K_N_estimates}, we obtain
\begin{equation}
\label{I*est}
\| \mathcal{I}_* (1)\| \le \frac{5}{2}(2\delta)^{-1/2} c_*^{-1/2} \|X_1\|^3.
\end{equation}

By~\eqref{abstr_t0_fixation}, \eqref{abstr_Z_R_S_est}, \eqref{abstr_C5}, \eqref{abstr_K_N_estimates}, \eqref{abstr_S_nondegenerated}, \eqref{abstr_J_est}, \eqref{abstr_C8_C9}, and \eqref{2.6a}, it is easy to check that the term
\eqref{Phi} satisfies
\begin{equation}
\label{abstr_A_sqrt_threshold_pre}
 \| \Phi (t) \| \le C_{8} |t|^3, \quad |t| \le t^0;
\quad C_{8} = \beta_{8} \delta^{-1} \bigl( \|X_1\|^4 c_*^{-1/2} + \|X_1\|^6 c_*^{-3/2} + \|X_1\|^8  c_*^{-5/2}\bigr).
\end{equation}

Finally, combining \eqref{repr}, \eqref{II*}, and \eqref{abstr_A_sqrt_threshold_pre}, we arrive at the following result.

\begin{proposition}
 We have
     \begin{align}
     \label{abstr_A_sqrt_threshold_2}
     &A(t)^{1/2} F(t) = |t| S^{1/2} P  + t |t| G + \Phi (t),\quad |t| \le t^0,
\\
     \label{abstr_G}
     &G := Z S^{1/2} P + S^{1/2}P Z^* + \frac{1}{2} N_0 S^{-1/2} P  + S^{-1/2} N_* P +   N_* S^{-1/2} P + \mathcal{I}_* (1).
     \end{align}
    The operator $\Phi(t)$ satisfies estimate \eqref{abstr_A_sqrt_threshold_pre}.
\end{proposition}

\subsection{Approximation of the operator $e^{-i \tau A(t)^{1/2}} P$}

Consider the operator
\begin{equation}
    \label{abstr_E}
    E (t, \tau) := e^{-i\tau A(t)^{1/2}} P - e^{-i\tau (t^2 S)^{1/2}P} P.
\end{equation}
We have
\begin{gather}
    \label{abstr_E_E1_E2}
    E (t, \tau) = E_1 (t, \tau) + E_2 (t, \tau), \\
    \notag
    E_1 (t, \tau) := e^{-i\tau A(t)^{1/2}} F(t)^{\perp} P - F(t)^{\perp} e^{-i\tau (t^2 S)^{1/2}P} P, \\
    \label{abstr_E2}
    E_2 (t, \tau) := e^{-i\tau A(t)^{1/2}} F(t) P - F(t) e^{-i\tau (t^2 S)^{1/2}P} P.
\end{gather}
Since  $F(t)^\perp P = (P - F(t))P$, then \eqref{abstr_F(t)_threshold_1}  implies that
\begin{equation}
    \label{abstr_E1_estimate}
    \| E_1 (t, \tau) \| \le 2 C_1 |t|, \quad |t| \le t^0.
\end{equation}

The operator \eqref{abstr_E2}  can be written as (cf. \cite[Section 2]{BSu5})
\begin{equation}
    \label{abstr_E2_Sigma}
    {E}_2(t, \tau) = -i \int_0^{\tau} e^{i (\tilde{\tau} -\tau) A(t)^{1/2}} F(t) \bigl(A(t)^{1/2}F(t)  -  (t^2 S)^{1/2}P \bigr)
e^{-i \tilde{\tau} (t^2S)^{1/2}P} P \, d \tilde{\tau}.
\end{equation}
By \eqref{abstr_A_sqrt_threshold_1} and \eqref{abstr_E2_Sigma}, $\| E_2 (t, \tau) \| \le C_5 |\tau| t^2$ for $|t| \le t^0$.
Together with~\eqref{abstr_E_E1_E2} and~\eqref{abstr_E1_estimate} this implies
\begin{equation}
    \label{estimate_E}
    \| E(t,\tau) \| \le 2 C_1 |t| + C_5 |\tau| t^2, \quad |t| \le t^0.
\end{equation}

Now, we show that estimate \eqref{estimate_E} can be improved under the additional assumptions.
From~\eqref{abstr_A_sqrt_threshold_2} and \eqref{abstr_E2_Sigma} it follows that
\begin{equation}
    \label{abstr_Sigma}
    E_2(t, \tau) = -i \int_0^{\tau} e^{i (\tilde{\tau}-\tau) A(t)^{1/2}}F(t) ( t |t| G  + \Phi (t)) e^{-i \tilde{\tau} (t^2 S)^{1/2}P} P \, d \tilde{\tau}.
\end{equation}
Taking~\eqref{abstr_G} into account and recalling that $Z^* P = 0$, we represent the operator~\eqref{abstr_Sigma} as
\begin{align}
    \label{abstr_Sigma_tildeSigma_hatSigma}
    E_2 (t, \tau) &= \widetilde{E}_2(t, \tau) + {E}_0(t, \tau) + E_*(t,\tau),\\
    \label{abstr_tildeSigma}
    \widetilde{E}_2(t, \tau)& := -i \int_0^{\tau} e^{i (\tilde{\tau} -\tau) A(t)^{1/2}}F(t) \bigl( t |t|  Z S^{1/2} P + \Phi (t) \bigr) e^{-i \tilde{\tau} (t^2 S)^{1/2}P} P \, d \tilde{\tau},
\\
    \label{2.28a}
    E_0(t, \tau) &:= -\frac{i}{2} t |t|  \int_0^{\tau} e^{i (\tilde{\tau}-\tau) A(t)^{1/2}}F(t) S^{-1/2} N_0 e^{-i \tilde{\tau} (t^2 S)^{1/2}P} P \, d \tilde{\tau}, \\
    \label{abstr_E*}
    E_*(t, \tau) &:= -i  t |t|  \int_0^{\tau} e^{i (\tilde{\tau}-\tau) A(t)^{1/2}}F(t)
G_*  e^{-i \tilde{\tau} (t^2 S)^{1/2}P} P \, d \tilde{\tau},
\\
    \label{G*}
G_* &:= S^{-1/2} N_*P  +   N_* S^{-1/2}P  +  \mathcal{I}_* (1).
\end{align}
Since $P Z = 0$, then $F(t)Z S^{1/2} P = (F(t) - P ) Z S^{1/2} P$. Hence, relations~\eqref{abstr_Z_R_S_est},~\eqref{abstr_F(t)_threshold_1}, and~\eqref{abstr_A_sqrt_threshold_pre} imply the following estimate for the term~\eqref{abstr_tildeSigma}:
\begin{equation}
    \label{abstr_tildeSigma_estimate}
    \| \widetilde{E}_2(t, \tau) \| \le C_{9} | \tau | |t|^3, \quad |t| \le t^0,
\end{equation}
where $C_{9} = C_1 (8 \delta)^{-1/2} \| X_1 \|^2 + C_{8}$.

Note that, if $N=0$, then also $N_0=N_*=0$, ${\mathcal I}_*(1)=0$, whence
 ${E}_0(t,\tau)=E_*(t,\tau)= 0$.
Then relations~\eqref{abstr_E_E1_E2}, \eqref{abstr_E1_estimate}, \eqref{abstr_Sigma_tildeSigma_hatSigma},
and \eqref{abstr_tildeSigma_estimate} imply that
\begin{equation}
\label{abstr_cos_enchanced_est_1_wo_eps}
\| E(t,\tau)  \| \le 2 C_1 |t| + C_{9} | \tau | |t|^3,\quad |t| \le t^0,\quad \text{if}\  N=0.
 \end{equation}

\subsection{Estimate for the term $E_*(t,\tau)$}

This subsection concerns the case where $n \ge 2$.
We will use the notation and the results of Subsection~\ref{abstr_cluster_section}.
Recall that $N=N_0 + N_*$. By Remark~\ref{abstr_N_remark},
\begin{equation}
    \label{PNP}
    P_j N_* P_j = 0,\quad j = 1, \ldots ,p;\qquad P_l N_0 P_j  =0\quad \text{for} \  l \ne j.
\end{equation}
Thus, the operators $N_0$ and $N_*$ admit the following invariant representations:
\begin{equation}
    \label{abstr_N_invar_repers}
    N_0 = \sum_{j=1}^{p} P_j N P_j, \quad N_* = \sum_{{1 \le j, l \le p: \; j \ne l}} P_j N P_l.
\end{equation}

In this subsection, we obtain the analog of estimate \eqref{abstr_cos_enchanced_est_1_wo_eps} under the weaker assumption that $N_0=0$.
For this, we have to estimate the operator~\eqref{abstr_E*}. However, we are able to do this only for a
smaller interval of $t$. By \eqref{abstr_P_Pj}, \eqref{III*}, \eqref{G*}, and \eqref{PNP},
it is seen that
$G_* = \sum_{{1 \le j, l \le p: \;  j \ne l}} P_j G_* P_l$.
Hence, the term \eqref{abstr_E*} can be represented as
\begin{gather}
    \label{abstr_E*_sumJjl}
    E_*(t, \tau) = -i e^{-i \tau A(t)^{1/2}} \sum_{{1 \le j, l \le p: \;  j \ne l}} J_{jl} (t, \tau),\\
    \label{abstr_Jjl}
    J_{jl} (t, \tau) := t |t| \int_0^{\tau} e^{i \tilde{\tau} A(t)^{1/2}}F(t) P_j G_* P_l e^{-i \tilde{\tau} (t^2 S)^{1/2}P} P \, d \tilde{\tau}.
\end{gather}

We have to estimate only those terms in \eqref{abstr_E*_sumJjl} for which $P_j N P_l \ne 0$.
So, let $j\ne l$, and let $P_j N P_l \ne 0$. Suppose that $c^\circ_{jl}$ is defined by~\eqref{abstr_c_circ_jl},
and $t^{00}_{jl}$ is subject to~\eqref{abstr_t00_jl}.
By~\eqref{abstr_F(t)_F(1)_F(2)},  the operator~\eqref{abstr_Jjl} can be represented as
\begin{gather}
    \label{abstr_J_jl_sumJ^r_jl}
    J_{jl}(t, \tau) = J^{(1)}_{jl}(t, \tau) + J^{(2)}_{jl}(t, \tau),  \quad |t|\le t^{00}_{jl},
\\
    \label{abstr_J^r_jl}
    \begin{split}
    J^{(r)}_{jl} (t, \tau) :=  t |t| \int_0^{\tau} e^{i \tilde{\tau} A(t)^{1/2}}F^{(r)}_{jl}(t)  P_j G_*  P_l  e^{-i \tilde{\tau} (t^2 S)^{1/2}P} P \, d \tilde{\tau},
\quad  r=1,2.
    \end{split}
\end{gather}

For definiteness, let $j<l$. Then $j < i_0+1$ and, by~\eqref{abstr_P_Pj} and~\eqref{abstr_cluster_thresold},
\begin{equation}
    \label{abstr_F^(2)_jl_P_j}
    \|F^{(2)}_{jl}(t) P_j \| = \| ( F^{(2)}_{jl}(t) - (P_{i_0+1} + \ldots + P_p) ) P_j \| \le C_{4,jl}|t|, \quad |t|\le t^{00}_{jl}.
\end{equation}
Here $C_{4,jl}$ is as in Proposition~\ref{prop2.1}.
Combining~\eqref{G*},~\eqref{abstr_J^r_jl},~\eqref{abstr_F^(2)_jl_P_j}, and taking~\eqref{abstr_K_N_estimates}, \eqref{abstr_S_nondegenerated},  and~\eqref{I*est} into account, we obtain
\begin{equation}
    \label{abstr_J^2_jl_estimate}
    \| J^{(2)}_{jl} (t, \tau)\| \le C_{4,jl} |t|^3 | \tau | (2 c_*^{-1/2} \|N\| + \|\mathcal{I}_* (1)\|) \le C_{10,jl}  | \tau | |t|^3, \quad |t| \le t^{00}_{jl},
\end{equation}
where $C_{10,jl} =\beta_{10} \delta^{-1} c_*^{-1/2} \|X_1\|^{8} (c^{\circ}_{jl})^{-2}$.

It remains to consider the term $J_{jl}^{(1)}(t,\tau)$.
Obviously, $P_l e^{-i \tilde{\tau} (t^2 S)^{1/2}P} P = e^{-i \tilde{\tau} |t| \sqrt{\gamma_l^{\circ}} } P_l$.
Next, we have
\begin{equation*}
    e^{i \tilde{\tau} A(t)^{1/2}}F^{(1)}_{jl}(t) = \sum_{r=1}^{i_0} \sum_{q=1}^{k_r}
e^{i \tilde{\tau} \sqrt{\lambda^{(r)}_q (t)}} (\cdot, \varphi^{(r)}_{q}(t))\varphi^{(r)}_{q}(t).
\end{equation*}
As a result, the operator  $J_{jl}^{(1)}(t,\tau)$ can be written as
\begin{equation}
    \label{abstr_J^1_jl}
    J^{(1)}_{jl} (t, \tau) = t |t| \sum_{r=1}^{i_0} \sum_{q=1}^{k_r} \left(  \int_0^{\tau} e^{i \tilde{\tau}  \bigl( \sqrt{ \mathstrut \lambda^{(r)}_q (t)}-\sqrt{\mathstrut \gamma^{\circ}_l} |t|\bigr) } d \tilde{\tau} \right) ( P_j G_* P_l \cdot, \varphi^{(r)}_{q}(t))\varphi^{(r)}_{q}(t).
\end{equation}
Calculating the integral and taking into account that
\hbox{$| \lambda^{(r)}_q (t)-\gamma^{\circ}_l t^2 | \ge \frac{3}{4} c^{\circ}_{jl} t^2$}
and
    $\left|  \sqrt{{\lambda^{(r)}_q (t)}} - \sqrt{{\gamma^{\circ}_l}} |t| \right| \ge (2\| X_1 \| |t|)^{-1}
| \lambda^{(r)}_q (t)- \gamma^{\circ}_l t^2 |$ for $|t| \le t^{00}_{jl}$, we obtain
\begin{equation}
    \label{abstr_J^1_jl_est_a}
    \left| \int_0^{\tau} e^{i \tilde{\tau}  \bigl( \sqrt{ \mathstrut \lambda^{(r)}_q (t)}-\sqrt{\mathstrut \gamma^{\circ}_l} |t|\bigr) }
 d \tilde{\tau} \right|
\le 2 \left|  \sqrt{{\lambda^{(r)}_q (t)}} - \sqrt{{\gamma^{\circ}_l}} |t| \right|^{-1}
\le 16 \| X_1 \| (3 c^{\circ}_{jl})^{-1}  |t|^{-1}.
\end{equation}
Now, relations~\eqref{abstr_J^1_jl} and~\eqref{abstr_J^1_jl_est_a} together with~\eqref{abstr_K_N_estimates},~\eqref{abstr_S_nondegenerated}, \eqref{I*est}, and \eqref{G*}
imply that
\begin{equation}
    \label{abstr_J^1_jl_estimate}
    \| J^{(1)}_{jl} (t, \tau) \| \le 16 \| X_1 \| (3 c^{\circ}_{jl})^{-1}  |t| (2 c_*^{-1/2}\|N\| + \|  \mathcal{I}_* (1)\|) \le C_{11,jl} |t|, \quad |t| \le  t^{00}_{jl},
\end{equation}
where
$C_{11,jl} = \beta_{11} \delta^{-1/2} c_*^{-1/2} \| X_1 \|^4 (c^{\circ}_{jl})^{-1}$.
The case where $j>l$ can be treated similarly.

Denote $\mathcal{Z} = \{ (j, l) \colon 1 \le j, l \le p, \; j \ne l, \; P_j N P_l \ne 0\}$. We put
\begin{equation}
    \label{abstr_c^circ}
    c^{\circ} := \min_{(j, l) \in \mathcal{Z}} c^{\circ}_{jl}
\end{equation}
and choose a number $t^{00} \le t^0$ such that
\begin{equation}
    \label{abstr_t00}
    t^{00} \le (4 \beta_2 )^{-1} \delta^{1/2} \|X_1\|^{-3} c^{\circ}.
\end{equation}
We may assume that  $t^{00} \le t^{00}_{jl}$ for all $(j,l)\in {\mathcal Z}$ (see~\eqref{abstr_t00_jl}).
Now, relations~\eqref{abstr_E*_sumJjl},~\eqref{abstr_J_jl_sumJ^r_jl},~\eqref{abstr_J^2_jl_estimate}, and~\eqref{abstr_J^1_jl_estimate},
together with expressions for the constants~$C_{10,jl}$ and~$C_{11,jl}$ yield
\begin{equation}
    \label{abstr_E*_estimate}
 \| E_*(t, \tau) \| \le C_{10} | \tau | |t|^3 + C_{11} |t|, \quad |t| \le t^{00},
\end{equation}
where       $C_{10} = \beta_{10} n^2    \delta^{-1}  c_*^{-1/2} \| X_1 \|^{8}  (c^{\circ})^{-2}$ and
       $C_{11} = \beta_{11} n^2   \delta^{-1/2} c_*^{-1/2} \| X_1 \|^4  ( c^{\circ})^{-1}$.

Note that, if $N_0=0$, then $E_0(t,\tau)=0$ and $E(t,\tau)= E_1(t,\tau)  + \widetilde{E}_2(t,\tau) + {E}_*(t,\tau)$
(see ~\eqref{abstr_E_E1_E2}, \eqref{abstr_Sigma_tildeSigma_hatSigma}, and \eqref{2.28a}).
Then estimates \eqref{abstr_E1_estimate}, \eqref{abstr_tildeSigma_estimate}, and \eqref{abstr_E*_estimate} imply that
\begin{equation}
\label{N0=0}
\| E(t,\tau)   \| \le C_{12} |t| + C_{13} | \tau | |t|^3,\quad |t|\le t^{00}, \quad \text{if}\ N_0=0;\quad C_{12} = 2C_1 + C_{11},\
C_{13} = C_{9} + C_{10}.
\end{equation}

\begin{remark}
Let $\mu_l$, $l=1,\dots,n,$  be the coefficients at $t^3$ in the expansions \emph{\eqref{abstr_A(t)_eigenvalues_series}}.
By Remark~\emph{\ref{abstr_N_remark}}, the condition $N_0=0$ is equivalent to the relations $\mu_l=0$ for all $l=1,\dots,n$.
\end{remark}

\subsection{Approximation for the operator $ A(t)^{-1/2} e^{-i \tau A(t)^{1/2}}$}

We need the following statement.

\begin{proposition}
\label{lemma_M}
   For $0< |t| \le t^0$ we have
    \begin{equation}
    \label{lemma_M_1}
    \| A(t)^{-1/2} P  - (t^2 S)^{-1/2} P \| \le {C}_{14};
 \quad
    {C}_{14} = {\beta}_{14} \delta^{-1/2} c_*^{-1/2} \| X_1\|  (1+  c_*^{-1} \|X_1\|^2 ).
    \end{equation}
\end{proposition}

\begin{proof}
In \cite[Lemma 1.1]{M} it was checked that
$\| A(t)^{-1/2} F(t)  - (t^2 S)^{-1/2} P \| \le \check{C}_{14}$ for $0< |t| \le t^0$,
    where $\check{C}_{14} = \check{\beta}_{14} \delta^{-1/2} c_*^{-1/2} \| X_1\|  (1+  c_*^{-1} \|X_1\|^2 )$.
Next, by \eqref{abstr_F(t)_threshold_1} and Condition~{\ref{nondeg}},
$\| A(t)^{-1/2} (F(t) - P)\| \le C_1 c_*^{-1/2}$.
This yields \eqref{lemma_M_1} with the constant ${C}_{14} = \check{C}_{14} + C_1 c_*^{-1/2}$.
\end{proof}

Consider the operator
\begin{equation}
    \label{sin_E}
    {\mathcal E} (t, \tau) := A(t)^{-1/2} e^{-i\tau A(t)^{1/2}} P - (t^2 S)^{-1/2} e^{-i\tau (t^2 S)^{1/2}P} P.
\end{equation}
By \eqref{abstr_E} and \eqref{sin_E}, we have
\begin{equation*}
    {\mathcal E} (t, \tau) = {\mathcal E}^\circ (t, \tau) + E(t,\tau)  (t^2 S)^{-1/2}P,\quad  {\mathcal E}^\circ (t, \tau):=
 e^{-i\tau A(t)^{1/2}} \bigl(A(t)^{-1/2} P -  (t^2 S)^{-1/2} P \bigr).
\end{equation*}
The operator ${\mathcal E}^\circ (t, \tau)$ is estimated by Proposition~\ref{lemma_M},
the operator $E(t,\tau)  (t^2 S)^{-1/2}P$ is estimated with the help of \eqref{abstr_S_nondegenerated} and \eqref{estimate_E}.
As a result, we obtain
    \begin{equation}
    \label{theorem_M_1}
    \| {\mathcal E}(t,\tau) \| \le {C}_{15} + C_{16} |\tau| |t|,\quad 0< |t|\le t^0;\quad
{C}_{15} = C_{14}+ 2 c_*^{-1/2} C_1,\ C_{16} = c_*^{-1/2} C_5.
    \end{equation}

Similarly, if $N=0$, using estimate \eqref{abstr_cos_enchanced_est_1_wo_eps}, we have
    \begin{equation}
    \label{abstr_sin_enchanced_est_1_wo_eps}
    \| {\mathcal E}(t,\tau) \| \le {C}_{15} + C_{17} |\tau| t^2,\quad 0< |t|\le t^0,\quad \text{if}\ N=0;\quad
C_{17} = c_*^{-1/2} C_{9}.
    \end{equation}

Finally, if $N_0=0$, applying \eqref{N0=0}, we arrive at
   \begin{equation}
    \label{N000}
    \| {\mathcal E}(t,\tau) \| \le {C}_{18} + C_{19} |\tau| t^2,\  0< |t|\le t^{00},\quad \text{if}\ N_0=0;\quad
C_{18} = {C}_{14} + c_*^{-1/2} C_{12}, \   C_{19} = c_*^{-1/2} C_{13}.
    \end{equation}

\subsection{Results: approximations for the operators $\cos(\tau A(t)^{1/2})P$ and $A(t)^{-1/2}\sin(\tau A(t)^{1/2})P$}
\label{abstract_results}

Now, we summarize the results.

\begin{theorem}
    \label{cos_thrm}
For $\tau \in \mathbb{R}$ and $|t| \le t^0$ we have
    \begin{align}
        \label{abstr_cos_general_thrm_wo_eps}
        &\| \cos (\tau A(t)^{1/2}) P - \cos (\tau (t^2 S)^{1/2}P) P \| \le 2C_1 |t| + C_5 |\tau| t^2,
\\
        \label{abstr_sin_general_thrm_wo_eps}
        &\| A(t)^{-1/2}\sin (\tau A(t)^{1/2}) P - (t^2 S)^{-1/2} \sin (\tau (t^2 S)^{1/2}P) P \|
 \le C_{15} + C_{16} |\tau| |t|.
    \end{align}
\end{theorem}

\begin{theorem}
\label{abstr_cos_enchanced_thrm_1_wo_eps}
Let $N=0$. Then for $\tau \in \mathbb{R}$ and $|t| \le t^{0}$ we have
\begin{align}
\label{NNN}
&\| \cos(\tau A(t)^{1/2}) P - \cos(\tau (t^2 S)^{1/2}P) P \| \le 2 C_{1} |t| + C_{9} | \tau | |t|^3,
\\
\label{NNNN}
&\| A(t)^{-1/2}\sin (\tau A(t)^{1/2}) P - (t^2 S)^{-1/2} \sin (\tau (t^2 S)^{1/2}P) P \|
 \le C_{15} + C_{17} |\tau| t^2.
 \end{align}
 \end{theorem}

\begin{theorem}
\label{abstr_cos_enchanced_thrm_2_wo_eps}
Let $N_0=0$. Then for $\tau \in \mathbb{R}$ and $|t| \le t^{00}$ we have
\begin{align*}
&\| \cos(\tau A(t)^{1/2}) P - \cos(\tau (t^2 S)^{1/2}P) P \| \le C_{12} |t| + C_{13} | \tau | |t|^3.
\\
&\| A(t)^{-1/2}\sin (\tau A(t)^{1/2}) P - (t^2 S)^{-1/2} \sin (\tau (t^2 S)^{1/2}P) P \|
 \le C_{18} + C_{19} |\tau| t^2.
 \end{align*}
Here the number $t^{00}\le t^0$ is subject to~\emph{\eqref{abstr_c^circ}},~\emph{\eqref{abstr_t00}}.
 \end{theorem}

Note that for $t=0$ the operator under the norm sign in \eqref{abstr_sin_general_thrm_wo_eps}
is understood as the limit operator, as $t\to 0$. Clearly, this limit is equal to zero.
Theorem~\ref{cos_thrm} follows from estimates \eqref{estimate_E} and \eqref{theorem_M_1};
Theorem~\ref{abstr_cos_enchanced_thrm_1_wo_eps} is deduced from \eqref{abstr_cos_enchanced_est_1_wo_eps} and \eqref{abstr_sin_enchanced_est_1_wo_eps};
Theorem~\ref{abstr_cos_enchanced_thrm_2_wo_eps} is a consequence of  \eqref{N0=0} and \eqref{N000}.
Theorem~\ref{cos_thrm} was known before (see~\cite[Theorem~2.5]{BSu5} and \cite[Proposition 2.1]{M}).

\section{Approximation of the operators $\cos(\varepsilon^{-1} \tau A(t)^{1/2})$ and $A(t)^{-1/2}\sin(\varepsilon^{-1} \tau A(t)^{1/2})$}
\label{abstr_aprox_thrm_section}

Let $\varepsilon > 0$. We study the behavior of the operators $\cos(\varepsilon^{-1} \tau A(t)^{1/2})$
and $A(t)^{-1/2}\sin(\varepsilon^{-1} \tau A(t)^{1/2})$
for small $\varepsilon$, multiplying them by the ``smoothing factor''
\hbox{$\varepsilon^s (t^2 + \varepsilon^2)^{-s/2}P$}, where $s > 0$.
(In applications to DOs  this factor turns into the smoothing operator.)
Our goal is to find approximation for the first operator with an error $O (\varepsilon)$ and for the second one with an error $O(1)$ for minimal possible $s$.

\subsection{The general case}

\begin{theorem}
    \label{abstr_cos_general_thrm}
For~$\varepsilon > 0$, $\tau \in \mathbb{R}$, and $|t| \le t^0$ we have
    \begin{align*}
       &\| \cos (\varepsilon^{-1} \tau A(t)^{1/2}) P - \cos (\varepsilon^{-1} \tau (t^2 S)^{1/2}P) P \| \varepsilon^2 (t^2 + \varepsilon^2)^{-1}
\le (C_1 + C_5 |\tau|) \varepsilon,
\\
                &\| A(t)^{-1/2}\sin (\varepsilon^{-1} \tau A(t)^{1/2}) P - (t^2 S)^{-1/2} \sin (\varepsilon^{-1} \tau (t^2 S)^{1/2}P) P \|
\varepsilon (t^2 + \varepsilon^2)^{-1/2} \le C_{15} + C_{16} |\tau|.
    \end{align*}
\end{theorem}

Theorem~\ref{abstr_cos_general_thrm} is deduced from estimates \eqref{abstr_cos_general_thrm_wo_eps}
and \eqref{abstr_sin_general_thrm_wo_eps} with $\tau$ replaced by $\varepsilon^{-1} \tau$.
This result was known before (see \cite[Theorem 2.7]{BSu5} and \cite[Theorem~2.3]{M}).

\subsection{ Refinement of approximations under the additional assumptions}

Now we obtain new results and show
that Theorem~\ref{abstr_cos_general_thrm} can be improved under the additional assumptions.

\begin{theorem}
    \label{abstr_cos_enchanced_thrm_1}
 Suppose that $N = 0$.
Then for~$\varepsilon > 0$, $\tau \in \mathbb{R}$, and~$|t| \le t^0$ we have
    \begin{align}
        \label{abstr_cos_enchanced_est_1}
&\| \cos(\varepsilon^{-1} \tau A(t)^{1/2}) P - \cos(\varepsilon^{-1} \tau (t^2 S)^{1/2}P) P \| \varepsilon^{3/2} (t^2 + \varepsilon^2)^{-3/4} \le ( C'_1 + C_{9} | \tau | ) \varepsilon,
\\
        \label{abstr_sin_enchanced_est_1}
&\| A(t)^{-1/2}\sin(\varepsilon^{-1} \tau A(t)^{1/2}) P - (t^2 S)^{-1/2} \sin(\varepsilon^{-1} \tau (t^2 S)^{1/2}P) P
\| \varepsilon^{1/2} (t^2 + \varepsilon^2)^{-1/4}   \le C_{15}' + C_{17} | \tau | .
      \end{align}
Here $C'_1 = \max\{2, 2 C_1\}$ and $C_{15}'= \max \{ 2 c_*^{-1/2}, C_{15}\}$.
\end{theorem}

\begin{proof} For $|t| \ge {\varepsilon}^{1/3}$ we have $\varepsilon^{3/2} (t^2 + \varepsilon^2)^{-3/4} \le \varepsilon$,
whence the left-hand side of~(\ref{abstr_cos_enchanced_est_1}) does not exceed $2 \varepsilon$.
 For~$|t| < {\varepsilon}^{1/3}$ we use~\eqref{NNN} with $\tau$ replaced by
$\varepsilon^{-1} \tau$ and estimate the left-hand side by
$\left( 2 C_1 |t| + C_{9} \varepsilon^{-1}  | \tau | |t|^3 \right) \varepsilon^{3/2} (t^2 + \varepsilon^2)^{-3/4}
   \le (2 C_1  + C_{9} |\tau|) \varepsilon$.
Estimate \eqref{abstr_cos_enchanced_est_1} follows.

Similarly, if $|t| \ge {\varepsilon}^{1/3}$, then $|t|^{-1} \varepsilon^{1/2} (t^2 + \varepsilon^2)^{-1/4} \le 1$. Hence, by Condition~\ref{nondeg} and \eqref{abstr_S_nondegenerated}, the left-hand side of~(\ref{abstr_sin_enchanced_est_1}) does not exceed $2 c_*^{-1/2}$. For~$|t| < {\varepsilon}^{1/3}$, we apply~\eqref{NNNN} with $\tau$ replaced by
$\varepsilon^{-1} \tau$ and note that
$\left(  C_{15} + C_{17} \varepsilon^{-1}  | \tau | t^2 \right) \varepsilon^{1/2} (t^2 + \varepsilon^2)^{-1/4}
    \le C_{15}  + C_{17} |\tau|$. This yields \eqref{abstr_sin_enchanced_est_1}.
\end{proof}

Similarly, Theorem~\ref{abstr_cos_enchanced_thrm_2_wo_eps} implies the following result.

\begin{theorem}
    \label{abstr_cos_enchanced_thrm_2}
 Suppose that $N_0 = 0$.
    Then for $\varepsilon > 0$, $\tau \in \mathbb{R}$, and $|t| \le t^{00}$ we have
    \begin{align*}
        &\| \cos(\varepsilon^{-1} \tau A(t)^{1/2}) P - \cos(\varepsilon^{-1} \tau (t^2 S)^{1/2}P) P \| \varepsilon^{3/2} (t^2 + \varepsilon^2)^{-3/4} \le (C'_{12} + C_{13} | \tau |) \varepsilon,
\\
        &\| A(t)^{-1/2}\sin(\varepsilon^{-1} \tau A(t)^{1/2}) P - (t^2 S)^{-1/2} \sin(\varepsilon^{-1} \tau (t^2 S)^{1/2}P) P
\| \varepsilon^{1/2} (t^2 + \varepsilon^2)^{-1/4}   \le C_{18}' + C_{19} | \tau |.
      \end{align*}
    Here $C'_{12} = \max \{2, C_{12}\}$ and $C_{18}' = \max \{ 2 c_*^{-1/2}, C_{18}\}$. The number $t^{00}\le t^0$ is subject to~\emph{\eqref{abstr_t00}}.
\end{theorem}

\begin{remark}
\label{rem_coeff}
We have traced the dependence of constants in estimates on the parameters.
The constants $C_1$, $C_5$, $C_{15}$, $C_{16}$ from Theorem~\emph{\ref{abstr_cos_general_thrm}}
and the constants $C_1'$, $C_{9}$, $C_{15}'$, $C_{17}$ from Theorem~\emph{\ref{abstr_cos_enchanced_thrm_1}}
are majorated by polynomials with {\rm (}absolute{\rm )} positive coefficients of the variables $\delta^{-1/2}$, $c_*^{-1/2}$, $\|X_1\|$.
The constants $C_{12}'$, $C_{13}$, $C_{18}'$, $C_{19}$ from Theorem~\emph{\ref{abstr_cos_enchanced_thrm_2}}
are majorated by polynomials with positive coefficients of the same variables and of $(c^\circ)^{-1}$ and $n$.
\end{remark}

\subsection{Sharpness of the result in the general case}

Now we show that the result of Theorem  \ref{abstr_cos_general_thrm}  is sharp in the general case.

\begin{theorem}
    \label{abstr_s<2_general_thrm}
 Let $N_0 \ne 0$.

\noindent $1^\circ$. Let $\tau \ne 0$ and $0 \le s < 2$. Then there does not exist a constant $C(\tau)>0$
    such that the estimate
        \begin{equation}
        \label{abstr_s<2_est_imp}
        \| \cos(\varepsilon^{-1} \tau A(t)^{1/2}) P - \cos(\varepsilon^{-1} \tau (t^2 S)^{1/2}P) P \| \varepsilon^{s} (t^2 + \varepsilon^2)^{-s/2} \le C(\tau) \varepsilon
    \end{equation}
    holds for all sufficiently small $|t|$ and $\varepsilon > 0$.

\noindent $2^\circ$. Let $\tau \ne 0$ and $0 \le s < 1$. Then there does not exist a constant $C(\tau)>0$
    such that the estimate
        \begin{equation}
        \label{sin_s<1_est_imp}
        \| A(t)^{-1/2}\sin(\varepsilon^{-1} \tau A(t)^{1/2}) P - (t^2 S)^{-1/2} \sin(\varepsilon^{-1} \tau (t^2 S)^{1/2}P) P \| \varepsilon^{s} (t^2 + \varepsilon^2)^{-s/2} \le C(\tau)
    \end{equation}
    holds for all sufficiently small $|t|$ and $\varepsilon > 0$.
\end{theorem}

\begin{proof}
First, we prove assertion $1^\circ$.
It suffices to consider $1 \le s < 2$.
We start with a preliminary remark.
Since $F(t)^{\perp}P = (P - F(t))P$, from~(\ref{abstr_F(t)_threshold_1}) it follows that
\begin{equation}
    \label{abstr_s<2_cos_Fperp_est}
    \| \cos(\varepsilon^{-1} \tau A(t)^{1/2}) F(t)^{\perp} P \| \varepsilon (t^2 + \varepsilon^2)^{-1/2}  \le C_1 \varepsilon,
    \quad   |t| \le t^0.
\end{equation}

We prove by contradiction. Suppose that for some $0 \ne \tau \in \mathbb{R}$ and $1\le s<2$ there exists a constant $C(\tau)>0$ such that~\eqref{abstr_s<2_est_imp}
 is valid for all sufficiently small~$|t|$ and~$\varepsilon$.
 By~\eqref{abstr_s<2_cos_Fperp_est}, this assumption is equivalent to
the existence of a constant $\widetilde{C}(\tau)$ such that
\begin{equation}
    \label{abstr_s<2_est_a}
    \left\| \bigl( \cos(\varepsilon^{-1} \tau A(t)^{1/2}) F(t) - \cos(\varepsilon^{-1} \tau (t^2 S)^{1/2}P) P\bigr)  P \right\| \varepsilon^{s} (t^2 + \varepsilon^2)^{-s/2} \le \widetilde{C} (\tau) \varepsilon
\end{equation}
for all sufficiently small~$|t|$ and~$\varepsilon$.

Recall that for $|t| \le t_*$ the power series expansions~\eqref{abstr_A(t)_eigenvalues_series} and~\eqref{abstr_A(t)_eigenvectors_series} are convergent, and $\gamma_l \ge c_* >0$.
Together with~the Taylor formula for
   $\sqrt{1+x}, \; |x| < 1,$ this implies that for some $0 < t_{**} \le t_*$ we have the following convergent power series expansions:
   \begin{equation}
    \label{abstr_sqrtA(t)_eigenvalues_series}
    \sqrt{\lambda_l(t)} = \sqrt{\gamma_l} |t| \left( 1 + {\mu_l}({2 \gamma_l})^{-1} t + \ldots\right) , \qquad l = 1, \ldots, n, \quad  |t| \le t_{**}.
\end{equation}

For $ |t| \le t^0$ we have
\begin{equation}
\label{s<2_a}
    \cos(\varepsilon^{-1} \tau A(t)^{1/2})F (t) = \sum_{l=1}^{n} \cos(\varepsilon^{-1} \tau \sqrt{\lambda_l(t)})(\cdot,\varphi_l(t))\varphi_l(t).
\end{equation}
From the convergence of the power series expansions~\eqref{abstr_A(t)_eigenvectors_series} it follows that
    $\| \varphi_l(t) - \omega_l \| \le c_1 |t|$ for $|t| \le t_*$, $l = 1,\ldots,n$.
Together with \eqref{abstr_s<2_est_a} and \eqref{s<2_a} this yields that there exists a constant
$\widehat{C} (\tau)$ such that
\begin{equation}
    \label{abstr_s<2_est_b}
    \biggl\| \sum_{l=1}^{n} \left( \cos(\varepsilon^{-1} \tau \sqrt{\mathstrut \lambda_l(t)}) - \cos(\varepsilon^{-1} \tau |t| \sqrt{\mathstrut \gamma_l}) \right) (\cdot,\omega_l)\omega_l  \biggr\| \varepsilon^{s} (t^2 + \varepsilon^2)^{-s/2} \le \widehat{C} (\tau) \varepsilon
\end{equation}
for all sufficiently small $|t|$ and $\varepsilon$.
The condition $N_0 \ne 0$ means that $\mu_j \ne 0$ at least for one $j$.
Applying the operator under the norm sign in~\eqref{abstr_s<2_est_b} to $\omega_j$, we obtain
\begin{equation}
    \label{abstr_s<2_est_c}
    \left| \cos \bigl(\varepsilon^{-1} \tau \sqrt{\mathstrut \smash{\lambda_j(t)}}\bigr)
    - \cos\bigl(\varepsilon^{-1} \tau |t| \sqrt{\mathstrut \smash{\gamma_j}}\bigr) \right| \varepsilon^{s} (t^2 + \varepsilon^2)^{-s/2} \le \widehat{C} (\tau) \varepsilon
\end{equation}
for all sufficiently small $|t|$ and $\varepsilon$.
Now, we put $t = t(\varepsilon) = (2 \pi)^{1/2}  {\gamma_j}^{1/4} |{\mu_j \tau}|^{-1/2} \varepsilon^{1/2} = c \varepsilon^{1/2}$.
Then $\cos\bigl(\varepsilon^{-1} \tau |t(\varepsilon)| \sqrt{\mathstrut \smash{\gamma_j}}\bigr) =\cos\bigl(\alpha_j \varepsilon^{-1/2}\bigr)$,
where $\alpha_j := (2 \pi)^{1/2} \gamma_j^{3/4} |\tau|^{1/2} |\mu_j|^{-1/2}$.
If $\varepsilon$ is so small that $t(\varepsilon) \le t_{**}$,
by~\eqref{abstr_sqrtA(t)_eigenvalues_series}, we have
$$
\cos \bigl(\varepsilon^{-1} \tau \sqrt{\mathstrut \smash{\lambda_j(t(\varepsilon))}}\bigr) =\cos \bigl(\alpha_j \varepsilon^{-1/2} +
\pi \, \sgn \,\mu_j + O(\varepsilon^{1/2}) \bigr) = - \cos \bigl( \alpha_j \varepsilon^{-1/2} + O(\varepsilon^{1/2})\bigr).
$$
Hence, by \eqref{abstr_s<2_est_c}, the function
$\left| \cos \bigl( \alpha_j \varepsilon^{-1/2} + O(\varepsilon^{1/2})\bigr) +
\cos\bigl(\alpha_j \varepsilon^{-1/2}\bigr) \right| \varepsilon^{s/2-1}(c^2 +  \varepsilon)^{-s/2}$
is uniformly bounded for small $\varepsilon >0$.
But this is not so, provided that $s<2$ (it suffices to consider the sequence $\varepsilon_k = (2 \pi k)^{-2}\alpha_j^2$, $k \in {\mathbb N}$).
This contradiction completes the proof of assertion $1^\circ$.

Now, we prove assertion $2^\circ$. As above, it is easy to show that, if  \eqref{sin_s<1_est_imp} holds for some $s<1$, then for some $j$
we have $\mu_j \ne 0$ and
\begin{equation}
    \label{sin_s<1_est_c}
    \left| \lambda_j(t)^{-1/2} \sin \bigl(\varepsilon^{-1} \tau \sqrt{\mathstrut \smash{\lambda_j(t)}}\bigr)
    - |t|^{-1} \gamma_j^{-1/2} \sin\bigl(\varepsilon^{-1} \tau |t| \sqrt{\mathstrut \smash{\gamma_j}}\bigr) \right| \varepsilon^{s} (t^2 + \varepsilon^2)^{-s/2} \le \widehat{C} (\tau)
\end{equation}
 for all sufficiently small $|t|$ and $\varepsilon$.
Taking $t=t(\varepsilon) =(2 \pi)^{1/2} {\gamma_j}^{1/4} |{\mu_j \tau}|^{-1/2} \varepsilon^{1/2} = c \varepsilon^{1/2}$ as above,
and using \eqref{abstr_sqrtA(t)_eigenvalues_series}, we deduce from \eqref{sin_s<1_est_c}
that the function
$\varepsilon^{(s-1)/2} |\sin(\alpha_j \varepsilon^{-1/2}) |$ is uniformly bounded for sufficiently small $\varepsilon >0$.
Choosing $\varepsilon_k = \alpha_j^2 (\pi/2 + 2\pi k)^{-2}$, $k \in {\mathbb N}$, we arrive at a contradiction.
\end{proof}

\section{The results for the operator family $A(t) = M^* \widehat{A} (t) M$}
\label{abstr_sandwiched_section}

\subsection{The operator family $A(t) = M^* \widehat{A} (t) M$}
\label{abstr_A_and_Ahat_section}

Let $\widehat{\mathfrak{H}}$ be yet another separable Hilbert space.
Suppose that the operator family $\widehat{X} (t) = \widehat{X}_0 + t \widehat{X}_1 \colon \widehat{\mathfrak{H}} \to \mathfrak{H}_* $
satisfies the assumptions of Subsection~\ref{abstr_X_A_section}.
Let $M \colon \mathfrak{H} \to \widehat{\mathfrak{H}}$~be an isomorphism. Suppose that
$M \Dom X_0 = \Dom \widehat{X}_0$ and $X(t) = \widehat{X} (t) M$.
In $\widehat{\mathfrak{H}}$ we consider the family of selfadjoint operators $\widehat{A} (t) = \widehat{X} (t)^* \widehat{X} (t)$.
Then, obviously,
\begin{equation}
\label{abstr_A_and_Ahat}
A(t) = M^* \widehat{A} (t) M.
\end{equation}
In what follows, all the objects corresponding to the family $\widehat{A}(t)$ are marked by~\textquotedblleft hat\textquotedblright.  Note that $\widehat{\mathfrak{N}} = M \mathfrak{N}$ and $\widehat{\mathfrak{N}}_* =  \mathfrak{N}_*$.
In $\widehat{\mathfrak{H}}$ we consider the positive definite operator
$Q := (M M^*)^{-1}$. Let
$Q_{\widehat{\mathfrak{N}}} = \widehat{P} Q|_{\widehat{\mathfrak{N}}}$~be the block of  $Q$ in the subspace~$\widehat{\mathfrak{N}}$.
Obviously, $Q_{\widehat{\mathfrak{N}}}$~is an isomorphism in~$\widehat{\mathfrak{N}}$.

According to \cite[Proposition~1.2]{Su2}, the orthogonal projection $P$ of~$\mathfrak{H}$ onto $\mathfrak{N}$ and the orthogonal
projection~$\widehat{P}$ of $\widehat{\mathfrak{H}}$ onto $\widehat{\mathfrak{N}}$ satisfy the following relation
\begin{equation}
\label{abstr_P_and_P_hat_relation}
P = M^{-1} (Q_{\widehat{\mathfrak{N}}})^{-1} \widehat{P} (M^*)^{-1}.
\end{equation}
Let $\widehat{S} \colon \widehat{\mathfrak{N}} \to \widehat{\mathfrak{N}}$~be the spectral germ of~$\widehat{A} (t)$ at $t = 0$,
and let~$S$~be the germ of $A (t)$. According to \cite[Chapter~1, Subsection~1.5]{BSu1}, we have
\begin{equation}
\label{abstr_S_and_S_hat_relation}
S = P M^* \widehat{S} M |_\mathfrak{N}.
\end{equation}

\subsection{The operators $\widehat{Z}_Q$ and $\widehat{N}_Q$}
\label{abstr_hatZ_Q_and_hatN_Q_section}

Let~$\widehat{Z}_Q$ be the operator in~$\widehat{\mathfrak{H}}$
taking~$\widehat{u} \in \widehat{\mathfrak{H}}$ to the solution $\widehat{\phi}_Q$ of the problem
$\widehat{X}^*_0 (\widehat{X}_0 \widehat{\phi}_Q + \widehat{X}_1 \widehat{\omega}) = 0$, $Q \widehat{\phi}_Q \perp \widehat{\mathfrak{N}}$,
where $\widehat{\omega} = \widehat{P} \widehat{u}$. As shown in~\cite[Section~6]{BSu2},
the operator $Z$ for $A(t)$ and the operator~$\widehat{Z}_Q$
satisfy~$\widehat{Z}_Q =M Z M^{-1} \widehat{P}$.  We put
$\widehat{N}_Q := \widehat{Z}_Q^* \widehat{X}_1^* \widehat{R} \widehat{P} + (\widehat{R} \widehat{P})^* \widehat{X}_1  \widehat{Z}_Q$.
According to~\cite[Section~6]{BSu2}, the operator~$N$ for $A(t)$ and the operator~$\widehat{N}_Q$ satisfy
\begin{equation}
\label{abstr_N_and_hatN_Q_relat}
\widehat{N}_Q = \widehat{P} (M^*)^{-1} N M^{-1} \widehat{P}.
\end{equation}
Since $N=N_0 + N_*$, we have $\widehat{N}_Q = \widehat{N}_{0,Q} + \widehat{N}_{*,Q}$, where
\begin{equation}
\label{abstr_N0*_and_hatN0*_Q_relat}
\widehat{N}_{0,Q} = \widehat{P} (M^*)^{-1} N_0 M^{-1} \widehat{P}, \qquad \widehat{N}_{*,Q} = \widehat{P} (M^*)^{-1} N_* M^{-1} \widehat{P}.
\end{equation}

We need the following lemma proved in~\cite[Lemma~5.1]{Su4}.

\begin{lemma}[\cite{Su4}]
    \label{abstr_N_and_Nhat_lemma}
    The relation~$N = 0$ is equivalent to the relation~$\widehat{N}_Q = 0$.
    The relation~$N_0 = 0$ is equivalent to the relation~$\widehat{N}_{0,Q} = 0$.
\end{lemma}

\subsection{Relations between the operators and the coefficients of the power series expansions}

(See~\cite[Subsections~1.6,~1.7]{BSu3}.) Denote $\zeta_l := M \omega_l \in \widehat{\mathfrak{N}}, \, l = 1, \ldots, n$.
Then relations~\eqref{abstr_S_eigenvectors},~\eqref{abstr_P_and_P_hat_relation}, and \eqref{abstr_S_and_S_hat_relation} show that
\begin{equation}
\label{abstr_hatS_gener_spec_problem}
\widehat{S} \zeta_l  = \gamma_l Q_{\widehat{\mathfrak{N}}} \zeta_l, \quad l = 1, \ldots, n.
\end{equation}
The set $\zeta_1, \ldots, \zeta_n$ forms a basis in~$\widehat{\mathfrak{N}}$ orthonormal with the weight~$Q_{\widehat{\mathfrak{N}}}$:
$(Q_{\widehat{\mathfrak{N}}} \zeta_l, \zeta_j) = \delta_{lj}$, $l,j = 1,\ldots,n$.

The operators $\widehat{N}_{0,Q}$ and $\widehat{N}_{*,Q}$ can be described in terms of the coefficients
of the expansions~\eqref{abstr_A(t)_eigenvalues_series} and~\eqref{abstr_A(t)_eigenvectors_series}; cf.~\eqref{abstr_N_0_N_*}.
We put~$\widetilde{\zeta}_l := M \widetilde{\omega}_l \in \widehat{\mathfrak{N}}, \; l = 1, \ldots,n $. Then
\begin{equation}
\label{abstr_hatN_0Q_N_*Q}
\widehat{N}_{0,Q} = \sum_{k=1}^{n} \mu_k (\cdot, Q_{\widehat{\mathfrak{N}}} \zeta_k) Q_{\widehat{\mathfrak{N}}} \zeta_k,
\quad \widehat{N}_{*,Q} = \sum_{k=1}^{n}
\gamma_k \left( (\cdot, Q_{\widehat{\mathfrak{N}}} \widetilde{\zeta}_k) Q_{\widehat{\mathfrak{N}}} \zeta_k + (\cdot, Q_{\widehat{\mathfrak{N}}} \zeta_k) Q_{\widehat{\mathfrak{N}}} \widetilde{\zeta}_k \right).
\end{equation}

Now we return to the notation of Section~\ref{abstr_cluster_section}.
Recall that the different eigenvalues of the germ $S$ are denoted by
$\gamma^{\circ}_j,\, j = 1,\ldots,p$, and the corresponding eigenspaces by~$\mathfrak{N}_j$.
The vectors~$\omega^{(j)}_i, \, i = 1,\ldots, k_j,$ form an orthonormal basis in~$\mathfrak{N}_j$.
Then the same numbers~$\gamma^{\circ}_j, \, j = 1,\ldots,p$,~are different eigenvalues of the problem~\eqref{abstr_hatS_gener_spec_problem}, and
$M \mathfrak{N}_j$~are the corresponding eigenspaces. The vectors~$\zeta^{(j)}_i = M\omega^{(j)}_i, \, i = 1,\ldots, k_j,$
 form a basis in~$M \mathfrak{N}_j$ (orthonormal with the weight~$Q_{\widehat{\mathfrak{N}}}$).
 By~$\mathcal{P}_j$ we denote the ``skew'' projection of $\widehat{\mathfrak H}$ onto~$M \mathfrak{N}_j$ that  is orthogonal with respect to the inner product~$(Q_{\widehat{\mathfrak{N}}} \cdot, \cdot)$, i.~e.,
$\mathcal{P}_j = \sum_{i=1}^{k_j} (\cdot, Q_{\widehat{\mathfrak{N}}} \zeta^{(j)}_i) \zeta^{(j)}_i$, $j = 1, \ldots, p$.
It is easily seen that~$\mathcal{P}_j =M P_j M^{-1} \widehat{P}$.
Using~\eqref{abstr_N_invar_repers}, \eqref{abstr_N_and_hatN_Q_relat}, and~\eqref{abstr_N0*_and_hatN0*_Q_relat},
it is easy to check that
\begin{equation}
\label{abstr_hatN_0Q_N_*Q_invar_repr}
\widehat{N}_{0,Q} = \sum_{j=1}^{p} \mathcal{P}_j^* \widehat{N}_Q \mathcal{P}_j, \quad \widehat{N}_{*,Q} = \sum_{{1 \le l,j \le p: \; l \ne j}} \mathcal{P}_l^* \widehat{N}_Q \mathcal{P}_j.
\end{equation}

\subsection{Approximations of the operators $\cos(\varepsilon^{-1} \tau A(t)^{1/2})$ and $A(t)^{-1/2}\sin(\varepsilon^{-1} \tau A(t)^{1/2})$}

In this subsection, we find approximations for the operators $\cos(\varepsilon^{-1}\tau A(t)^{1/2})$ and $A(t)^{-1/2}\sin(\varepsilon^{-1}\tau A(t)^{1/2})$, where
$A(t)$ is given by~\eqref{abstr_A_and_Ahat}, in terms of the germ~$\widehat{S}$ of $\widehat{A}(t)$ and the isomorphism~$M$.
It is convenient to border these operators by appropriate factors.
Denote $M_0 := (Q_{\widehat{\mathfrak{N}}})^{-1/2}$. We have
\begin{align}
\label{abstr_sandwiched_cos_S_relation}
& M \cos(\tau (t^2 S)^{1/2} P) P M^* = M_0 \cos (\tau (t^2 M_0 \widehat{S} M_0)^{1/2}) M_0 \widehat{P},
\\
\label{abstr_sandwiched_sin_S_relation}
& M (t^2 S)^{-1/2} \sin(\tau (t^2 S)^{1/2} P) P M^* = M_0 (t^2 M_0 \widehat{S} M_0)^{-1/2} \sin (\tau (t^2 M_0 \widehat{S} M_0)^{1/2}) M_0 \widehat{P}.
\end{align}
Relation~\eqref{abstr_sandwiched_cos_S_relation} was checked in~\cite[Proposition~3.3]{BSu5};
identity~\eqref{abstr_sandwiched_sin_S_relation} follows from \eqref{abstr_sandwiched_cos_S_relation}
by integration over the interval $(0,\tau)$. Denote
\begin{align}
\label{Jdef}
{J}_1(t,\tau)& := M \cos(\tau A(t)^{1/2})M^{-1} \widehat{P} - M_0 \cos (\tau (t^2 M_0 \widehat{S} M_0)^{1/2}) M_0^{-1} \widehat{P},
\\
\label{J2def}
\widetilde{J}_2(t,\tau)& := M A(t)^{-1/2}\sin(\tau A(t)^{1/2}) P M^* - M_0 (t^2 M_0 \widehat{S} M_0)^{-1/2}
\sin (\tau (t^2 M_0 \widehat{S} M_0)^{1/2}) M_0 \widehat{P}.
\end{align}

\begin{lemma}
    \label{abstr_cos_sandwiched_est_lemma}
   Under the assumptions of Subsection~\emph{\ref{abstr_A_and_Ahat_section}}, we have
   \begin{align}
   \label{abstr_cos_sandwiched_est_1}
   	&\| {J}_1(t,\tau)\|       \le \| M \| \| M^{-1} \| \| \cos (\tau A(t)^{1/2}) P -  \cos ( \tau (t^2 S)^{1/2} P) P \|,
   \\
   \label{abstr_sin_sandwiched_est_1}
   	&\| \widetilde{J}_2(t,\tau)\|       \le \| M \|^2 \| A(t)^{-1/2}\sin (\tau A(t)^{1/2}) P -  (t^2 S)^{-1/2} \sin ( \tau (t^2 S)^{1/2} P) P \|,
   \\
   \label{abstr_cos_sandwiched_est_2}
   &\| \cos (\tau  A(t)^{1/2}) P -  \cos ( \tau (t^2 S)^{1/2} P) P \|
   \le \| M \|^2 \| M^{-1} \|^2 \| {J}_1(t,\tau)\|,
\\
   \label{abstr_sin_sandwiched_est_2}
   &\| A(t)^{-1/2}\sin (\tau  A(t)^{1/2}) P -  (t^2 S)^{-1/2} \sin ( \tau (t^2 S)^{1/2} P) P \|
   \le \| M^{-1} \|^2 \| \widetilde{J}_2(t,\tau)\|.
   \end{align}
\end{lemma}

\begin{proof}
Since $M_0 = (Q_{\widehat{\mathfrak{N}}})^{-1/2}$ and
$M^{-1} Q_{\widehat{\mathfrak{N}}}^{-1} \widehat{P} = P M^*$ (see~\eqref{abstr_P_and_P_hat_relation}),
using~\eqref{abstr_sandwiched_cos_S_relation}, we obtain
\begin{equation}
\label{JJJ}
 {J}_1(t,\tau) =  M \bigl( \cos (\tau A(t)^{1/2}) P -
\cos ( \tau (t^2 S)^{1/2} P) P \bigr) M^{-1} \widehat{P}.
\end{equation}
This implies~(\ref{abstr_cos_sandwiched_est_1}).
Estimate~\eqref{abstr_cos_sandwiched_est_2} can be checked similarly in the ``inverse way''.
Obviously,
\begin{equation*}
\| \cos (\tau A(t)^{1/2}) P -  \cos ( \tau (t^2 S)^{1/2} P) P\| \le
\| M^{-1} \|^2 \|M\cos (\tau A(t)^{1/2}) PM^* -  M\cos ( \tau (t^2 S)^{1/2} P)P M^* \|.
\end{equation*}
By \eqref{JJJ}, the right-hand side can be written as
$\| M^{-1} \|^2 \| {J}_1(t,\tau) Q_{\widehat{\mathfrak{N}}}^{-1} \widehat{P} \|$.
 Together with the inequality \hbox{$\| Q_{\widehat{\mathfrak{N}}}^{-1} \widehat{P}\| \le \| M \|^2$}
(which follows from the identity $Q_{\widehat{\mathfrak{N}}}^{-1} \widehat{P} = M P M^*$),
this implies~\eqref{abstr_cos_sandwiched_est_2}.

Estimates \eqref{abstr_sin_sandwiched_est_1} and \eqref{abstr_sin_sandwiched_est_2} follow directly from~\eqref{abstr_sandwiched_sin_S_relation}
and \eqref{J2def}.
\end{proof}

Next, we put
\begin{equation}
\label{J22def}
{J}_2(t,\tau) := M A(t)^{-1/2}\sin(\tau A(t)^{1/2}) M^* \widehat{P} - M_0 (t^2 M_0 \widehat{S} M_0)^{-1/2}
\sin (\tau (t^2 M_0 \widehat{S} M_0)^{1/2}) M_0 \widehat{P}.
\end{equation}
By \eqref{abstr_P_and_P_hat_relation}, $P M^* = M^{-1} Q_{\widehat{\mathfrak{N}}}^{-1} \widehat{P}$. Hence, $P M^* = P M^* \widehat{P}$.
From \eqref{J2def} and \eqref{J22def} it is seen that
${J}_2(t,\tau) - \widetilde{J}_2(t,\tau) = M A(t)^{-1/2}\sin(\tau A(t)^{1/2}) (I - P) M^* \widehat{P}$.
Applying \eqref{abstr_F(t)_threshold_1} and Condition~\ref{nondeg}, we obtain
\begin{equation}
\label{J - Jtilde}
\|\widetilde{J}_2(t,\tau) - {J}_2(t,\tau)\|\le \| M\|^2 (C_1 c_*^{-1/2} + \delta^{-1/2}) =: \widetilde{C},\quad \tau \in {\mathbb R},\ |t|\le t^0.
\end{equation}

Using inequalities \eqref{abstr_cos_sandwiched_est_1}, \eqref{abstr_sin_sandwiched_est_1}, \eqref{J - Jtilde} and
taking Lemma~\ref{abstr_N_and_Nhat_lemma} into account,
we deduce the following results from Theorems \ref{abstr_cos_general_thrm},
\ref{abstr_cos_enchanced_thrm_1}, and \ref{abstr_cos_enchanced_thrm_2}.
In formulations, we use the notation \eqref{Jdef} and \eqref{J22def}.

\begin{theorem}
    \label{abstr_cos_sandwiched_general_thrm}
        Suppose that the assumptions of Subsection~\emph{\ref{abstr_A_and_Ahat_section}} and Condition~\emph{\ref{nondeg}}
    are satisfied. Then for $\tau \in \mathbb{R}$, $\varepsilon >0$, and $|t| \le t^0$ we have
    \begin{align*}
    \| {J}_1(t,\varepsilon^{-1} \tau) \| \varepsilon^{2} (t^2 + \varepsilon^2)^{-1}
    &\le \| M \| \| M^{-1} \| (C_1  + C_5 |\tau| ) \varepsilon,
\\
    \| {J}_2(t,\varepsilon^{-1} \tau) \| \varepsilon (t^2 + \varepsilon^2)^{-1/2}
    &\le \| M \|^2 (C_{15}  + C_{16} |\tau| ) + \widetilde{C}.
    \end{align*}
\end{theorem}

\begin{theorem}
    \label{abstr_cos_sandwiched_enchanced_thrm_1}
    Under the assumptions of Theorem~\emph{\ref{abstr_cos_sandwiched_general_thrm}}, suppose that $\widehat{N}_Q=0$. Then for
    $\tau \in \mathbb{R}$, $\varepsilon >0$, and $|t| \le t^0$ we have
    \begin{align}
    \label{4.19a}
    \| {J}_1(t,\varepsilon^{-1} \tau) \| \varepsilon^{3/2} (t^2 + \varepsilon^2)^{-3/4}
    &\le \| M \| \| M^{-1} \| (C'_1  + C_{9} |\tau| ) \varepsilon,
\\
\label{4.19b}
       \| {J}_2(t,\varepsilon^{-1} \tau) \| \varepsilon^{1/2} (t^2 + \varepsilon^2)^{-1/4}
    &\le \| M \|^2 (C'_{15}  + C_{17} |\tau| ) + \widetilde{C}.
    \end{align}
\end{theorem}

\begin{theorem}
    \label{abstr_cos_sandwiched_enchanced_thrm_2}
    Under the assumptions of Theorem~\emph{\ref{abstr_cos_sandwiched_general_thrm}}, suppose that $\widehat{N}_{0,Q}=0$. Then for
    $\tau \in \mathbb{R}$, $\varepsilon >0$, and $|t| \le t^{00}$ we have
    \begin{align*}
    \| {J}_1(t,\varepsilon^{-1} \tau) \| \varepsilon^{3/2} (t^2 + \varepsilon^2)^{-3/4}
    &\le \| M \| \| M^{-1} \| (C'_{12}  + C_{13} |\tau| ) \varepsilon,
\\
    \| {J}_2(t,\varepsilon^{-1} \tau) \| \varepsilon^{1/2} (t^2 + \varepsilon^2)^{-1/4}
    &\le \| M \|^2 \| (C'_{18}  + C_{19} |\tau| ) + \widetilde{C}.
    \end{align*}
\end{theorem}

Theorem \ref{abstr_cos_sandwiched_general_thrm} was known before (see \cite[Theorem~3.4]{BSu5} and \cite[Theorem~3.3]{M}).

\subsection{The sharpness of the result}

Now we confirm that the result of Theorem~\ref{abstr_cos_sandwiched_general_thrm} is sharp in the general case.

\begin{theorem}
    \label{abstr_sndwchd_s<2_general_thrm}
    Suppose that the assumptions of Theorem~\emph{\ref{abstr_cos_sandwiched_general_thrm}}
    are satisfied. Let $\widehat{N}_{0,Q} \ne 0$.

\noindent $1^\circ$. Let $\tau \ne 0$ and $0 \le s < 2$.
    Then there does not exist a constant $C(\tau)>0$ such that the estimate
    \begin{equation}
    \label{abstr_sndwchd_s<2_est_imp}
    \| {J}_1(t,\varepsilon^{-1} \tau) \| \varepsilon^{s} (t^2 + \varepsilon^2)^{-s/2} \le C(\tau) \varepsilon
    \end{equation}
    holds for all sufficiently small $|t|$ and $\varepsilon >0$.

\noindent $2^\circ$. Let $\tau \ne 0$ and $0 \le s < 1$.
    Then there does not exist a constant $C(\tau)>0$ such that the estimate
    \begin{equation}
\label{4.23}
    \| {J}_2(t,\varepsilon^{-1} \tau) \| \varepsilon^{s} (t^2 + \varepsilon^2)^{-s/2} \le C(\tau)
    \end{equation}
    holds for all sufficiently small $|t|$ and $\varepsilon >0$.
\end{theorem}

\begin{proof} By Lemma~\ref{abstr_N_and_Nhat_lemma}, under our assumptions we have $N_0 \ne 0$.
Let us check assertion $1^\circ$. We prove by contradiction.
Suppose that for some $\tau \ne 0$ and $0 \le s < 2$ there exists a constant $C (\tau) > 0$ such that~\eqref{abstr_sndwchd_s<2_est_imp}
 holds for all sufficiently small $|t|$ and $\varepsilon$.
 By~\eqref{abstr_cos_sandwiched_est_2}, this means that the inequality of the form~\eqref{abstr_s<2_est_imp} also holds
 (with some other constant).
But this contradicts the statement of Theorem~\ref{abstr_s<2_general_thrm}$(1^\circ)$.

Assertion $2^\circ$ is proved similarly with the help of \eqref{abstr_sin_sandwiched_est_2}, \eqref{J - Jtilde}, and Theorem~\ref{abstr_s<2_general_thrm}$(2^\circ)$.
\end{proof}

\section*{Chapter 2. Periodic differential operators in $L_2(\mathbb{R}^d; \mathbb{C}^n)$}

\section{The class of differential operators in $L_2(\mathbb{R}^d; \mathbb{C}^n)$}

\subsection{Preliminaries: Lattices in ${\mathbb R}^d$ and the Gelfand transformation}

Let $\Gamma$ be a lattice in $\mathbb{R}^d$ generated by
the basis $\mathbf{a}_1, \ldots , \mathbf{a}_d$, i.~e.,
$\Gamma = \left\{ \mathbf{a} \in \mathbb{R}^d \colon \mathbf{a} = \sum_{j=1}^{d} n_j \mathbf{a}_j, \; n_j \in \mathbb{Z} \right\}$,
and let $\Omega$ be the (elementary) cell of this lattice:
$\Omega := \left\{ \mathbf{x} \in \mathbb{R}^d \colon \mathbf{x} = \sum_{j=1}^{d} \xi_j \mathbf{a}_j, \; 0 < \xi_j < 1 \right\}$.
The basis $\mathbf{b}_1, \ldots , \mathbf{b}_d$ dual to $\mathbf{a}_1, \ldots , \mathbf{a}_d$ is defined by the relations
   $\langle \mathbf{b}_l, \mathbf{a}_j \rangle = 2 \pi \delta_{jl}$.
This basis generates the \textit{lattice  $\widetilde \Gamma$ dual to} $\Gamma$.
Let $\widetilde \Omega$   be the central \emph{Brillouin zone} of the lattice $\widetilde \Gamma $:
\begin{equation}
\label{Brillouin_zone}
\widetilde \Omega = \left\{ \mathbf{k} \in \mathbb{R}^d \colon | \mathbf{k} | < | \mathbf{k} - \mathbf{b} |, \; 0 \ne \mathbf{b} \in \widetilde \Gamma \right\}.
\end{equation}
Denote $| \Omega | = \mes \Omega$ and $| \widetilde \Omega | = \mes \widetilde \Omega$.
Note that $| \Omega |  | \widetilde \Omega | = (2 \pi)^d$. Let $r_0$ be the maximal radius of the ball containing in
$\clos \widetilde \Omega$. We have
$2 r_0 = \min_{0 \ne \mathbf{b} \in \widetilde \Gamma} |\mathbf{b}|$.

We need the discrete Fourier transformation  $ \{ \hat{\mathbf{v}}_{\mathbf{b}}\} \mapsto \mathbf{v}$:
$\mathbf{v}(\mathbf{x}) = | \Omega |^{-1/2} \sum_{\mathbf{b} \in \widetilde \Gamma} \hat{\mathbf{v}}_{\mathbf{b}} e^{i \left<\mathbf{b}, \mathbf{x} \right>}$,
which is a unitary mapping of $l_2 (\widetilde \Gamma; \mathbb{C}^n) $ onto $L_2 (\Omega; \mathbb{C}^n)$.
\textit{By $\widetilde H^1(\Omega; \mathbb{C}^n)$ we denote the subspace in $H^1(\Omega; \mathbb{C}^n)$
consisting of the functions whose $\Gamma$-periodic extension to
$\mathbb{R}^d$ belongs to $H^1_{\mathrm{loc}}(\mathbb{R}^d; \mathbb{C}^n)$}. We have
\begin{equation}
\label{D_and_fourier}
\int_{\Omega} |(\mathbf{D} + \mathbf{k}) \mathbf{u}|^2\, d\mathbf{x} = \sum_{\mathbf{b} \in \widetilde{\Gamma}} |\mathbf{b} + \mathbf{k} |^2 |\hat{\mathbf{u}}_{\mathbf{b}}|^2, \quad \mathbf{u} \in \widetilde{H}^1(\Omega; \mathbb{C}^n), \; \mathbf{k} \in \mathbb{R}^d,
\end{equation}
and convergence of the series in the right-hand side of~\eqref{D_and_fourier} is equivalent to the relation
$\mathbf{u} \in \widetilde{H}^1(\Omega; \mathbb{C}^n)$. From~\eqref{Brillouin_zone} and~\eqref{D_and_fourier} it follows that
\begin{equation}
\label{AO}
\int_{\Omega} |(\mathbf{D} + \mathbf{k}) \mathbf{u}|^2 d\mathbf{x} \ge \sum_{\mathbf{b} \in \widetilde{\Gamma}} | \mathbf{k} |^2 |\hat{\mathbf{u}}_{\mathbf{b}}|^2 = | \mathbf{k} |^2 \int_{\Omega} |\mathbf{u}|^2 d\mathbf{x}, \quad \mathbf{u} \in \widetilde{H}^1(\Omega; \mathbb{C}^n), \; \mathbf{k} \in \widetilde{\Omega}.
\end{equation}

Initially, the \textit{Gelfand transformation} $\mathcal{U}$ is defined on the functions of the Schwartz class $\mathbf{v} \in \mathcal{S} (\mathbb{R}^d; \mathbb{C}^n)$ by the formula
\begin{equation*}
\widetilde{\mathbf{v}} ( \mathbf{k}, \mathbf{x}) = (\mathcal{U} \- \mathbf{v}) (\mathbf{k}, \mathbf{x}) = | \widetilde \Omega |^{-1/2} \sum_{\mathbf{a} \in \Gamma} e^{- i \left< \mathbf{k}, \mathbf{x} + \mathbf{a} \right>} \mathbf{v} ( \mathbf{x} + \mathbf{a}),
 \quad \mathbf{x} \in \Omega, \quad \mathbf{k} \in \widetilde \Omega.
\end{equation*}
It extends by continuity up to a \textit{unitary mapping}
$\mathcal{U} \colon L_2 (\mathbb{R}^d; \mathbb{C}^n) \to \int_{\widetilde \Omega} \oplus  L_2 (\Omega; \mathbb{C}^n) d \mathbf{k} =: \mathcal{H}.$

\subsection{Factorized second order operators $\mathcal{A}$}
\label{A_oper_subsect}

Let $b(\mathbf{D}) = \sum_{l=1}^d b_l D_l$, where $b_l$ are constant  $(m \times n)$-matrices
(in general, with complex entries). \emph{Assume that $m \ge n$}.
Consider the symbol $b(\boldsymbol{\xi})= \sum^d_{l=1} b_l \xi_l$, $ \boldsymbol{\xi} \in \mathbb{R}^d$, and \emph{assume that}
\hbox{$\rank b( \boldsymbol{\xi} ) = n$} for $0 \ne  \boldsymbol{\xi} \in \mathbb{R}^d $.
This condition is equivalent to the inequalities
\begin{equation}
\label{rank_alpha_ineq}
\alpha_0 \mathbf{1}_n \le b( \boldsymbol{\theta} )^* b( \boldsymbol{\theta} ) \le \alpha_1 \mathbf{1}_n, \quad  \boldsymbol{\theta} \in \mathbb{S}^{d-1}, \quad 0 < \alpha_0 \le \alpha_1 < \infty,
\end{equation}
with some positive constants $\alpha_0, \alpha_1$.
Suppose that an $(m\times m)$-matrix-valued function $h(\mathbf{x})$ and an $(n\times n)$-matrix-valued function $f(\mathbf{x})$
(in general, with complex entries) are $\Gamma$-periodic and such that
\begin{equation}
\label{h_f_L_inf}
h, h^{-1} \in L_{\infty} (\mathbb{R}^d); \quad f, f^{-1} \in L_{\infty} (\mathbb{R}^d).
\end{equation}
Consider the closed operator
$\mathcal{X}  \colon  L_2 (\mathbb{R}^d ; \mathbb{C}^n) \to  L_2 (\mathbb{R}^d ; \mathbb{C}^m)$
given by $\mathcal{X} = h b( \mathbf{D} ) f$ on the domain $\Dom \mathcal{X} = \left\lbrace \mathbf{u}
\in L_2 (\mathbb{R}^d ; \mathbb{C}^n) \colon f \mathbf{u} \in H^1  (\mathbb{R}^d ; \mathbb{C}^n) \right\rbrace$.
In $L_2 (\mathbb{R}^d ; \mathbb{C}^n)$, consider the selfadjoint operator
$\mathcal{A} = \mathcal{X}^* \mathcal{X}$ generated by the closed quadratic form
$\mathfrak{a}[\mathbf{u}, \mathbf{u}] = \| \mathcal{X} \mathbf{u} \|^2_{L_2(\mathbb{R}^d)}, \; \mathbf{u} \in \Dom \mathcal{X}$.
 Formally, we have
\begin{equation}
\label{A}
\mathcal{A} = f (\mathbf{x})^* b( \mathbf{D} )^* g( \mathbf{x} )  b( \mathbf{D} ) f(\mathbf{x}),
\end{equation}
where $g(\mathbf{x}) := h(\mathbf{x} )^* h( \mathbf{x})$.
Note that the Hermitian matrix-valued function $g(\mathbf{x})$ is bounded and uniformly positive definite.
Using the Fourier transformation and \eqref{rank_alpha_ineq},~\eqref{h_f_L_inf}, it is easy to check that
\begin{equation*}
\alpha_0 \| g^{-1} \|_{L_{\infty}}^{-1} \| \mathbf{D} (f \mathbf{u}) \|_{L_2({\mathbb R}^d)}^2 \le \mathfrak{a}[\mathbf{u}, \mathbf{u}] \le \alpha_1 \| g \|_{L_{\infty}} \| \mathbf{D} (f \mathbf{u}) \|_{L_2({\mathbb R}^d)}^2, \  \mathbf{u} \in \Dom \mathcal{X}.
\end{equation*}

\subsection{The operators $\mathcal{A}(\mathbf{k})$}

Putting $\mathfrak{H} = L_2 (\Omega; \mathbb{C}^n)$ and $\mathfrak{H}_* = L_2 (\Omega; \mathbb{C}^m)$,
we consider the closed operator $\mathcal{X} (\mathbf{k}) \colon \mathfrak{H} \to \mathfrak{H}_*$ depending on the parameter $\mathbf{k} \in \mathbb{R}^d$ and given by
$\mathcal{X} (\mathbf{k}) = hb(\mathbf{D} + \mathbf{k})f$ on the domain $\Dom \mathcal{X} (\mathbf{k}) = \left\lbrace \mathbf{u} \in \mathfrak{H} \colon   f \mathbf{u} \in \widetilde{H}^1 (\Omega; \mathbb{C}^n)\right\rbrace =: \mathfrak{d}$.
The selfadjoint operator
$\mathcal{A} (\mathbf{k}) =\mathcal{X} (\mathbf{k})^* \mathcal{X} (\mathbf{k})$ in $\mathfrak{H}$
is generated by the closed quadratic form
$\mathfrak{a}(\mathbf{k})[\mathbf{u}, \mathbf{u}] = \| \mathcal{X}(\mathbf{k}) \mathbf{u} \|_{\mathfrak{H}_*}^2$, $\mathbf{u} \in \mathfrak{d}$.
Using the Fourier series expansion for $\mathbf{v}= f \mathbf{u}$ and conditions~\eqref{rank_alpha_ineq},~\eqref{h_f_L_inf},
it is easy to check that
\begin{equation}
\label{a(k)_form_est}
\alpha_0 \|g^{-1} \|_{L_\infty}^{-1} \|(\mathbf{D} + \mathbf{k}) f \mathbf{u} \|_{L_2 (\Omega)}^2 \le \mathfrak{a}(\mathbf{k})[\mathbf{u}, \mathbf{u}] \le \alpha_1 \|g \|_{L_\infty} \|(\mathbf{D} + \mathbf{k}) f \mathbf{u} \|_{L_2 (\Omega)}^2,
\quad
\mathbf{u} \in \mathfrak{d}.
\end{equation}
From \eqref{a(k)_form_est} and the compactness of the embedding of $\widetilde{H}^1(\Omega;\mathbb{C}^n)$
in $\mathfrak{H}$ it follows that the resolvent of $\mathcal{A}({\mathbf k})$ is compact and depends on ${\mathbf k}$ continuously.
By the lower estimate~\eqref{a(k)_form_est} and~\eqref{AO}, we have
\begin{align}
\label{A(k)_nondegenerated_and_c_*}	
\mathcal{A} (\mathbf{k}) &\ge c_* |\mathbf{k}|^2 I, \quad \mathbf{k} \in \widetilde{\Omega},
\\
\label{c*}
 c_* &= \alpha_0\|f^{-1} \|_{L_\infty}^{-2} \|g^{-1} \|_{L_\infty}^{-1} .
\end{align}

We put $\mathfrak{N} := \Ker \mathcal{A} (0) = \Ker \mathcal{X} (0)$.
Relations~\eqref{a(k)_form_est} with $\mathbf{k} = 0$ show that
\begin{equation}
\label{Ker2}
\mathfrak{N} = \left\lbrace \mathbf{u} \in L_2 (\Omega; \mathbb{C}^n) \colon f \mathbf{u} = \mathbf{c} \in \mathbb{C}^n \right\rbrace, \quad \dim \mathfrak{N} = n.
\end{equation}

\subsection{The direct integral expansion for the operator $\mathcal{A}$}
With the help of the Gelfand transformation, the operator $\mathcal A$ is represented as
\begin{equation}
\label{decompose}
\mathcal{U} \mathcal{A}  \mathcal{U}^{-1} = \int_{\widetilde \Omega} \oplus \mathcal{A} (\mathbf{k}) \, d \mathbf{k}.
\end{equation}
This means the following.
If $\mathbf{v} \in \Dom \mathcal{X}$, then
$\widetilde{\mathbf{v}}(\mathbf{k}, \cdot) \in \mathfrak{d}$ for a.~e. $\mathbf{k} \in \widetilde \Omega$, and
\begin{equation}
\label{AB}
\mathfrak{a}[\mathbf{v}, \mathbf{v}] = \int_{\widetilde{\Omega}} \mathfrak{a}(\mathbf{k}) [\widetilde{\mathbf{v}}(\mathbf{k}, \cdot), \widetilde{\mathbf{v}}(\mathbf{k}, \cdot)] \, d \mathbf{k} .
\end{equation}
Conversely, if $\widetilde{\mathbf{v}} \in \mathcal{H}$ satisfies~$\widetilde{\mathbf{v}}(\mathbf{k}, \cdot) \in \mathfrak{d}$ for a.~e. $\mathbf{k} \in \widetilde \Omega$ and the integral in~\eqref{AB} is finite, then
$\mathbf{v} \in \Dom \mathcal{X}$ and~\eqref{AB} is valid.

\subsection{Incorporation of the operators $\mathcal{A} (\mathbf{k})$ in the pattern of Chapter~1}
For $\mathbf{k} \in {\mathbb R}^d$ we put $\mathbf{k} = t \boldsymbol{\theta}$,
$t = | \mathbf{k}|$, $\boldsymbol{\theta} \in {\mathbb S}^{d-1}$.
We will apply the scheme of Chapter~1. Then all constructions and estimates will depend on the additional parameter $\boldsymbol{\theta}$.
We have to make our estimates uniform in $\boldsymbol{\theta}$.

We put $\mathfrak{H}:= L_2(\Omega;{\mathbb C}^n)$ and $\mathfrak{H}_* := L_2(\Omega;{\mathbb C}^m)$.
The role of $X(t)$ is played by $X(t, \boldsymbol{\theta}) := \mathcal{X}(t \boldsymbol{\theta})$.
Then $X(t, \boldsymbol{\theta}) = X_0 + t  X_1 (\boldsymbol{\theta})$, where
$X_0 = h(\mathbf{x}) b (\mathbf{D}) f(\mathbf{x}), \; \Dom X_0 = \mathfrak{d}$,
and $ X_1 (\boldsymbol{\theta}) = h b(\boldsymbol{\theta}) f$.
The role of $A(t)$ is played by $A(t, \boldsymbol{\theta}) := \mathcal{A}(t \boldsymbol{\theta})$. We have
$A(t, \boldsymbol{\theta}) = X(t, \boldsymbol{\theta})^* X(t, \boldsymbol{\theta})$.
The kernel $\mathfrak{N} = \Ker X_0$ is described by \eqref{Ker2}.
In \cite[Chapter~2,~Section~3]{BSu1}, it was shown that the distance $d^0$ from the point $\lambda_0=0$ to the rest of the spectrum of
${\mathcal A}(0)$ satisfies the estimate~$d^0 \ge 4 c_* r_0^2$. The condition $m \ge n$ ensures that $n \le n_* = \dim \Ker X^*_0$.

   In Subsection~\ref{abstr_X_A_section}, it was required to fix a number $\delta \in (0, d^0/8)$.
   Since $d^0 \ge 4 c_* r_0^2$, we choose
\begin{equation}
\label{delta_fixation}
\delta = \frac{1}{4} c_* r^2_0 = \frac{1}{4} \alpha_0\|f^{-1} \|_{L_\infty}^{-2} \|g^{-1} \|_{L_\infty}^{-1} r^2_0.
\end{equation}
Next, by~\eqref{rank_alpha_ineq} and~\eqref{h_f_L_inf}, we have
\begin{equation}
\label{X_1_estimate}
\| X_1 (\boldsymbol{\theta}) \| \le  \alpha^{1/2}_1 \| h \|_{L_{\infty}} \| f \|_{L_{\infty}}, \quad \boldsymbol{\theta} \in \mathbb{S}^{d-1}.
\end{equation}
This allows us to choose $t^0$ (see~\eqref{abstr_t0_fixation})
to be independent of $\boldsymbol{\theta}$. We put
\begin{equation}
\label{t0_fixation}
\begin{aligned}
t^0 = \frac{r_0}{2} \alpha_0^{1/2} \alpha_1^{-1/2} \left( \| h \|_{L_{\infty}} \| h^{-1} \|_{L_{\infty}} \|f \|_{L_\infty} \|f^{-1} \|_{L_\infty} \right)^{-1}.
\end{aligned}
\end{equation}
Obviously, $t^0 \le r_0/2$. Thus, the ball $|\mathbf{k}| \le t^0$ lies inside $\widetilde{\Omega}$.
 It is important that $c_*$, $\delta$, and $t^0$ (see~\eqref{c*}, \eqref{delta_fixation}, and
 \eqref{t0_fixation}) do not depend on $\boldsymbol{\theta}$.
By~\eqref{A(k)_nondegenerated_and_c_*},
Condition~\ref{nondeg} is now satisfied.
The germ $S(\boldsymbol{\theta})$ of the operator $A(t, \boldsymbol{\theta})$ is nondegenerate uniformly in
$\boldsymbol{\theta}$: we have $S(\boldsymbol{\theta}) \ge c_* I_{\mathfrak{N}}$ (cf.~\eqref{abstr_S_nondegenerated}).

\section{The effective characteristics of the operator $\widehat{\mathcal{A}}$}

\subsection{The case where $f = \mathbf{1}_n$}
In the case where $f = \mathbf{1}_n$, all the objects will be marked by \textquotedblleft hat\textquotedblright.
For instance, for the operator
\begin{equation}
\label{hatA}
\widehat{\mathcal{A}} = b(\mathbf{D})^* g(\mathbf{x}) b(\mathbf{D})
\end{equation}
the family $\widehat{\mathcal{A}} (\mathbf{k})$ is denoted by
$\widehat{A} (t, \boldsymbol{\theta})$. The kernel~(\ref{Ker2}) takes the form
\begin{equation}
\label{Ker3}
\widehat{\mathfrak{N}} = \left\lbrace \mathbf{u} \in L_2 (\Omega; \mathbb{C}^n) \colon \mathbf{u} = \mathbf{c} \in \mathbb{C}^n \right\rbrace,
\end{equation}
i.~e., $\widehat{\mathfrak{N}}$ consists of constant vector-valued functions.
The orthogonal projection $\widehat{P}$ of the space $L_2 (\Omega; \mathbb{C}^n)$ onto the subspace \eqref{Ker3}
is the operator of averaging over the cell:
\begin{equation}
    \label{Phat_projector}
	\widehat{P} \mathbf{u} = |\Omega|^{-1} \int_{\Omega} \mathbf{u} (\mathbf{x}) \, d\mathbf{x}.
\end{equation}

If $f = \mathbf{1}_n$, the constants \eqref{c*}, \eqref{delta_fixation}, and~\eqref{t0_fixation}
take the form
\begin{align}
\label{hatc_*}
&\widehat{c}_*  = \alpha_0 \|g^{-1} \|_{L_\infty}^{-1}, \\
\label{hatdelta_fixation}
&\widehat{\delta} =  \frac{1}{4} \alpha_0 \|g^{-1} \|_{L_\infty}^{-1} r^2_0,\\
\label{hatt0_fixation}
&\widehat{t}^{\,0}  = \frac{r_0}{2} \alpha_0^{1/2} \alpha_1^{-1/2} \left( \| h \|_{L_{\infty}} \| h^{-1} \|_{L_{\infty}} \right)^{-1}.
\end{align}

Inequality \eqref{X_1_estimate} takes the form
\begin{equation}
\label{hatX_1_estmate}
\| \widehat{X}_1 (\boldsymbol{\theta}) \| \le \alpha_1^{1/2} \| g \|_{L_{\infty}}^{1/2}.
\end{equation}

According to~\cite[Chapter~3,~Section~1]{BSu1},
the spectral germ $\widehat{S} (\boldsymbol{\theta}) : \widehat{\mathfrak{N}} \to \widehat{\mathfrak{N}}$ of the family
$\widehat{A}(t, \boldsymbol{\theta})$ is represented as
$\widehat{S} (\boldsymbol{\theta}) = b(\boldsymbol{\theta})^* g^0 b(\boldsymbol{\theta})$, $\boldsymbol{\theta} \in \mathbb{S}^{d-1}$,
where $g^0$ is the so-called \emph{effective matrix}. The constant
       ($m \times m$)-matrix $g^0$ is defined as follows. Suppose that a $\Gamma$-periodic $(n \times m)$-matrix-valued function
       $\Lambda \in \widetilde{H}^1 (\Omega)$ is the weak solution of the problem
       \begin{equation}
\label{equation_for_Lambda}
b(\mathbf{D})^* g(\mathbf{x}) \left(b(\mathbf{D}) \Lambda (\mathbf{x}) + \mathbf{1}_m \right) = 0, \quad \int_{\Omega} \Lambda (\mathbf{x}) \, d \mathbf{x} = 0.
\end{equation}
The effective matrix $g^0$ is given by
\begin{align}
\label{g0}
g^0 &= | \Omega |^{-1} \int_{\Omega} \widetilde{g} (\mathbf{x}) \, d \mathbf{x},
\\
\label{g_tilde}
\widetilde{g} (\mathbf{x}) &:= g(\mathbf{x})( b(\mathbf{D}) \Lambda (\mathbf{x}) + \mathbf{1}_m).
\end{align}
It turns out that the matrix $g^0$ is positive definite.
Consider the symbol
\begin{equation}
\label{effective_oper_symb}
\widehat{S} (\mathbf{k}) := t^2 \widehat{S} (\boldsymbol{\theta}) = b(\mathbf{k})^* g^0 b(\mathbf{k}), \quad \mathbf{k} \in \mathbb{R}^{d}.
\end{equation}
Expression~(\ref{effective_oper_symb}) is the symbol of the DO
\begin{equation}
\label{hatA0}
\widehat{\mathcal{A}}^0 = b(\mathbf{D})^* g^0 b(\mathbf{D})
\end{equation}
acting in $L_2(\mathbb{R}^d; \mathbb{C}^n)$ and called the \emph{effective operator} for the operator~$\widehat{\mathcal{A}}$.

Let $\widehat{\mathcal{A}}^0 (\mathbf{k})$ be the operator family in~$L_2(\Omega; \mathbb{C}^n)$ corresponding to the operator~\eqref{hatA0}.
Then $\widehat{\mathcal{A}}^0 (\mathbf{k})$ is given by the expression
$b(\mathbf{D} + \mathbf{k})^* g^0 b(\mathbf{D} + \mathbf{k})$ with periodic boundary conditions. By~\eqref{Phat_projector} and~\eqref{effective_oper_symb},
\begin{equation}
\label{hatS_P=hatA^0_P}
\widehat{S} (\mathbf{k}) \widehat{P} = \widehat{\mathcal{A}}^0 (\mathbf{k}) \widehat{P}.
\end{equation}

\subsection{Properties of the effective matrix}

The following properties of the matrix $g^0$ were checked in~\cite[Chapter~3, Theorem~1.5]{BSu1}.
\begin{proposition}
    The effective matrix satisfies the estimates
    \begin{equation}
    \label{Voigt_Reuss}
    \underline{g} \le g^0 \le \overline{g},
    \end{equation}
    where
    $ \overline{g} := | \Omega |^{-1} \int_{\Omega} g (\mathbf{x}) \, d \mathbf{x}$ and
    $\underline{g} := \left( | \Omega |^{-1} \int_{\Omega} g (\mathbf{x})^{-1} \, d \mathbf{x}\right)^{-1}$.
    If $m = n$, then $g^0 = \underline{g}$.
\end{proposition}

For specific DOs, estimates~\eqref{Voigt_Reuss} are known in homogenization theory as the Voigt-Reuss bracketing.
Now we distinguish the cases where one of the inequalities in~\eqref{Voigt_Reuss} becomes an identity.
    The following statements were obtained in~\cite[Chapter~3, Propositions~1.6, 1.7]{BSu1}.

\begin{proposition}
    The identity $g^0 = \overline{g}$ is equivalent to the relations
    \begin{equation}
    \label{g0=overline_g_relat}
    b(\mathbf{D})^* \mathbf{g}_k (\mathbf{x}) = 0, \quad k = 1, \ldots, m,
    \end{equation}
    where $\mathbf{g}_k (\mathbf{x}), \; k = 1, \ldots,m,$~are the columns of the matrix~$g (\mathbf{x})$.
\end{proposition}

\begin{proposition}
    The identity $g^0 = \underline{g}$ is equivalent to the representations
    \begin{equation}
    \label{g0=underline_g_relat}
    \mathbf{l}_k (\mathbf{x}) = \mathbf{l}^0_k + b(\mathbf{D}) \mathbf{w}_k(\mathbf{x}), \quad \mathbf{l}^0_k \in \mathbb{C}^m, \quad \mathbf{w}_k \in \widetilde{H}^1 (\Omega; \mathbb{C}^n), \quad k = 1, \ldots,m,
    \end{equation}
    where $\mathbf{l}_k (\mathbf{x}), \; k = 1, \ldots,m,$~are the columns of the matrix~$g (\mathbf{x})^{-1}$.
\end{proposition}

\subsection{The analytic branches of eigenvalues and eigenvectors}

The analytic (in $t$) branches of the eigenvalues~$\widehat{\lambda}_l (t, \boldsymbol{\theta})$ and the analytic branches of
the eigenvectors $\widehat{\varphi}_l (t, \boldsymbol{\theta})$ of $\widehat{A} (t, \boldsymbol{\theta})$ admit the power series expansions of the form~\eqref{abstr_A(t)_eigenvalues_series} and~\eqref{abstr_A(t)_eigenvectors_series}
with the coefficients depending on $\boldsymbol{\theta}$:
\begin{gather}
\label{hatA_eigenvalues_series}
\widehat{\lambda}_l (t, \boldsymbol{\theta}) = \widehat{\gamma}_l (\boldsymbol{\theta}) t^2 + \widehat{\mu}_l (\boldsymbol{\theta}) t^3 + \ldots, \quad l = 1, \ldots, n, \\
\label{hatA_eigenvectors_series}
\widehat{\varphi}_l (t, \boldsymbol{\theta}) = \widehat{\omega}_l (\boldsymbol{\theta}) + t \widehat{\psi}^{(1)}_l (\boldsymbol{\theta}) + \ldots, \quad l = 1, \ldots, n.
\end{gather}
(However,  we do not control the interval of convergence $t = |\mathbf{k}| \le t_* (\boldsymbol{\theta})$.)
According to~\eqref{abstr_S_eigenvectors}, $\widehat{\gamma}_l (\boldsymbol{\theta})$ and $\widehat{\omega}_l (\boldsymbol{\theta})$ are eigenvalues and eigenvectors of the germ:
$$
b(\boldsymbol{\theta})^* g^0 b(\boldsymbol{\theta}) \widehat{\omega}_l (\boldsymbol{\theta}) = \widehat{\gamma}_l (\boldsymbol{\theta}) \widehat{\omega}_l (\boldsymbol{\theta}), \quad l = 1, \ldots, n.
$$

\subsection{The operator $\widehat{N} (\boldsymbol{\theta})$}

We need to describe the operator~$N$ (which in abstract terms is defined in~Proposition~\ref{th1.3}).
According to~\cite[Section~4]{BSu3}, for the family $\widehat{A} (t, \boldsymbol{\theta})$ this operator takes the form
\begin{align}
\label{N(theta)}
\widehat{N} (\boldsymbol{\theta}) &= b(\boldsymbol{\theta})^* L(\boldsymbol{\theta}) b(\boldsymbol{\theta}) \widehat{P},
\\
\label{L(theta)}
L (\boldsymbol{\theta}) &:= | \Omega |^{-1} \int_{\Omega}  \bigl( \Lambda (\mathbf{x})^* b(\boldsymbol{\theta})^* \widetilde{g}(\mathbf{x}) + \widetilde{g}(\mathbf{x})^* b(\boldsymbol{\theta}) \Lambda (\mathbf{x})  \bigr) \, d \mathbf{x}.
\end{align}
Here $\Lambda (\mathbf{x})$ is the $\Gamma$-periodic solution of problem~\eqref{equation_for_Lambda}, and $\widetilde{g}(\mathbf{x})$~is
given by~\eqref{g_tilde}.

Some cases where~$\widehat{N}(\boldsymbol{\theta})=0$ were distinguished in~\cite[Section~4]{BSu3}.

\begin{proposition}
    \label{N=0_proposit}
    Suppose that at least one of the following conditions is fulfilled{\rm :}

 $1^\circ$.  $\widehat{\mathcal{A}} = \mathbf{D}^* g(\mathbf{x}) \mathbf{D}$, where $g(\mathbf{x})$~is a symmetric matrix with real entries.

 $2^\circ$. Relations~\emph{\eqref{g0=overline_g_relat}} are satisfied, i.~e., $g^0 = \overline{g}$.

 $3^\circ$. Relations~\emph{\eqref{g0=underline_g_relat}} are satisfied, i.~e., $g^0 = \underline{g}$.
\emph{(}If $m = n$, this is the case.\emph{)}

\noindent
    Then $\widehat{N} (\boldsymbol{\theta}) = 0$ for any $\boldsymbol{\theta} \in \mathbb{S}^{d-1}$.
\end{proposition}

Recall (see Remark~\ref{abstr_N_remark}) that
$\widehat{N} (\boldsymbol{\theta}) = \widehat{N}_0 (\boldsymbol{\theta}) + \widehat{N}_* (\boldsymbol{\theta})$, where the operator $\widehat{N}_0 (\boldsymbol{\theta})$ is diagonal in the basis $\{ \widehat{\omega}_l (\boldsymbol{\theta})\}_{l=1}^n$,
while  the diagonal elements of the operator $\widehat{N}_* (\boldsymbol{\theta})$
are equal to zero. We have
\begin{equation*}
(\widehat{N} (\boldsymbol{\theta}) \widehat{\omega}_l (\boldsymbol{\theta}), \widehat{\omega}_l (\boldsymbol{\theta}))_{L_2 (\Omega)} = (\widehat{N}_0 (\boldsymbol{\theta}) \widehat{\omega}_l (\boldsymbol{\theta}), \widehat{\omega}_l (\boldsymbol{\theta}))_{L_2 (\Omega)} = \widehat{\mu}_l (\boldsymbol{\theta}), \quad l=1, \ldots, n.
\end{equation*}

In~\cite[Subsection~4.3]{BSu3}, the following statement was proved.

\begin{proposition}
    \label{N_0=0_proposit}
    Suppose that $b(\boldsymbol{\theta})$ and $g (\mathbf{x})$~have real entries. Suppose that in the expansions~\emph{\eqref{hatA_eigenvectors_series}}
the    \textquotedblleft embryos\textquotedblright \ $\widehat{\omega}_l (\boldsymbol{\theta}), \; l = 1, \ldots, n,$ can be chosen to be real. Then in~\emph{\eqref{hatA_eigenvalues_series}} we have $\widehat{\mu}_l (\boldsymbol{\theta}) = 0$, $l=1, \ldots, n,$ i.~e.,
    $\widehat{N}_0 (\boldsymbol{\theta}) = 0$.
\end{proposition}

  In the \textquotedblleft real\textquotedblright \ case under consideration, the germ $\widehat{S} (\boldsymbol{\theta})$
  is a symmetric matrix with real entries. Clearly, if the eigenvalue $\widehat{\gamma}_j (\boldsymbol{\theta})$ of the germ is simple,
then the embryo $\widehat{\omega}_j (\boldsymbol{\theta})$ is defined uniquely up to a phase factor, and we can always choose
$\widehat{\omega}_j (\boldsymbol{\theta})$ to be real.  We arrive at the following corollary.

\begin{corollary}
    \label{S_spec_simple_coroll}
    Suppose that $b(\boldsymbol{\theta})$ and $g (\mathbf{x})$
    have real entries. Suppose that the spectrum of the germ~$\widehat{S} (\boldsymbol{\theta})$ is simple. Then
    $\widehat{N}_0 (\boldsymbol{\theta}) = 0$.
\end{corollary}

\subsection{Multiplicities of the eigenvalues of the germ}
\label{eigenval_multipl_section}
Considerations of this subsection concern the case where $n\ge 2$. Now we return to the notation of Section~\ref{abstr_cluster_section}.
In general, the number $p(\boldsymbol{\theta})$ of different eigenvalues $\widehat{\gamma}^{\circ}_1 (\boldsymbol{\theta}), \ldots, \widehat{\gamma}^{\circ}_{p(\boldsymbol{\theta})} (\boldsymbol{\theta})$ of the spectral germ $\widehat{S}(\boldsymbol{\theta})$
and their multiplicities depend on the parameter $\boldsymbol{\theta} \in \mathbb{S}^{d-1}$. For a fixed $\boldsymbol{\theta}$ denote by
$\widehat{P}_j (\boldsymbol{\theta})$ the orthogonal projection of $L_2 (\Omega; \mathbb{C}^n)$ onto the eigenspace of the germ
$\widehat{S}(\boldsymbol{\theta})$ corresponding to the eigenvalue $\widehat{\gamma}_j^{\circ} (\boldsymbol{\theta})$. According to~\eqref{abstr_N_invar_repers}, the operators $\widehat{N}_0 (\boldsymbol{\theta})$ and $\widehat{N}_* (\boldsymbol{\theta})$
admit the following invariant representations:
\begin{equation}
\label{N0_invar_repr}
\widehat{N}_0 (\boldsymbol{\theta}) = \sum_{j=1}^{p(\boldsymbol{\theta})} \widehat{P}_j (\boldsymbol{\theta}) \widehat{N} (\boldsymbol{\theta}) \widehat{P}_j (\boldsymbol{\theta}), \\
\quad
\widehat{N}_* (\boldsymbol{\theta}) = \sum_{{1 \le j, l \le p(\boldsymbol{\theta}):\; j \ne l}} \widehat{P}_j (\boldsymbol{\theta}) \widehat{N} (\boldsymbol{\theta}) \widehat{P}_l (\boldsymbol{\theta}).
\end{equation}

In conclusion of this section, we give Example~8.7 from~\cite{Su4}.

\begin{example}[\cite{Su4}]
    \label{elast_exmpl_N0_ne_0}
Let $d=2$, $n=2$, and $m=3$. Let $\Gamma = (2\pi \mathbb{Z})^2$. Suppose that
    \begin{gather*}
    b(\mathbf{D})= \begin{pmatrix}
    D_1 & 0 \\
    \frac{1}{2}D_2 & \frac{1}{2}D_1 \\
    0 & D_2
    \end{pmatrix}, \quad
    g(\mathbf{x})= \begin{pmatrix}
    1 & 0 & 0 \\
    0 & g_2 (x_1) & 0\\
    0 & 0 & g_3 (x_1)
    \end{pmatrix},
    \end{gather*}
where $g_2(x_1)$ and $g_3(x_1)$~are $(2\pi)$-periodic bounded and positive definite functions of $x_1$, and $\underline{g_2}=4$, $\overline{g_3}=1$.
Then the effective matrix is given by~$g^0 = \diag \{1, 4, 1\}$.  The spectral germ takes the form
$\widehat{S} (\boldsymbol{\theta}) = \begin{pmatrix}
1 & \theta_1 \theta_2  \\ \theta_1 \theta_2 & 1 \end{pmatrix}$, $\boldsymbol{\theta} \in \mathbb{S}^1$.
The eigenvalues of the germ $\widehat{\gamma}_1(\boldsymbol{\theta}) = 1 + \theta_1 \theta_2$ and $\widehat{\gamma}_2(\boldsymbol{\theta}) = 1 - \theta_1 \theta_2$
coincide at four points~$\boldsymbol{\theta}^{(1)} = (0, 1)$, $\boldsymbol{\theta}^{(2)} = (0, -1)$, $\boldsymbol{\theta}^{(3)} = (1, 0)$,  $\boldsymbol{\theta}^{(4)} = (-1, 0)$.
The operator $\widehat{N}(\boldsymbol{\theta})$ is given by
\begin{equation*}
\widehat{N}(\boldsymbol{\theta})  = \frac{1}{2} \theta_2^3 \begin{pmatrix}
0 & \overline{\Lambda_{22}^* g_3} \\
\overline{\Lambda_{22} g_3} & 0
\end{pmatrix},
\end{equation*}
where $\Lambda_{22}(x_1)$~is the $(2\pi)$-periodic solution of the problem
$\frac{1}{2} D_1 \Lambda_{22}(x_1) + 1 = 4 (g_2 (x_1))^{-1}$,  $\overline{\Lambda_{22}} =0$.
Assume that~$\overline{\Lambda_{22} g_3} \ne 0$.
{\rm (}For instance, if $g_2(x_1) = 4\left( 1 + \tfrac{1}{2} \sin x_1\right)^{-1}$ and $g_3(x_1) = 1 + \tfrac{1}{2} \cos x_1$,
all the conditions are fulfilled.{\rm )}
For $\boldsymbol{\theta} \ne \boldsymbol{\theta}^{(j)}$, $j = 1, 2, 3, 4$, we have
$\widehat{\gamma}_1 (\boldsymbol{\theta}) \ne \widehat{\gamma}_2 (\boldsymbol{\theta})$ and then  $\widehat{N}(\boldsymbol{\theta}) = \widehat{N}_*(\boldsymbol{\theta}) \ne 0$. At the points $\boldsymbol{\theta}^{(1)}$ and $\boldsymbol{\theta}^{(2)}$ we have
$\widehat{\gamma}_1 (\boldsymbol{\theta}^{(j)}) = \widehat{\gamma}_2 (\boldsymbol{\theta}^{(j)}) = 1$ and
$\widehat{N}(\boldsymbol{\theta}^{(j)}) = \widehat{N}_0(\boldsymbol{\theta}^{(j)}) \ne 0$, $j = 1, 2$.
The numbers $\pm \mu$, where $\mu = \frac{1}{2} \left| \overline{\Lambda_{22} g_3} \right| $,
 are the eigenvalues of the operator~$\widehat{N}(\boldsymbol{\theta}^{(j)})$ for~$j = 1, 2$.
Then $\widehat{\lambda}_1 (t, \boldsymbol{\theta}^{(j)}) = t^2 + \mu t^3 + \ldots$ and
$\widehat{\lambda}_2 (t, \boldsymbol{\theta}^{(j)}) = t^2 - \mu t^3 + \ldots$ for $j = 1,2$.
In this case, the embryos~$\widehat{\omega}_1 (\boldsymbol{\theta}^{(j)})$, $\widehat{\omega}_2 (\boldsymbol{\theta}^{(j)})$
 in the expansions~\eqref{hatA_eigenvectors_series}  cannot be real {\rm (}see Proposition~\emph{\ref{N_0=0_proposit})}.
 At the points $\boldsymbol{\theta}^{(3)}$ and $\boldsymbol{\theta}^{(4)}$  we have
$\widehat{\gamma}_1 (\boldsymbol{\theta}^{(j)}) = \widehat{\gamma}_2 (\boldsymbol{\theta}^{(j)}) = 1$ and
$\widehat{N}(\boldsymbol{\theta}^{(j)}) = 0$, $j = 3, 4$.
\end{example}

\section{Approximations of the operators $\cos(\varepsilon^{-1} \tau \widehat{\mathcal{A}}(\mathbf{k})^{1/2})$ and $\widehat{\mathcal{A}}(\mathbf{k})^{-1/2}\sin(\varepsilon^{-1} \tau \widehat{\mathcal{A}}(\mathbf{k})^{1/2})$}

\subsection{The general case}

Consider the operator $\mathcal{H}_0 = -\Delta$ in $L_2 (\mathbb{R}^d; \mathbb{C}^n)$.
Under the Gelfand transformation, this operator expands in the direct integral of the operators
$\mathcal{H}_0 (\mathbf{k})$ acting in $L_2 (\Omega; \mathbb{C}^n)$ and given by
the differential expression $| \mathbf{D} + \mathbf{k} |^2$  with periodic boundary conditions.
Denote
\begin{equation}
\label{R(k, epsilon)}
\mathcal{R}(\mathbf{k}, \varepsilon) := \varepsilon^2 (\mathcal{H}_0 (\mathbf{k}) + \varepsilon^2 I)^{-1}.
\end{equation}
Obviously,
\begin{equation}
\label{R_P}
\mathcal{R}(\mathbf{k}, \varepsilon)^{s/2}\widehat{P} = \varepsilon^s (t^2 + \varepsilon^2)^{-s/2} \widehat{P}, \qquad s > 0.
\end{equation}

Note that for $ | \mathbf{k} | > \widehat{t}^{\,0}$ we have
\begin{equation}
\label{R_hatP_est}
\| \mathcal{R}(\mathbf{k}, \varepsilon)^{s/2}\widehat{P} \|_{L_2 (\Omega) \to L_2 (\Omega)} \le (\widehat{t}^{\,0})^{-s} \varepsilon^s, \quad \varepsilon > 0, \; \mathbf{k} \in \widetilde{\Omega}, \; | \mathbf{k} | > \widehat{t}^{\,0}.
\end{equation}
Next, using the discrete Fourier transformation, we see that
\begin{equation}
\label{R(k,eps)(I-P)_est}
\| \mathcal{R}(\mathbf{k}, \varepsilon)^{s/2} (I - \widehat{P}) \|_{L_2(\Omega) \to L_2 (\Omega) }  = \sup_{0 \ne \mathbf{b} \in \widetilde{\Gamma}} \varepsilon^s (|\mathbf{b} + \mathbf{k}|^2 + \varepsilon^2)^{-s/2} \le r_0^{-s} \varepsilon^s,\quad \varepsilon > 0,\ \mathbf{k} \in \widetilde{\Omega}.
\end{equation}

Denote
\begin{align}
\label{J(k,tau)}
\widehat{J}_1(\mathbf{k}, \tau) &:=
\cos ( \tau \widehat{\mathcal{A}}(\mathbf{k})^{1/2})  - \cos (\tau \widehat{\mathcal{A}}^0(\mathbf{k})^{1/2}),
\\
\label{J2(k,tau)}
\widehat{J}_2(\mathbf{k}, \tau) &:=
\widehat{\mathcal{A}}(\mathbf{k})^{-1/2} \sin ( \tau \widehat{\mathcal{A}}(\mathbf{k})^{1/2})  - \widehat{\mathcal{A}}^0(\mathbf{k})^{-1/2} \sin (\tau \widehat{\mathcal{A}}^0(\mathbf{k})^{1/2}).
\end{align}

We will apply theorems of Section~\ref{abstr_aprox_thrm_section} to the operator
 $\widehat{A}(t, \boldsymbol{\theta}) = \widehat{\mathcal{A}}(\mathbf{k})$.
 Due to Remark~\ref{rem_coeff}, we may trace the dependence of the constants in estimates on the problem data. Note that $\widehat{c}_*$, $\widehat{\delta}$, $\widehat{t}^{\,0}$
 do not depend on $\boldsymbol{\theta}$ (see \eqref{hatc_*}--\eqref{hatt0_fixation}).
 By \eqref{hatX_1_estmate}, the norm $\| \widehat{X}_1 (\boldsymbol{\theta}) \|$ can be replaced by $\alpha_1^{1/2} \| g \|_{L_{\infty}}^{1/2}$.
 Hence, the constants in Theorems~\ref{abstr_cos_general_thrm} and~\ref{abstr_cos_enchanced_thrm_1} (as applied to the operator $\widehat{\mathcal{A}}(\mathbf{k})$)
 will be independent of $\boldsymbol{\theta}$. They depend only on $\alpha_0$, $\alpha_1$, $\|g\|_{L_\infty}$, $\|g^{-1}\|_{L_\infty}$, and $r_0$.

\begin{theorem}
    \label{cos_general_thrm}
For $\tau \in \mathbb{R}$, $\varepsilon > 0$, and $\mathbf{k} \in \widetilde{\Omega}$ we have
 \begin{align*}
    \| \widehat{J}_1(\mathbf{k}, \varepsilon^{-1}\tau)
\mathcal{R}(\mathbf{k}, \varepsilon)\|_{L_2(\Omega) \to L_2 (\Omega) }  &\le \widehat{\mathcal{C}}_1 (1+ |\tau|) \varepsilon,
\\
    \| \widehat{J}_2(\mathbf{k}, \varepsilon^{-1}\tau)
\mathcal{R}(\mathbf{k}, \varepsilon)^{1/2} \|_{L_2(\Omega) \to L_2 (\Omega) }  &\le \widehat{\mathcal{C}}_2(1 + |\tau|).
    \end{align*}
The constants $\widehat{\mathcal{C}}_1$ and $\widehat{\mathcal{C}}_2$ depend only on
$\alpha_0$, $\alpha_1$, $\|g\|_{L_\infty}$, $\|g^{-1}\|_{L_\infty}$, and $r_0$.
\end{theorem}

 Theorem \ref{cos_general_thrm} is deduced from Theorem~\ref{abstr_cos_general_thrm} and relations \eqref{hatS_P=hatA^0_P}, \eqref{R_P}--\eqref{R(k,eps)(I-P)_est}.
 This result was known before (see \cite[Theorem 7.2]{BSu5} and \cite[Subsection 7.3]{M}).

\subsection{The case where~$\widehat{N}(\boldsymbol{\theta}) = 0$}
Now, we improve the result under the additional assumptions.

\begin{theorem}
    \label{cos_enchanced_thrm_11}
Let   $\widehat{N}(\boldsymbol{\theta})$ be the operator defined by~\emph{(\ref{N(theta)})}.
Suppose that~$\widehat{N}(\boldsymbol{\theta}) = 0$ for any $\boldsymbol{\theta} \in \mathbb{S}^{d-1}$.
Then for $\tau \in \mathbb{R}$, $\varepsilon > 0$, and $\mathbf{k} \in \widetilde{\Omega}$ we have
    \begin{align}
    \label{th8.2est1}
    \| \widehat{J}_1(\mathbf{k}, \varepsilon^{-1}\tau)
\mathcal{R}(\mathbf{k}, \varepsilon)^{3/4}\|_{L_2(\Omega) \to L_2 (\Omega) }
&\le \widehat{\mathcal{C}}_3(1 + |\tau|) \varepsilon,
\\
\label{th8.2est2}
    \| \widehat{J}_2(\mathbf{k}, \varepsilon^{-1}\tau)
\mathcal{R}(\mathbf{k}, \varepsilon)^{1/4}\|_{L_2(\Omega) \to L_2 (\Omega) }
&\le \widehat{\mathcal{C}}_4(1 + |\tau|).
    \end{align}
The constants $\widehat{\mathcal{C}}_3$ and $\widehat{\mathcal{C}}_4$
depend only on $\alpha_0$, $\alpha_1$, $\|g\|_{L_\infty}$, $\|g^{-1}\|_{L_\infty}$, and $r_0$.
\end{theorem}

\begin{proof}
From \eqref{abstr_cos_enchanced_est_1}, \eqref{hatS_P=hatA^0_P}, and~\eqref{R_P}
it follows that for $|{\mathbf k}| \le \widehat{t}^{\,0}$ the term
$ \| \widehat{J}_1(\mathbf{k}, \varepsilon^{-1}\tau) \mathcal{R}(\mathbf{k}, \varepsilon)^{3/4} \widehat{P}\|$
satisfies the required estimate of the form  \eqref{th8.2est1}.
By \eqref{R_hatP_est}, for $\mathbf{k} \in \widetilde{\Omega}$ and $|{\mathbf k}| > \widehat{t}^{\,0}$ this term does not
exceed $2 (\widehat{t}^{\,0})^{-1}\varepsilon$.
Finally, by \eqref{R(k,eps)(I-P)_est}, the term
$ \| \widehat{J}_1(\mathbf{k}, \varepsilon^{-1}\tau) \mathcal{R}(\mathbf{k}, \varepsilon)^{3/4} (I - \widehat{P})\|$
does not exceed $2 r_0^{-1} \varepsilon$ for $\mathbf{k} \in \widetilde{\Omega}$. We arrive at \eqref{th8.2est1}.

Let us check \eqref{th8.2est2}.
Relations \eqref{abstr_sin_enchanced_est_1}, \eqref{hatS_P=hatA^0_P}, and~\eqref{R_P}
imply the required estimate for the term
$ \| \widehat{J}_2(\mathbf{k}, \varepsilon^{-1}\tau) \mathcal{R}(\mathbf{k}, \varepsilon)^{1/4} \widehat{P}\|$
for $|{\mathbf k}| \le \widehat{t}^{\,0}$.

Next, in order to estimate the term
$ \| \widehat{J}_2(\mathbf{k}, \varepsilon^{-1}\tau) \mathcal{R}(\mathbf{k}, \varepsilon)^{1/4} (I - \widehat{P})\|$
for $|{\mathbf k}| \le \widehat{t}^{\,0}$, we use the identity
$$
\widehat{\mathcal A}({\mathbf k})^{-1/2} (I - \widehat{P}) = \widehat{\mathcal A}({\mathbf k})^{-1/2} \widehat{F}({\mathbf k})^\perp
+ \widehat{\mathcal A}({\mathbf k})^{-1/2} (\widehat{F}({\mathbf k}) - \widehat{P}),
$$
where $\widehat{F}({\mathbf k})$ is the spectral projection of the operator $\widehat{\mathcal A}({\mathbf k})$ for the interval $[0,\widehat{\delta}]$.
By \eqref{abstr_F(t)_threshold_1} and \eqref{A(k)_nondegenerated_and_c_*} (for $\widehat{\mathcal A}({\mathbf k})$),
 it follows that
the norm $\|\widehat{\mathcal A}({\mathbf k})^{-1/2} (I - \widehat{P})\|$ is uniformly bounded for $|{\mathbf k}| \le \widehat{t}^{\,0}$.
The same is true for $\|\widehat{\mathcal A}^0({\mathbf k})^{-1/2} (I - \widehat{P})\|$. Hence,
the term $ \| \widehat{J}_2(\mathbf{k}, \varepsilon^{-1}\tau) \mathcal{R}(\mathbf{k}, \varepsilon)^{1/4} (I - \widehat{P})\|$
does not exceed a constant (independent of $\tau$, $\varepsilon$, and ${\mathbf k}$) for $|{\mathbf k}| \le \widehat{t}^{\,0}$.

Finally, for $\mathbf{k} \in \widetilde{\Omega}$ and $|{\mathbf k}| > \widehat{t}^{\,0}$
the left-hand side of \eqref{th8.2est2} does not exceed $2 \widehat{c}_*^{-1/2} (\widehat{t}^0)^{-1}$ due to
\eqref{A(k)_nondegenerated_and_c_*} (for $\widehat{\mathcal A}({\mathbf k})$) 
and a similar estimate for $\widehat{\mathcal A}^0({\mathbf k})$.
As a result, we obtain \eqref{th8.2est2}.
\end{proof}

\subsection{The case where~$\widehat{N}_0(\boldsymbol{\theta}) = 0$}
\label{ench_approx2_section}

Now, we reject the assumption of Theorem~\ref{cos_enchanced_thrm_11}, but instead we assume that~$\widehat{N}_0(\boldsymbol{\theta}) = 0$
for any~$\boldsymbol{\theta}$.
We would like to apply  Theorem~\ref{abstr_cos_enchanced_thrm_2}.
However, there is an additional difficulty:
the multiplicities of the eigenvalues of the germ~$\widehat{S} (\boldsymbol{\theta})$
 may change at some points~$\boldsymbol{\theta}$.
 Near such points the distance between some pair of different eigenvalues
 tends to zero, and we are not able to choose the parameters $\widehat{c}^{\circ}_{jl}$ and $\widehat{t}^{\,00}_{jl}$ to be independent of~$\boldsymbol{\theta}$. Therefore, we are forced to impose an additional condition.
We have to take care only about those pairs of eigenvalues for which the corresponding term in~the second formula in \eqref{N0_invar_repr}
is not zero.  Now it is more convenient to use the initial enumeration of the eigenvalues of $\widehat{S} (\boldsymbol{\theta})$:
each eigenvalue is repeated according to its multiplicity and
$\widehat{\gamma}_1 (\boldsymbol{\theta}) \le  \ldots  \le \widehat{\gamma}_n (\boldsymbol{\theta})$.
Denote by $\widehat{P}^{(k)} (\boldsymbol{\theta})$ the orthogonal projection of~$L_2 (\Omega; \mathbb{C}^n)$
onto the eigenspace of $\widehat{S} (\boldsymbol{\theta})$ corresponding to the eigenvalue $\widehat{\gamma}_k (\boldsymbol{\theta})$.
Clearly, for each $\boldsymbol{\theta}$ the operator $\widehat{P}^{(k)} (\boldsymbol{\theta})$ coincides with one of the projections
$\widehat{P}_j (\boldsymbol{\theta})$       introduced in Subsection~\ref{eigenval_multipl_section} (but the number $j$ may depend on~$\boldsymbol{\theta}$).

\begin{condition}
    \label{cond9}

    $1^\circ$. $\widehat{N}_0(\boldsymbol{\theta})=0$ for any $\boldsymbol{\theta} \in \mathbb{S}^{d-1}$.
$2^\circ$.  For any pair of indices $(k,l), 1 \le k,l \le n, k \ne l$, such that $\widehat{\gamma}_k (\boldsymbol{\theta}_0) = \widehat{\gamma}_l (\boldsymbol{\theta}_0) $ for some $\boldsymbol{\theta}_0 \in \mathbb{S}^{d-1}$, we have
$\widehat{P}^{(k)} (\boldsymbol{\theta}) \widehat{N} (\boldsymbol{\theta}) \widehat{P}^{(l)} (\boldsymbol{\theta}) = 0$ for any
\hbox{$\boldsymbol{\theta} \in \mathbb{S}^{d-1}$}.
\end{condition}

  Condition $2^\circ$ can be reformulated
    as follows: we assume that, for the
 \textquotedblleft blocks\textquotedblright \ $\widehat{P}^{(k)} (\boldsymbol{\theta}) \widehat{N} (\boldsymbol{\theta}) \widehat{P}^{(l)} (\boldsymbol{\theta})$ of the operator $\widehat{N} (\boldsymbol{\theta})$  that are not
    identically zero, the corresponding branches of the eigenvalues  $\widehat{\gamma}_k (\boldsymbol{\theta})$ and  $\widehat{\gamma}_l (\boldsymbol{\theta})$   do not intersect.

 Obviously, Condition~\ref{cond9} is ensured by the following more restrictive condition.

\begin{condition}
     \label{cond99}

         $1^\circ$.
$\widehat{N}_0(\boldsymbol{\theta})=0$ for any $\boldsymbol{\theta} \in \mathbb{S}^{d-1}$.
$2^\circ$.  The number $p$ of different eigenvalues of the spectral germ $\widehat{S}(\boldsymbol{\theta})$
does not depend on $\boldsymbol{\theta} \in \mathbb{S}^{d-1}$.
        \end{condition}

Under Condition~\ref{cond99},  denote different eigenvalues of the germ enumerated in the increasing order
 by $\widehat{\gamma}^{\circ}_1(\boldsymbol{\theta}), \ldots, \widehat{\gamma}^{\circ}_p(\boldsymbol{\theta})$.
Then their multiplicities $k_1, \ldots, k_p$ do not depend on $\boldsymbol{\theta} \in \mathbb{S}^{d-1}$.

\begin{remark}\label{rem9.5}
$1^\circ$. Assumption $2^\circ$ of Condition~\emph{\ref{cond99}} is a fortiori satisfied, if the spectrum of the germ $\widehat{S}(\boldsymbol{\theta})$
is simple for any $\boldsymbol{\theta} \in \mathbb{S}^{d-1}$.
$2^\circ$. From Corollary {\rm \ref{S_spec_simple_coroll}} it follows that Condition~\emph{\ref{cond99}} is satisfied, if $b(\boldsymbol{\theta})$ and $g(\mathbf{x})$
have real entries, and the spectrum of the germ $\widehat{S}(\boldsymbol{\theta})$ is simple for any  $\boldsymbol{\theta} \in \mathbb{S}^{d-1}$.
\end{remark}

Under Condition~\ref{cond9}, let
$\widehat{\mathcal{K}} := \{ (k,l) \colon 1 \le k,l \le n, \; k \ne l, \;  \widehat{P}^{(k)} (\boldsymbol{\theta}) \widehat{N} (\boldsymbol{\theta}) \widehat{P}^{(l)} (\boldsymbol{\theta}) \not\equiv 0 \}$.
Denote
$\widehat{c}^{\circ}_{kl} (\boldsymbol{\theta}) := \min \{\widehat{c}_*, n^{-1} |\widehat{\gamma}_k (\boldsymbol{\theta}) - \widehat{\gamma}_l (\boldsymbol{\theta})| \}$, $(k,l) \in \widehat{\mathcal{K}}$.
Since $\widehat{S} (\boldsymbol{\theta})$  is continuous in $\boldsymbol{\theta} \in \mathbb{S}^{d-1}$,
then the perturbation theory of discrete spectrum implies that the functions
  $\widehat{\gamma}_j (\boldsymbol{\theta})$~are continuous  on~$\mathbb{S}^{d-1}$.
By Condition~\ref{cond9}($2^\circ$), for $(k,l) \in \widehat{\mathcal{K}}$ we
have~$|\widehat{\gamma}_k (\boldsymbol{\theta}) - \widehat{\gamma}_l (\boldsymbol{\theta})| > 0$ for any $\boldsymbol{\theta} \in \mathbb{S}^{d-1}$, whence
$\widehat{c}^{\circ}_{kl}:= \min_{\boldsymbol{\theta} \in \mathbb{S}^{d-1}} \widehat{c}^{\circ}_{kl} (\boldsymbol{\theta}) > 0$ for $(k,l) \in \widehat{\mathcal{K}}$. We put
\begin{equation}
\label{hatc^circ}
\widehat{c}^{\circ} := \min_{(k,l) \in \widehat{\mathcal{K}}} \widehat{c}^{\circ}_{kl}.
\end{equation}
Clearly, the number~\eqref{hatc^circ}~ is a realization of~\eqref{abstr_c^circ} chosen independently of $\boldsymbol{\theta}$.
Under Condition~\ref{cond9}, the number $\widehat{t}^{\,00}$ subject to~\eqref{abstr_t00}
also can be chosen independently of $\boldsymbol{\theta} \in \mathbb{S}^{d-1}$. Taking~\eqref{hatdelta_fixation} and~\eqref{hatX_1_estmate} into account, we put
\begin{equation*}
\widehat{t}^{\,00} = (8 \beta_2)^{-1} r_0 \alpha_1^{-3/2} \alpha_0^{1/2} \| g\|_{L_{\infty}}^{-3/2} \| g^{-1}\|_{L_{\infty}}^{-1/2} \widehat{c}^{\circ}.
\end{equation*}
The condition $\widehat{t}^{\,00} \le \widehat{t}^{\,0}$
 is valid, since $\widehat{c}^{\circ} \le \| \widehat{S} (\boldsymbol{\theta}) \| \le \alpha_1 \|g\|_{L_{\infty}}$.

\begin{remark}
Unlike $\widehat{t}^{\,0}$ \emph{(}see~\emph{\eqref{hatt0_fixation}}\emph{)}  that is controlled only in terms of
$r_0$, $\alpha_0$, $\alpha_1$, $\|g\|_{L_{\infty}}$, and $\|g^{-1}\|_{L_{\infty}}$, the number $\widehat{t}^{\,00}$
depends on the spectral characteristics of the germ, namely, on the  minimal distance between its different eigenvalues
$\widehat{\gamma}_k (\boldsymbol{\theta})$ and $ \widehat{\gamma}_l (\boldsymbol{\theta})$ \emph{(}where $(k,l)$
runs through $\widehat{\mathcal{K}}$\emph{)}.
\end{remark}

Under Condition~\ref{cond9}, we apply Theorem~\ref{abstr_cos_enchanced_thrm_2} and
deduce the following result, similarly to the proof of Theorem~\ref{cos_enchanced_thrm_11}.
We have to take into account that now (see Remark~\ref{rem_coeff}) the constants
will depend not only on $\alpha_0$, $\alpha_1$, $\|g\|_{L_\infty}$, $\|g^{-1}\|_{L_\infty}$, and $r_0$, but also on $\widehat{c}^{\circ}$ and $n$.

\begin{theorem}
    \label{cos_enchanced_thrm_2}
     Suppose that Condition~\emph{\ref{cond9}} \emph{(}or more restrictive Condition~\emph{\ref{cond99}}\emph{)}
is satisfied. Then for $\tau \in \mathbb{R}$, $\varepsilon > 0$, and $\mathbf{k} \in \widetilde{\Omega}$ we have
     \begin{align*}
     \| \widehat{J}_1(\mathbf{k}, \varepsilon^{-1} \tau)
\mathcal{R}(\mathbf{k}, \varepsilon)^{3/4}\|_{L_2(\Omega) \to L_2 (\Omega) }
\le \widehat{\mathcal{C}}_5(1 + |\tau|) \varepsilon,
\\
     \| \widehat{J}_2(\mathbf{k}, \varepsilon^{-1} \tau)
\mathcal{R}(\mathbf{k}, \varepsilon)^{1/4}\|_{L_2(\Omega) \to L_2 (\Omega) }
\le \widehat{\mathcal{C}}_{6} (1 + |\tau|).
     \end{align*}
The constants $\widehat{\mathcal{C}}_5$ and $\widehat{\mathcal{C}}_{6}$
  depend only on $\alpha_0$, $\alpha_1$, $\|g\|_{L_\infty}$, $\|g^{-1}\|_{L_\infty}$, $r_0$, $n$, and $\widehat{c}^{\circ}$.
\end{theorem}

\subsection{The sharpness of the result in the general case}

Application of Theorem~\ref{abstr_s<2_general_thrm} allows us to confirm the sharpness of the result of Theorem~\ref{cos_general_thrm}.

\begin{theorem}
    \label{hat_s<2_thrm}
 Let $\widehat{N}_0 (\boldsymbol{\theta})$ be the operator defined by \emph{\eqref{N0_invar_repr}}.
Suppose that  $\widehat{N}_0 (\boldsymbol{\theta}_0) \ne 0$ for some $\boldsymbol{\theta}_0 \in \mathbb{S}^{d-1}$.

\noindent
$1^\circ$. Let $0 \ne \tau \in \mathbb{R}$ and $0 \le s < 2$. Then there does not exist a constant $\mathcal{C} (\tau) > 0$
such that the estimate
    \begin{equation}
    \label{hat_s<2_est_imp}
     \| ( \cos (\varepsilon^{-1} \tau \widehat{\mathcal{A}}(\mathbf{k})^{1/2})  - \cos (\varepsilon^{-1} \tau \widehat{\mathcal{A}}^0(\mathbf{k})^{1/2})) \mathcal{R}(\mathbf{k}, \varepsilon)^{s/2}\|_{L_2(\Omega) \to L_2 (\Omega) }  \le \mathcal{C} (\tau) \varepsilon
    \end{equation}
    holds for almost all $\mathbf{k} = t \boldsymbol{\theta} \in \widetilde{\Omega}$ and sufficiently small $\varepsilon > 0$.

\noindent
$2^\circ$. Let $0 \ne \tau \in \mathbb{R}$ and $0 \le s < 1$. Then there does not exist a constant $\mathcal{C} (\tau) > 0$
such that the estimate
    \begin{equation*}
         \| ( \widehat{\mathcal{A}}(\mathbf{k})^{-1/2} \sin (\varepsilon^{-1} \tau \widehat{\mathcal{A}}(\mathbf{k})^{1/2})  -
     \widehat{\mathcal{A}}^0(\mathbf{k})^{-1/2} \sin (\varepsilon^{-1} \tau \widehat{\mathcal{A}}^0(\mathbf{k})^{1/2})) \mathcal{R}(\mathbf{k}, \varepsilon)^{s/2}\|_{L_2(\Omega) \to L_2 (\Omega) }  \le \mathcal{C} (\tau)
    \end{equation*}
    holds for almost all $\mathbf{k} = t \boldsymbol{\theta} \in \widetilde{\Omega}$ and sufficiently small $\varepsilon > 0$.
\end{theorem}

For the proof we need the following lemma which is similar to Lemma~9.9 from~\cite{Su4}.

\begin{lemma}
    \label{Lipschitz_lemma}
Let $\widehat{\delta}$ and $\widehat{t}^{\,0}$  be given by~\emph{\eqref{hatdelta_fixation}} and~\emph{\eqref{hatt0_fixation}}, respectively.
Let $\widehat{F} (\mathbf{k}) = \widehat{F} (t, \boldsymbol{\theta})$~be the spectral projection of the operator $\widehat{\mathcal{A}}(\mathbf{k})$ for the interval $[0, \widehat{\delta}]$. Then for $|\mathbf{k}| \le \widehat{t}^{\,0}$ and $|\mathbf{k}_0| \le \widehat{t}^{\,0}$ we have
    \begin{gather}
    \label{F(k) - F(k_0)}
    \| \widehat{F} (\mathbf{k}) - \widehat{F} (\mathbf{k}_0)\|_{L_2(\Omega) \to L_2 (\Omega) } \le \widehat{C}' | \mathbf{k} - \mathbf{k}_0|, \\
    \label{hatA^1/2 (k) - hatA^1/2 (k_0)}
    \| \widehat{\mathcal{A}}(\mathbf{k})^{1/2} \widehat{F} (\mathbf{k}) - \widehat{\mathcal{A}}(\mathbf{k}_0)^{1/2} \widehat{F} (\mathbf{k}_0)\|_{L_2(\Omega) \to L_2 (\Omega) } \le \widehat{C}'' | \mathbf{k} - \mathbf{k}_0|, \\
    \label{cos_hatA^1/2 (k) - cos_hatA^1/2 (k_0)}
  \| \cos(\tau \widehat{\mathcal{A}}(\mathbf{k})^{1/2}) \widehat{F} (\mathbf{k}) - \cos(\tau \widehat{\mathcal{A}}(\mathbf{k}_0)^{1/2}) \widehat{F} (\mathbf{k}_0)\|_{L_2(\Omega) \to L_2 (\Omega) }
\le (2 \widehat{C}' + \widehat{C}'' |\tau|) | \mathbf{k} - \mathbf{k}_0|,
\\
\label{7.13a}
\begin{split}
  \| \widehat{\mathcal{A}}(\mathbf{k})^{-1/2} \sin(\tau \widehat{\mathcal{A}}(\mathbf{k})^{1/2}) \widehat{F} (\mathbf{k}) -
  \widehat{\mathcal{A}}(\mathbf{k}_0)^{-1/2} \sin(\tau \widehat{\mathcal{A}}(\mathbf{k}_0)^{1/2}) \widehat{F} (\mathbf{k}_0)\|_{L_2(\Omega) \to L_2 (\Omega) }
\\
\le (2\widehat{C}'|\tau| + \widehat{C}'' \tau^2 /2)   | \mathbf{k} - \mathbf{k}_0|.
\end{split}   
 \end{gather}
\end{lemma}

\begin{proof}
Estimate~(\ref{F(k) - F(k_0)}) has been proved in~\cite[Lemma~9.9]{Su4}.

Let us prove~\eqref{hatA^1/2 (k) - hatA^1/2 (k_0)}. Let $\mathbf{k}, \mathbf{k}_0 \ne 0$.
Using representation of the form~\eqref{abstr_A_sqrt_repres} for $\widehat{\mathcal{A}}(\mathbf{k})^{1/2} \widehat{F} (\mathbf{k})$, we have
$$
    \widehat{\mathcal{A}}(\mathbf{k})^{1/2} \widehat{F} (\mathbf{k}) - \widehat{\mathcal{A}}(\mathbf{k}_0)^{1/2} \widehat{F} (\mathbf{k}_0) =  \frac{1}{\pi} \int_{0}^{\infty} \zeta^{-1/2}\Upsilon(\zeta,\mathbf{k},\mathbf{k}_0)\, d\zeta,
$$
where
$\Upsilon(\zeta,\mathbf{k},\mathbf{k}_0):=
(\widehat{\mathcal{A}}(\mathbf{k}) \widehat{F} (\mathbf{k}) + \zeta I)^{-1} \widehat{\mathcal{A}}(\mathbf{k}) \widehat{F} (\mathbf{k}) - (\widehat{\mathcal{A}}(\mathbf{k}_0) \widehat{F} (\mathbf{k}_0) + \zeta I)^{-1} \widehat{\mathcal{A}}(\mathbf{k}_0) \widehat{F} (\mathbf{k}_0)$.
It is easily seen that
\begin{multline*}
\Upsilon(\zeta, \mathbf{k},\mathbf{k}_0)=
 \zeta (\widehat{\mathcal{A}}(\mathbf{k}) \widehat{F} (\mathbf{k}) + \zeta I)^{-1} \left( \widehat{\mathcal{A}}(\mathbf{k}) \widehat{F} (\mathbf{k}) - \widehat{\mathcal{A}}(\mathbf{k}_0) \widehat{F} (\mathbf{k}_0)\right)  (\widehat{\mathcal{A}}(\mathbf{k}_0) \widehat{F} (\mathbf{k}_0) + \zeta I)^{-1}
\\
 =    \Upsilon_1(\zeta,\mathbf{k}, \mathbf{k}_0) + \Upsilon_2(\zeta,\mathbf{k}, \mathbf{k}_0) + \Upsilon_3(\zeta, \mathbf{k}, \mathbf{k}_0),
\end{multline*}
where
\begin{align*}
 \Upsilon_1(\zeta, \mathbf{k}, \mathbf{k}_0) =& \,\zeta (\widehat{\mathcal{A}}(\mathbf{k})  + \zeta I)^{-1}
\left( \widehat{F} (\mathbf{k}) \widehat{\mathcal{A}}(\mathbf{k}) \widehat{F} (\mathbf{k}_0)  - \widehat{F} (\mathbf{k}) \widehat{\mathcal{A}}(\mathbf{k}_0) \widehat{F} (\mathbf{k}_0)\right)  (\widehat{\mathcal{A}}(\mathbf{k}_0) + \zeta I)^{-1},
\\
 \Upsilon_2(\zeta,\mathbf{k}, \mathbf{k}_0) =&
 - (I - \widehat{F} (\mathbf{k})) \widehat{\mathcal{A}}(\mathbf{k}_0) \widehat{F} (\mathbf{k}_0) (\widehat{\mathcal{A}}(\mathbf{k}_0) + \zeta I)^{-1}
=  (\widehat{F} (\mathbf{k}) - \widehat{F} (\mathbf{k}_0))
 \widehat{\mathcal{A}}(\mathbf{k}_0) \widehat{F} (\mathbf{k}_0) (\widehat{\mathcal{A}}(\mathbf{k}_0) + \zeta I)^{-1},
\\
\Upsilon_3(\zeta, \mathbf{k}, \mathbf{k}_0) =&\,
 (\widehat{\mathcal{A}}(\mathbf{k})  + \zeta I)^{-1} \widehat{F} (\mathbf{k}) \widehat{\mathcal{A}}(\mathbf{k}) (I-\widehat{F} (\mathbf{k}_0))
=  (\widehat{\mathcal{A}}(\mathbf{k})  + \zeta I)^{-1} \widehat{F} (\mathbf{k}) \widehat{\mathcal{A}}(\mathbf{k}) (\widehat{F} (\mathbf{k})-\widehat{F} (\mathbf{k}_0)).
\end{align*}
Hence,
$
\widehat{\mathcal{A}}(\mathbf{k})^{1/2} \widehat{F} (\mathbf{k}) - \widehat{\mathcal{A}}(\mathbf{k}_0)^{1/2} \widehat{F} (\mathbf{k}_0) =
\Omega_1(\mathbf{k}, \mathbf{k}_0)+ \Omega_2(\mathbf{k}, \mathbf{k}_0) + \Omega_3(\mathbf{k}, \mathbf{k}_0),
$
where
$$
\Omega_j(\mathbf{k}, \mathbf{k}_0) = \frac{1}{\pi} \int_{0}^{\infty} \zeta^{-1/2} \Upsilon_j(\zeta,\mathbf{k}, \mathbf{k}_0)
\,d\zeta,\quad j=1,2,3.
$$

To estimate $\Omega_1(\mathbf{k}, \mathbf{k}_0)$,
consider the difference of the sesquilinear forms of the operators
$\widehat{\mathcal{A}}(\mathbf{k})$ and $\widehat{\mathcal{A}}(\mathbf{k}_0)$
on the elements  $\mathbf{u}, \mathbf{v} \in \widetilde{H}^1 (\Omega; \mathbb{C}^n)$:
\begin{equation*}
\widehat{\mathfrak{a}} (\mathbf{k}) [\mathbf{u}, \mathbf{v}] - \widehat{\mathfrak{a}} (\mathbf{k}_0) [\mathbf{u}, \mathbf{v}] = 
( g b(\mathbf{k} - \mathbf{k}_0) \mathbf{u}, b(\mathbf{D} + \mathbf{k}_0) \mathbf{v} )_{L_2 (\Omega)} + ( g b(\mathbf{D} + \mathbf{k}) \mathbf{u}, b(\mathbf{k} - \mathbf{k}_0) \mathbf{v})_{L_2 (\Omega)}.
\end{equation*}
Combining this with~\eqref{rank_alpha_ineq}, we see that
\begin{equation}
\label{9.23a}
|\widehat{\mathfrak{a}} (\mathbf{k}) [\mathbf{u}, \mathbf{v}] - \widehat{\mathfrak{a}} (\mathbf{k}_0) [\mathbf{u}, \mathbf{v}]| \le 
\alpha_1^{1/2} \| g \|_{L_\infty}^{1/2}  |\mathbf{k}-\mathbf{k}_0| \left( \| \mathbf{u} \|_{L_2}  \| \widehat{\mathcal{A}}(\mathbf{k}_0)^{1/2} \mathbf{v} \|_{L_2} +  \| \widehat{\mathcal{A}}(\mathbf{k})^{1/2} \mathbf{u} \|_{L_2}  \|\mathbf{v} \|_{L_2}\right).
\end{equation}
Substituting $\mathbf{u} = \widehat{F} (\mathbf{k}) \boldsymbol{\varphi}$ and $\mathbf{v} =  \widehat{F} (\mathbf{k}_0) \boldsymbol{\psi}$
with $\boldsymbol{\varphi}, \boldsymbol{\psi} \in L_2 (\Omega; \mathbb{C}^n)$ in \eqref{9.23a},
and taking into account that
\begin{equation}
\label{hat_sqrtA(k)F(k)_up_est}
\| \widehat{\mathcal{A}}(\mathbf{k})^{1/2} \widehat{F} (\mathbf{k}) \|_{L_2(\Omega) \to L_2 (\Omega) } \le
(1+ \beta_2)^{1/2} \alpha_1^{1/2} \| g \|_{L_{\infty}}^{1/2} |\mathbf{k}|, \qquad |\mathbf{k}| \le \widehat{t}^{\,0},
\end{equation}
 (which follows from~\eqref{1.6b} and~\eqref{hatX_1_estmate}),
we have
\begin{equation*}
\| \widehat{F} (\mathbf{k}) \widehat{\mathcal{A}}(\mathbf{k}) \widehat{F} (\mathbf{k}_0)  - \widehat{F} (\mathbf{k}) \widehat{\mathcal{A}}(\mathbf{k}_0) \widehat{F} (\mathbf{k}_0) \|_{L_2(\Omega) \to L_2 (\Omega) }
\le  (1+\beta_2)^{1/2} \alpha_1 \| g \|_{L_{\infty}}   (|\mathbf{k}| + |\mathbf{k}_0|) |\mathbf{k} - \mathbf{k}_0|
\end{equation*}
for $|\mathbf{k}|,|\mathbf{k}_0|  \le \widehat{t}^{\,0}$.
Combining this with~\eqref{A(k)_nondegenerated_and_c_*}, we obtain
\begin{equation*}
\| \Omega_1(\mathbf{k}, \mathbf{k}_0) \|_{L_2(\Omega) \to L_2 (\Omega) } \le
(1+\beta_2)^{1/2}  \alpha_1 \| g \|_{L_{\infty}} \widehat{c}_*^{-1/2} |\mathbf{k} - \mathbf{k}_0|, \quad  0 <  |\mathbf{k}|,|\mathbf{k}_0|  \le \widehat{t}^{\,0}.
\end{equation*}
Next, from~\eqref{A(k)_nondegenerated_and_c_*}, \eqref{F(k) - F(k_0)}, and~\eqref{hat_sqrtA(k)F(k)_up_est} we deduce
\begin{equation*}
\| \Omega_2(\mathbf{k}, \mathbf{k}_0) \|_{L_2(\Omega) \to L_2 (\Omega) } \le
\widehat{C}' (1+\beta_2) \alpha_1 \| g \|_{L_{\infty}} \widehat{c}_*^{-1/2} \widehat{t}^{\,0} |\mathbf{k} - \mathbf{k}_0|, \quad  0 <  |\mathbf{k}|,|\mathbf{k}_0|  \le \widehat{t}^{\,0}.
\end{equation*}
 The term $\Omega_3(\mathbf{k}, \mathbf{k}_0)$
satisfies the same estimate. Thus, 
\eqref{hatA^1/2 (k) - hatA^1/2 (k_0)} is proved in the case where $\mathbf{k}, \mathbf{k}_0 \ne 0$.
If, for instance, $\mathbf{k}_0 = 0$, then~\eqref{hatA^1/2 (k) - hatA^1/2 (k_0)} follows directly from~\eqref{hat_sqrtA(k)F(k)_up_est} and  the relation $\widehat{\mathcal{A}}(0)^{1/2} \widehat{F} (0) = 0$.

Let us prove estimate~\eqref{cos_hatA^1/2 (k) - cos_hatA^1/2 (k_0)}. We have
\begin{multline}
\label{cos_k_k0_f1}
e^{i \tau \widehat{\mathcal{A}}(\mathbf{k})^{1/2}} \widehat{F} (\mathbf{k}) - e^{i\tau \widehat{\mathcal{A}}(\mathbf{k}_0)^{1/2}} \widehat{F} (\mathbf{k}_0) = e^{i \tau \widehat{\mathcal{A}}(\mathbf{k})^{1/2}} \widehat{F} (\mathbf{k}) (\widehat{F} (\mathbf{k}) - \widehat{F} (\mathbf{k}_0)) + \\ + (\widehat{F} (\mathbf{k}) - \widehat{F} (\mathbf{k}_0)) e^{i \tau \widehat{\mathcal{A}}(\mathbf{k}_0)^{1/2}} \widehat{F} (\mathbf{k}_0) + \Xi (\tau, \mathbf{k}, \mathbf{k}_0),
\end{multline}
where
$\Xi (\tau, \mathbf{k}, \mathbf{k}_0) = e^{i \tau \widehat{\mathcal{A}}(\mathbf{k})^{1/2}} \widehat{F} (\mathbf{k}) \widehat{F} (\mathbf{k}_0) - \widehat{F} (\mathbf{k}) e^{i\tau \widehat{\mathcal{A}}(\mathbf{k}_0)^{1/2}} \widehat{F} (\mathbf{k}_0)$.
The sum of the first two terms in~\eqref{cos_k_k0_f1} does not exceed $2 \widehat{C}' | \mathbf{k}- \mathbf{k}_0|$,
in view of~(\ref{F(k) - F(k_0)}). The third term can be written as
$\Xi (\tau, \mathbf{k}, \mathbf{k}_0) = e^{i \tau \widehat{\mathcal{A}}(\mathbf{k})^{1/2}} \Sigma(\tau, \mathbf{k}, \mathbf{k}_0)$,
where 
$\Sigma(\tau, \mathbf{k}, \mathbf{k}_0) = \widehat{F} (\mathbf{k}) \widehat{F} (\mathbf{k}_0) - e^{-i\tau \widehat{\mathcal{A}}(\mathbf{k})^{1/2}} \widehat{F} (\mathbf{k}) e^{i \tau \widehat{\mathcal{A}}(\mathbf{k}_0)^{1/2}} \widehat{F}(\mathbf{k}_0).$
Obviously, $\Sigma(0, \mathbf{k}, \mathbf{k}_0) = 0$, and
 $\Sigma'(\tau, \mathbf{k}, \mathbf{k}_0) = \frac{d \Sigma(\tau, \mathbf{k}, \mathbf{k}_0)}{d \tau}$ is given by
\begin{equation*}
\Sigma'(\tau, \mathbf{k}, \mathbf{k}_0)
= i \widehat{F} (\mathbf{k}) e^{-i\tau \widehat{\mathcal{A}}(\mathbf{k})^{1/2}} (\widehat{\mathcal{A}}(\mathbf{k})^{1/2} \widehat{F} (\mathbf{k}) - \widehat{\mathcal{A}}(\mathbf{k_0})^{1/2} \widehat{F} (\mathbf{k_0}))
e^{i\tau \widehat{\mathcal{A}}(\mathbf{k_0})^{1/2}} \widehat{F} (\mathbf{k_0}).
\end{equation*}
Integrating over the interval $[0, \tau]$ and taking~\eqref{hatA^1/2 (k) - hatA^1/2 (k_0)} into account, we obtain
\begin{equation*}
\| \Xi (\tau, \mathbf{k}, \mathbf{k}_0) \|_{L_2(\Omega) \to L_2 (\Omega) } = \|\Sigma(\tau, \mathbf{k}, \mathbf{k}_0)\|_{L_2(\Omega) \to L_2 (\Omega) } \le \widehat{C}'' |\tau| | \mathbf{k} - \mathbf{k}_0|.
\end{equation*}
We arrive at~\eqref{cos_hatA^1/2 (k) - cos_hatA^1/2 (k_0)}.

Estimate \eqref{7.13a} follows from \eqref{cos_hatA^1/2 (k) - cos_hatA^1/2 (k_0)} by integration over the interval $(0,\tau)$. 
\end{proof}

\noindent \textbf{Proof of Theorem~\ref{hat_s<2_thrm}}.
Let us prove assertion $1^\circ$.  It suffices to consider $1 \le s < 2 $.
We prove by contradiction.
Assume that for some $\tau \ne 0$ and $1 \le s < 2$ there exists a constant $\mathcal{C}(\tau) > 0$  such that estimate~\eqref{hat_s<2_est_imp}
holds for almost every $\mathbf{k} \in \widetilde{\Omega}$ and sufficiently small $\varepsilon > 0$.
Multiplying the operator under the norm sign in ~\eqref{hat_s<2_est_imp} by $\widehat{P}$ and using~\eqref{R_P},
 we see that
\begin{equation}
\label{hat_s<2_proof_f1}
\| \bigl( \cos (\varepsilon^{-1} \tau \widehat{\mathcal{A}}(\mathbf{k})^{1/2})  - \cos (\varepsilon^{-1} \tau \widehat{\mathcal{A}}^0(\mathbf{k})^{1/2}) \bigr) \widehat{P}\|_{L_2(\Omega) \to L_2 (\Omega) } \varepsilon^s (|\mathbf{k}|^2 + \varepsilon^2)^{-s/2}
 \le {\mathcal{C}}(\tau) \varepsilon
\end{equation}
for almost every $\mathbf{k} \in \widetilde{\Omega}$ and sufficiently small $\varepsilon > 0$.

Now, let $|\mathbf{k}| \le \widehat{t}^{\,0}$. By \eqref{abstr_F(t)_threshold_1},
\begin{equation}
\label{hat_s<2_proof_f2}
\| \widehat{F} (\mathbf{k}) - \widehat{P} \|_{L_2 (\Omega) \to L_2 (\Omega)} \le \widehat{C}_1 |\mathbf{k}|, \quad |\mathbf{k}| \le \widehat{t}^{\,0}.
\end{equation}
From~\eqref{hat_s<2_proof_f1} and~\eqref{hat_s<2_proof_f2}
 it follows that there exists a constant $\widetilde{\mathcal C}(\tau)>0$ such that
\begin{equation}
\label{hat_s<2_proof_f3}
\|  \cos (\varepsilon^{-1} \tau \widehat{\mathcal{A}}(\mathbf{k})^{1/2}) \widehat{F} (\mathbf{k})  - \cos (\varepsilon^{-1} \tau \widehat{\mathcal{A}}^0(\mathbf{k})^{1/2}) \widehat{P}\|_{L_2(\Omega) \to L_2 (\Omega) } \varepsilon^s (|\mathbf{k}|^2 + \varepsilon^2)^{-s/2}
\le \widetilde{\mathcal{C}}(\tau) \varepsilon
\end{equation}
for almost every $\mathbf{k}$ in the ball $|\mathbf{k}| \le \widehat{t}^{\,0}$
 and sufficiently small $\varepsilon > 0$.

Observe that  $\widehat{P}$ is the spectral projection of the operator $\widehat{\mathcal{A}}^0 (\mathbf{k})$ for the interval $[0, \widehat{\delta}]$. Applying Lemma~\ref{Lipschitz_lemma} to $\widehat{\mathcal{A}} (\mathbf{k})$ and $\widehat{\mathcal{A}}^0 (\mathbf{k})$,
 we conclude that for fixed $\tau$ and $\varepsilon$ the operator under the norm sign in~\eqref{hat_s<2_proof_f3} is continuous with respect to $\mathbf{k}$ in the ball $|\mathbf{k}| \le \widehat{t}^{\,0}$. Consequently, estimate~\eqref{hat_s<2_proof_f3} holds for
all $\mathbf{k}$ in that ball. In particular, it holds at the point $\mathbf{k} = t\boldsymbol{\theta}_0$ if $t \le \widehat{t}^{\,0}$. Applying~\eqref{hat_s<2_proof_f2} once more, we see that there exists a constant  $\check{\mathcal{C}}(\tau)>0$ such that
\begin{equation}
\label{hat_s<2_proof_f4}
\|  \bigl( \cos (\varepsilon^{-1} \tau \widehat{\mathcal{A}}(t\boldsymbol{\theta}_0)^{1/2}) - \cos (\varepsilon^{-1} \tau \widehat{\mathcal{A}}^0(t\boldsymbol{\theta}_0)^{1/2}) \bigr) \widehat{P}\|_{L_2(\Omega) \to L_2 (\Omega) } \varepsilon^s (t^2 + \varepsilon^2)^{-s/2}
\le \check{\mathcal{C}}(\tau) \varepsilon
\end{equation}
for all $t \le \widehat{t}^{\,0}$ and sufficiently small $\varepsilon$.
Estimate~\eqref{hat_s<2_proof_f4} corresponds to the abstract estimate~\eqref{abstr_s<2_est_imp}.
Since $\widehat{N}_0 (\boldsymbol{\theta}_0) \ne 0$, applying Theorem~\ref{abstr_s<2_general_thrm}($1^\circ$), we
arrive at a contradiction.

Assertion $2^\circ$ is deduced from Theorem~\ref{abstr_s<2_general_thrm}($2^\circ$) in a similar way.
   $\square$

\section{The operator $\mathcal{A} (\mathbf{k})$. Application of the scheme of Section~\ref{abstr_sandwiched_section}}

\subsection{The operator $\mathcal{A} (\mathbf{k})$}

We apply the scheme of Section~\ref{abstr_sandwiched_section} to study the operator
$\mathcal{A} (\mathbf{k}) = f^* \widehat{\mathcal{A}} (\mathbf{k}) f$.
Now  $\mathfrak{H} = \widehat{\mathfrak{H}} = L_2 (\Omega; \mathbb{C}^n)$, $\mathfrak{H}_* = L_2 (\Omega; \mathbb{C}^m)$, the role of  $A(t)$ is played by $A(t, \boldsymbol{\theta}) = \mathcal{A}(\mathbf{k})$, the role of $\widehat{A}(t)$ is played by $\widehat{A}(t, \boldsymbol{\theta}) = \widehat{\mathcal{A}}(\mathbf{k})$. Next, the  isomorphism $M$
is the operator of multiplication by the matrix-valued function $f(\mathbf{x})$.
The operator $Q$ is the operator of multiplication by
$Q(\mathbf{x}) = (f (\mathbf{x}) f (\mathbf{x})^*)^{-1}$.
The block of $Q$ in the subspace $\widehat{\mathfrak{N}}$ (see~\eqref{Ker3}) is the operator of multiplication by the constant matrix
$\overline{Q} = (\underline{f f^*})^{-1} = |\Omega|^{-1} \int_{\Omega} (f (\mathbf{x}) f (\mathbf{x})^*)^{-1} d \mathbf{x}$.
Next, $M_0$ is the operator of multiplication by the matrix
\begin{equation}
\label{f_0}
f_0 = (\overline{Q})^{-1/2} = (\underline{f f^*})^{1/2}.
\end{equation}
Note that $| f_0 | \le \| f \|_{L_{\infty}}$ and $| f_0^{-1} | \le \| f^{-1} \|_{L_{\infty}}$.

In $L_2 (\mathbb{R}^d; \mathbb{C}^n)$,  we define the operator
\begin{equation}
\label{A0}
\mathcal{A}^0 := f_0 \widehat{\mathcal{A}}^0 f_0 = f_0 b(\mathbf{D})^* g^0 b(\mathbf{D}) f_0.
\end{equation}
Let $\mathcal{A}^0 (\mathbf{k})$~be the corresponding family of operators in $L_2 (\Omega; \mathbb{C}^n)$. Then
$\mathcal{A}^0 (\mathbf{k}) = f_0 \widehat{\mathcal{A}}^0 (\mathbf{k}) f_0$.
By~\eqref{Ker3} and~\eqref{hatS_P=hatA^0_P}, we have
\begin{equation}
\label{f_0 hatS f_0 P = A^0}
f_0 \widehat{S} (\mathbf{k}) f_0 \widehat{P} = \mathcal{A}^0 (\mathbf{k}) \widehat{P}.
\end{equation}

\subsection{The analytic branches of eigenvalues and eigenvectors}
\label{sndw_eigenvalues_and_eigenvectors_section}
According to~\eqref{abstr_S_and_S_hat_relation},  the spectral germ $S(\boldsymbol{\theta})$ of the operator $A (t, \boldsymbol{\theta})$  acting in the subspace $\mathfrak{N}$ (see \eqref{Ker2}) is represented as
$S(\boldsymbol{\theta}) = P f^* b(\boldsymbol{\theta})^* g^0 b(\boldsymbol{\theta}) f|_{\mathfrak{N}}$,
where $P$~is the orthogonal projection of $L_2 (\Omega; \mathbb{C}^n)$ onto $\mathfrak{N}$.

The analytic (in $t$) branches of the eigenvalues
$\lambda_l (t, \boldsymbol{\theta})$ and the eigenvectors $\varphi_l (t, \boldsymbol{\theta})$ of $A (t, \boldsymbol{\theta})$
 admit the power series expansions of the form~\eqref{abstr_A(t)_eigenvalues_series}, \eqref{abstr_A(t)_eigenvectors_series}
 with the coefficients depending on $\boldsymbol{\theta}$:
\begin{gather}
\label{A_eigenvalues_series}
\lambda_l (t, \boldsymbol{\theta}) = \gamma_l (\boldsymbol{\theta}) t^2 + \mu_l (\boldsymbol{\theta}) t^3 + \ldots, \qquad l = 1, \ldots, n,
\\
\label{A_eigenvectors_series}
\varphi_l (t, \boldsymbol{\theta}) = \omega_l (\boldsymbol{\theta}) + t \psi^{(1)}_l (\boldsymbol{\theta}) + \ldots, \qquad l = 1, \ldots, n.
\end{gather}
The vectors $\omega_1 (\boldsymbol{\theta}), \ldots, \omega_n (\boldsymbol{\theta})$ form an orthonormal basis in the subspace~$\mathfrak{N}$, and the vectors
$\zeta_l (\boldsymbol{\theta}) = f \omega_l (\boldsymbol{\theta})$, $l = 1, \ldots, n$,
form a basis in $\widehat{\mathfrak{N}}$~(see~\eqref{Ker3})  orthonormal with the weight $\overline{Q}$.
The numbers $\gamma_l (\boldsymbol{\theta})$ and the elements $\omega_l (\boldsymbol{\theta})$
are eigenvalues and eigenvectors of the spectral germ $S(\boldsymbol{\theta})$.
According to~\eqref{abstr_hatS_gener_spec_problem},
\begin{equation}
\label{hatS_gener_spec_problem}
b(\boldsymbol{\theta})^* g^0 b(\boldsymbol{\theta}) \zeta_l (\boldsymbol{\theta}) = \gamma_l (\boldsymbol{\theta}) \overline{Q} \zeta_l (\boldsymbol{\theta}), \qquad l = 1, \ldots, n.
\end{equation}

\subsection{The operator $\widehat{N}_Q (\boldsymbol{\theta})$}
We need to describe the operator $\widehat{N}_Q$ (see Subsection~\ref{abstr_hatZ_Q_and_hatN_Q_section}).
Let $\Lambda_Q(\mathbf{x})$ be the $\Gamma$-periodic solution of the problem
\begin{equation*}
b(\mathbf{D})^* g(\mathbf{x}) (b(\mathbf{D}) \Lambda_Q(\mathbf{x}) + \mathbf{1}_m) = 0, \qquad \int_{\Omega} Q(\mathbf{x}) \Lambda_Q(\mathbf{x}) \, d \mathbf{x} = 0.
\end{equation*}
Clearly,  $\Lambda_Q(\mathbf{x}) = \Lambda(\mathbf{x}) -(\overline{Q})^{-1} (\overline{Q \Lambda})$.
As shown in~\cite[Section~5]{BSu3}, the operator $\widehat{N}_Q (\boldsymbol{\theta})$ takes the form
\begin{align}
\label{N_Q(theta)}
\widehat{N}_Q (\boldsymbol{\theta}) &= b(\boldsymbol{\theta})^* L_Q (\boldsymbol{\theta}) b(\boldsymbol{\theta}) \widehat{P},
\\
\nonumber
L_Q (\boldsymbol{\theta}) &:= | \Omega |^{-1} \int_{\Omega} \bigl(\Lambda_Q(\mathbf{x})^*b(\boldsymbol{\theta})^* \widetilde{g} (\mathbf{x}) + \widetilde{g} (\mathbf{x})^* b(\boldsymbol{\theta}) \Lambda_Q(\mathbf{x}) \bigr)\, d \mathbf{x}.
\end{align}

Some cases where $\widehat{N}_Q (\boldsymbol{\theta}) = 0$ were distinguished
in~\cite[Section~5]{BSu3}.

\begin{proposition}
    \label{N_Q=0_proposit}
   Suppose that at least one of the following conditions is fulfilled{\rm :}

$1^\circ$.
$\mathcal{A} = f(\mathbf{x})^*\mathbf{D}^* g(\mathbf{x}) \mathbf{D}f(\mathbf{x})$, where $g(\mathbf{x})$~ is a symmetric matrix with real entries.

$2^\circ$. Relations~\emph{\eqref{g0=overline_g_relat}} are satisfied, i.~e., $g^0 = \overline{g}$.

\noindent
    Then $\widehat{N}_Q (\boldsymbol{\theta}) = 0$ for any $\boldsymbol{\theta} \in \mathbb{S}^{d-1}$.
\end{proposition}

Recall that (see Subsection~\ref{abstr_hatZ_Q_and_hatN_Q_section}) $\widehat{N}_Q (\boldsymbol{\theta}) = \widehat{N}_{0, Q} (\boldsymbol{\theta}) + \widehat{N}_{*,Q} (\boldsymbol{\theta})$. By~\eqref{abstr_hatN_0Q_N_*Q},
$$
\widehat{N}_{0, Q} (\boldsymbol{\theta}) = \sum_{l=1}^{n} \mu_l (\boldsymbol{\theta}) (\cdot, \overline{Q} \zeta_l(\boldsymbol{\theta}))_{L_2(\Omega)} \overline{Q} \zeta_l(\boldsymbol{\theta}).
$$
We have
\begin{equation*}
(\widehat{N}_Q (\boldsymbol{\theta}) \zeta_l (\boldsymbol{\theta}), \zeta_l (\boldsymbol{\theta}))_{L_2 (\Omega)} = (\widehat{N}_{0,Q} (\boldsymbol{\theta}) \zeta_l (\boldsymbol{\theta}), \zeta_l (\boldsymbol{\theta}))_{L_2 (\Omega)} = \mu_l (\boldsymbol{\theta}), \  l=1, \ldots, n.
\end{equation*}

In \cite[Proposition 5.2]{BSu3}, the following statement was proved.

\begin{proposition}
    Suppose that the matrices $b(\boldsymbol{\theta})$, $g (\mathbf{x})$, and $Q(\mathbf{x})$~have real entries. Suppose that in~\emph{\eqref{A_eigenvectors_series}}  the \textquotedblleft embryos\textquotedblright \ $\omega_l (\boldsymbol{\theta}), \; l = 1, \ldots, n,$ can be chosen so that the vectors  $\zeta_l (\boldsymbol{\theta}) = f \omega_l (\boldsymbol{\theta})$ are real. Then in~\emph{\eqref{A_eigenvalues_series}} we have $\mu_l (\boldsymbol{\theta}) = 0$,  $l=1, \ldots, n,$
i.~e., $\widehat{N}_{0,Q} (\boldsymbol{\theta}) = 0$ for any $\boldsymbol{\theta} \in \mathbb{S}^{d-1}$.
\end{proposition}

In the \textquotedblleft real\textquotedblright \ case
$\widehat{S} (\boldsymbol{\theta})$  and
 $\overline{Q}$ are symmetric matrices with real entries.
 Clearly,  if the eigenvalue $\gamma_j (\boldsymbol{\theta})$ of the
problem~\eqref{hatS_gener_spec_problem}
is simple, then the eigenvector $\zeta_j (\boldsymbol{\theta}) = f \omega_j (\boldsymbol{\theta})$
 is defined uniquely up to a phase factor, and we can always choose it to be real.  We arrive at the following corollary.

\begin{corollary}
    \label{sndw_simple_spec_N0Q=0_cor}
Suppose that   the matrices $b(\boldsymbol{\theta})$, $g (\mathbf{x})$, and $Q(\mathbf{x})$~ have real entries.
Suppose that the spectrum of the generalized spectral problem~\emph{\eqref{hatS_gener_spec_problem}}
is simple. Then $\widehat{N}_{0,Q} (\boldsymbol{\theta}) = 0$ for any $\boldsymbol{\theta} \in \mathbb{S}^{d-1}$.
\end{corollary}

\subsection{Multiplicities of the eigenvalues of the germ}
\label{sndw_eigenval_multipl_section}

This subsection concerns the case where $n \ge 2$. We return to the notation of Section~\ref{abstr_cluster_section}.
In general, the number $p(\boldsymbol{\theta})$ of different eigenvalues $\gamma^{\circ}_1 (\boldsymbol{\theta}), \ldots, \gamma^{\circ}_{p(\boldsymbol{\theta})} (\boldsymbol{\theta})$ of the germ $S(\boldsymbol{\theta})$ (or of 
problem~\eqref{hatS_gener_spec_problem})
and their multiplicities depend on the parameter $\boldsymbol{\theta} \in \mathbb{S}^{d-1}$.  For a fixed $\boldsymbol{\theta}$, let $\mathfrak{N}_j (\boldsymbol{\theta})$ be the eigenspace  of the germ $S (\boldsymbol{\theta})$ corresponding to the eigenvalue $\gamma^{\circ}_j (\boldsymbol{\theta})$. Then $f \mathfrak{N}_j (\boldsymbol{\theta})$~ is the eigenspace of the problem~\eqref{hatS_gener_spec_problem}  corresponding to the same eigenvalue $\gamma^{\circ}_j (\boldsymbol{\theta})$.
Let $\mathcal{P}_j (\boldsymbol{\theta})$ denote the \textquotedblleft skew\textquotedblright \  projection of $L_2(\Omega; \mathbb{C}^n)$
onto the subspace $f \mathfrak{N}_j (\boldsymbol{\theta})$; $\mathcal{P}_j (\boldsymbol{\theta})$  is orthogonal with respect to the inner product with the weight $\overline{Q}$. Then, by~\eqref{abstr_hatN_0Q_N_*Q_invar_repr},  we have
\begin{equation}
\label{N0Q_invar_repr}
\widehat{N}_{0,Q} (\boldsymbol{\theta}) = \sum_{j =1}^{p(\boldsymbol{\theta})} \mathcal{P}_j (\boldsymbol{\theta})^* \widehat{N}_Q (\boldsymbol{\theta}) \mathcal{P}_j (\boldsymbol{\theta}), \quad
\widehat{N}_{*,Q} (\boldsymbol{\theta}) = \sum_{{1 \le j, l \le p(\boldsymbol{\theta}):\; j \ne l}} \mathcal{P}_j (\boldsymbol{\theta})^* \widehat{N}_Q (\boldsymbol{\theta}) \mathcal{P}_l (\boldsymbol{\theta}).
\end{equation}

\section{ Approximations of the sandwiched operators \\
$\cos(\varepsilon^{-1} \tau \mathcal{A}(\mathbf{k})^{1/2})$ and $\mathcal{A}(\mathbf{k})^{-1/2} \sin(\varepsilon^{-1} \tau \mathcal{A}(\mathbf{k})^{1/2})$}

\subsection{The general case}

Denote
\begin{align*}
&{J}_1(\mathbf{k},\tau):= f \cos (\tau \mathcal{A}(\mathbf{k})^{1/2}) f^{-1} -
 f_0\cos (\tau \mathcal{A}^0(\mathbf{k})^{1/2})f_0^{-1},
 \\
& {J}_2(\mathbf{k},\tau):= f \mathcal{A}(\mathbf{k})^{-1/2} \sin (\tau \mathcal{A}(\mathbf{k})^{1/2}) f^* -
 f_0 \mathcal{A}^0(\mathbf{k})^{-1/2} \sin (\tau \mathcal{A}^0(\mathbf{k})^{1/2})f_0.
\end{align*}

We will apply theorems of Section~\ref{abstr_sandwiched_section} to the operator
 ${A}(t, \boldsymbol{\theta}) = {\mathcal{A}}(\mathbf{k}) = f^* \widehat{\mathcal{A}}(\mathbf{k}) f$.
 Due to Remark~\ref{rem_coeff}, we may trace the dependence of the constants in estimates on the problem data.
 Note that ${c}_*$, ${\delta}$, ${t}^{\,0}$
 do not depend on $\boldsymbol{\theta}$ (see \eqref{c*}, \eqref{delta_fixation}, \eqref{t0_fixation}).
 By \eqref{X_1_estimate}, the norm $\| {X}_1 (\boldsymbol{\theta}) \|$ can be replaced by $\alpha_1^{1/2} \| g \|_{L_{\infty}}^{1/2} \| f \|_{L_{\infty}}$.
 Hence, the constants in Theorems~\ref{abstr_cos_sandwiched_general_thrm} and~\ref{abstr_cos_sandwiched_enchanced_thrm_1} (as applied to the operator ${\mathcal{A}}(\mathbf{k})$)
 will be independent of $\boldsymbol{\theta}$. They depend only on $\alpha_0$, $\alpha_1$, $\|g\|_{L_\infty}$, $\|g^{-1}\|_{L_\infty}$,
 $\|f\|_{L_\infty}$, $\|f^{-1}\|_{L_\infty}$, and $r_0$.

\begin{theorem}
    \label{sndw_cos_general_thrm}
 For $\tau \in \mathbb{R}$, $\varepsilon > 0$, and $\mathbf{k} \in \widetilde{\Omega}$ we have
    \begin{align*}
    &\|  {J}_1(\mathbf{k},\varepsilon^{-1}\tau)
\mathcal{R}(\mathbf{k}, \varepsilon)\|_{L_2(\Omega) \to L_2 (\Omega) }  \le \mathcal{C}_1(1 + |\tau|) \varepsilon,
\\
    &\|  {J}_2(\mathbf{k},\varepsilon^{-1}\tau)
\mathcal{R}(\mathbf{k}, \varepsilon)^{1/2}\|_{L_2(\Omega) \to L_2 (\Omega) }  \le \mathcal{C}_2(1 + |\tau|).
    \end{align*}
The constants ${\mathcal{C}}_1$ and ${\mathcal{C}}_2$ depend only on
$\alpha_0$, $\alpha_1$, $\|g\|_{L_\infty}$, $\|g^{-1}\|_{L_\infty}$, $\|f\|_{L_\infty}$, $\|f^{-1}\|_{L_\infty}$, and $r_0$.
\end{theorem}

 Theorem \ref{sndw_cos_general_thrm} is deduced from Theorem~\ref{abstr_cos_sandwiched_general_thrm} and relations \eqref{R_P}--\eqref{R(k,eps)(I-P)_est}, \eqref{f_0 hatS f_0 P = A^0}.
 This result was known before (see \cite[Theorem 8.2]{BSu5} and \cite[Subsection~8.2]{M}).

\subsection{The case where $\widehat{N}_Q(\boldsymbol{\theta}) = 0$}
Now we assume that $\widehat{N}_Q(\boldsymbol{\theta}) \equiv  0$ and apply
Theorem~\ref{abstr_cos_sandwiched_enchanced_thrm_1}.

\begin{theorem}
    \label{sndw_cos_enchanced_thrm_1}
 Let $\widehat{N}_Q(\boldsymbol{\theta})$ be the operator defined by~\emph{\eqref{N_Q(theta)}}.
 Suppose that  $\widehat{N}_Q(\boldsymbol{\theta})=0$ for any $\boldsymbol{\theta} \in \mathbb{S}^{d-1}$.
Then for $\tau \in \mathbb{R}$, $\varepsilon > 0$, and $\mathbf{k} \in \widetilde{\Omega}$ we have
    \begin{align}
    \label{9.5}
        \|   {J}_1(\mathbf{k},\varepsilon^{-1}\tau)
\mathcal{R}(\mathbf{k}, \varepsilon)^{3/4}\|_{L_2(\Omega) \to L_2 (\Omega) }  &\le \mathcal{C}_3 (1 + |\tau|) \varepsilon,
\\
\label{9.6}
        \|   {J}_2(\mathbf{k},\varepsilon^{-1}\tau)
\mathcal{R}(\mathbf{k}, \varepsilon)^{1/4}\|_{L_2(\Omega) \to L_2 (\Omega) }  &\le \mathcal{C}_4 (1 + |\tau|).
    \end{align}
The constants ${\mathcal{C}}_3$ and ${\mathcal{C}}_4$ depend only on
$\alpha_0$, $\alpha_1$, $\|g\|_{L_\infty}$, $\|g^{-1}\|_{L_\infty}$, $\|f\|_{L_\infty}$, $\|f^{-1}\|_{L_\infty}$, and $r_0$.
\end{theorem}

\begin{proof}
Estimate \eqref{9.5} follows from \eqref{4.19a}, \eqref{R_P}--\eqref{R(k,eps)(I-P)_est}, and \eqref{f_0 hatS f_0 P = A^0};
cf. the proof of inequality \eqref{th8.2est1}.

Let us check \eqref{9.6}. Estimate \eqref{4.19b} implies the required inequality for the term
$\|   {J}_2(\mathbf{k},\varepsilon^{-1}\tau) \mathcal{R}(\mathbf{k}, \varepsilon)^{1/4} \widehat{P}\|$ for $|{\mathbf k}|\le t^0$.

Now, we consider the term $\|  {J}_2(\mathbf{k},\varepsilon^{-1}\tau) \mathcal{R}(\mathbf{k}, \varepsilon)^{1/4} (I - \widehat{P})\|$.
We need to estimate the operators $\mathcal{A}(\mathbf{k})^{-1/2} f^* (I - \widehat{P})$ and
$\mathcal{A}^0(\mathbf{k})^{-1/2} (I - \widehat{P})$. Obviously, the second one is uniformly bounded.
To estimate the first one, note that $P f^* (I- \widehat{P})=0$ due to the identity
$P f^* = f^{-1}(\overline{Q})^{-1} \widehat{P}$ (see \eqref{abstr_P_and_P_hat_relation}).
Here $P$ is the orthogonal projection onto $\mathfrak N$ (see \eqref{Ker2}). Consequently, we have
$\mathcal{A}(\mathbf{k})^{-1/2} f^* (I - \widehat{P})= \mathcal{A}(\mathbf{k})^{-1/2} (I-P) f^* (I - \widehat{P})$.
By \eqref{abstr_F(t)_threshold_1} and \eqref{A(k)_nondegenerated_and_c_*}, this operator is uniformly bounded for $|{\mathbf k}| \le t^0$.
It follows that the term $\|  {J}_2(\mathbf{k},\varepsilon^{-1}\tau) \mathcal{R}(\mathbf{k}, \varepsilon)^{1/4} (I - \widehat{P})\|$
is uniformly bounded for $|{\mathbf k}| \le t^0$.

For $|{\mathbf k}| > t^0$ estimate \eqref{9.6} is a consequence of \eqref{A(k)_nondegenerated_and_c_*}
and a similar estimate for $\mathcal{A}^0(\mathbf{k})$.
\end{proof}

\subsection{The case where $\widehat{N}_{0,Q}(\boldsymbol{\theta}) = 0$}
Now we reject the assumption of Theorem~\ref{sndw_cos_enchanced_thrm_1}, but instead we assume that $\widehat{N}_{0,Q}(\boldsymbol{\theta}) = 0$ for any $\boldsymbol{\theta}$.
As in Subsection~\ref{ench_approx2_section}, in order to apply Theorem~\ref{abstr_cos_sandwiched_enchanced_thrm_2}, we have to impose an additional condition. We use the initial enumeration of the eigenvalues
$\gamma_1 (\boldsymbol{\theta})\le \ldots \le \gamma_n (\boldsymbol{\theta})$ of the germ $S (\boldsymbol{\theta})$.
They  are also the eigenvalues of the generalized spectral problem~\eqref{hatS_gener_spec_problem}. For each $\boldsymbol{\theta}$, let $\mathcal{P}^{(k)} (\boldsymbol{\theta})$ be the
\textquotedblleft skew\textquotedblright \ projection (orthogonal with the weight $\overline{Q}$) of $L_2 (\Omega; \mathbb{C}^n)$
 onto the eigenspace of the problem~\eqref{hatS_gener_spec_problem} corresponding to the eigenvalue
 $\gamma_k (\boldsymbol{\theta})$. Clearly,  $\mathcal{P}^{(k)} (\boldsymbol{\theta})$ coincides with one of the projections
$\mathcal{P}_j (\boldsymbol{\theta})$   introduced in Subsection~\ref{sndw_eigenval_multipl_section} (but the number $j$ may depend on $\boldsymbol{\theta}$).

\begin{condition}
    \label{sndw_cond1}

$1^\circ$. $\widehat{N}_{0,Q}(\boldsymbol{\theta})=0$ for any $\boldsymbol{\theta} \in \mathbb{S}^{d-1}$.
 $2^\circ$. For any pair of indices $(k,l)$, $1\le k,l \le n$, $k\ne l$, such that
$\gamma_k (\boldsymbol{\theta}_0) = \gamma_l (\boldsymbol{\theta}_0) $ for some $\boldsymbol{\theta}_0 \in \mathbb{S}^{d-1}$, we have  $(\mathcal{P}^{(k)} (\boldsymbol{\theta}))^* \widehat{N}_Q (\boldsymbol{\theta}) \mathcal{P}^{(l)} (\boldsymbol{\theta}) = 0$ for any $\boldsymbol{\theta} \in \mathbb{S}^{d-1}$.
\end{condition}

   Condition~$2^\circ$ can be reformulated as follows: it is assumed that, for the \textquotedblleft blocks\textquotedblright \ $(\mathcal{P}^{(k)} (\boldsymbol{\theta}))^* \widehat{N}_Q (\boldsymbol{\theta}) \mathcal{P}^{(l)} (\boldsymbol{\theta})$ of the operator $\widehat{N}_Q (\boldsymbol{\theta})$ that are not identically zero,
    the corresponding branches of the eigenvalues $\gamma_k (\boldsymbol{\theta})$ and  $\gamma_l (\boldsymbol{\theta})$ do not intersect.

 Condition~\ref{sndw_cond1} is ensured by the following more restrictive condition.

\begin{condition}
    \label{sndw_cond2}

        $1^\circ$. $\widehat{N}_{0,Q}(\boldsymbol{\theta})=0$ for any  $\boldsymbol{\theta} \in \mathbb{S}^{d-1}$.
$2^\circ$. Suppose that the number $p$ of different eigenvalues of the generalized spectral problem~\emph{\eqref{hatS_gener_spec_problem}} does not depend on $\boldsymbol{\theta} \in \mathbb{S}^{d-1}$.
    \end{condition}

Under Condition~\ref{sndw_cond2},
denote the different eigenvalues of the germ enumerated in the increasing order by
$\gamma^{\circ}_1(\boldsymbol{\theta}), \ldots, \gamma^{\circ}_p(\boldsymbol{\theta})$.
Then their multiplicities $k_1, \ldots, k_p$ do not depend on $\boldsymbol{\theta} \in \mathbb{S}^{d-1}$.

\begin{remark}
    \label{sndw_simple_spec_remark}
 $1^\circ$.   Assumption $2^\circ$ of Condition~\emph{\ref{sndw_cond2}} is a fortiori valid if the spectrum of the problem~\emph{\eqref{hatS_gener_spec_problem}} is simple for any $\boldsymbol{\theta} \in \mathbb{S}^{d-1}$.
$2^\circ$. From Corollary {\rm \ref{sndw_simple_spec_N0Q=0_cor}} it follows that Condition~\emph{\ref{sndw_cond2}} is satisfied,
 if $b(\boldsymbol{\theta})$, $g(\mathbf{x})$, and $Q(\mathbf{x})$
have real entries, and the spectrum of problem \emph{\eqref{hatS_gener_spec_problem}}
is simple for any  $\boldsymbol{\theta} \in \mathbb{S}^{d-1}$.
\end{remark}

Under Condition~\ref{sndw_cond1}, let
$    \mathcal{K} := \{ (k,l) \colon 1 \le k,l \le n, \; k \ne l, \;  (\mathcal{P}^{(k)} (\boldsymbol{\theta}))^* \widehat{N}_Q (\boldsymbol{\theta}) \mathcal{P}^{(l)} (\boldsymbol{\theta}) \not\equiv 0 \}$.
Denote
    $c^{\circ}_{kl} (\boldsymbol{\theta}) := \min \{c_*, n^{-1} |\gamma_k (\boldsymbol{\theta}) - \gamma_l (\boldsymbol{\theta})| \}$,
$(k,l) \in \mathcal{K}$.
Since  $S (\boldsymbol{\theta})$  is continuous with respect to \hbox{$\boldsymbol{\theta} \in \mathbb{S}^{d-1}$},
then $\gamma_j (\boldsymbol{\theta})$~are continuous functions on $\mathbb{S}^{d-1}$.
By Condition~\ref{sndw_cond1}($2^\circ$), for $(k,l) \in \mathcal{K}$ we have \hbox{$|\gamma_k (\boldsymbol{\theta}) - \gamma_l(\boldsymbol{\theta})| > 0$} for any $\boldsymbol{\theta} \in \mathbb{S}^{d-1}$, whence
$ c^{\circ}_{kl} := \min_{\boldsymbol{\theta} \in \mathbb{S}^{d-1}} c^{\circ}_{kl} (\boldsymbol{\theta}) > 0$ for $(k,l) \in \mathcal{K}$.
We put
\begin{equation}
    \label{c^circ}
    c^{\circ} := \min_{(k,l) \in \mathcal{K}} c^{\circ}_{kl}.
\end{equation}
Clearly, the number~\eqref{c^circ}~ plays the role of the number~\eqref{abstr_c^circ}
chosen  independently of $\boldsymbol{\theta}$.
The number $t^{00}$ subject to~\eqref{abstr_t00}
 also can be chosen independently of $\boldsymbol{\theta} \in \mathbb{S}^{d-1}$.
Taking~\eqref{delta_fixation} and~\eqref{X_1_estimate} into account, we put
\begin{equation*}
    t^{00} = (8 \beta_2)^{-1} r_0 \alpha_1^{-3/2} \alpha_0^{1/2} \| g\|_{L_{\infty}}^{-3/2} \| g^{-1}\|_{L_{\infty}}^{-1/2} \|f\|_{L_\infty}^{-3} \|f^{-1}\|_{L_\infty}^{-1} c^{\circ}.
\end{equation*}
(The condition $t^{00} \le t^{0}$ is valid, since $c^{\circ} \le \| S (\boldsymbol{\theta}) \| \le \alpha_1 \|g\|_{L_{\infty}}\|f\|_{L_\infty}^2$.)

Under Condition~\ref{sndw_cond1}, we apply Theorem~\ref{abstr_cos_sandwiched_enchanced_thrm_2}
and deduce the following result.
We have to take into account that now (see Remark~\ref{rem_coeff}) the constants
will depend not only on $\alpha_0$, $\alpha_1$, $\|g\|_{L_\infty}$, $\|g^{-1}\|_{L_\infty}$, $\|f\|_{L_\infty}$, $\|f^{-1}\|_{L_\infty}$,
and $r_0$, but also on ${c}^{\circ}$ and $n$.

\begin{theorem}
    \label{sndw_cos_enchanced_thrm_2}
   Suppose that Condition~\emph{\ref{sndw_cond1}} \emph{(}or more restrictive Condition~\emph{\ref{sndw_cond2}}\emph{)}
is satisfied. Then for $\tau \in \mathbb{R}$, $\varepsilon > 0$, and $\mathbf{k} \in \widetilde{\Omega}$ we have
    \begin{align*}
    &\|  {J}_1(\mathbf{k},\varepsilon^{-1}\tau)
 \mathcal{R}(\mathbf{k}, \varepsilon)^{3/4}\|_{L_2(\Omega) \to L_2 (\Omega) }  \le \mathcal{C}_5(1 + |\tau|) \varepsilon,
 \\
    &\|  {J}_2(\mathbf{k},\varepsilon^{-1}\tau)
 \mathcal{R}(\mathbf{k}, \varepsilon)^{1/4}\|_{L_2(\Omega) \to L_2 (\Omega) }  \le \mathcal{C}_{6}(1 + |\tau|).
    \end{align*}
The constants ${\mathcal{C}}_{5}$ and ${\mathcal{C}}_{6}$ depend only on
$\alpha_0$, $\alpha_1$, $\|g\|_{L_\infty}$, $\|g^{-1}\|_{L_\infty}$, $\|f\|_{L_\infty}$, $\|f^{-1}\|_{L_\infty}$, $r_0$, $n$, and $c^\circ$.
\end{theorem}

\subsection{ The sharpness of the result in the general case}

Application of Theorem~\ref{abstr_sndwchd_s<2_general_thrm}  allows us to confirm the sharpness of the result of Theorem~\ref{sndw_cos_general_thrm}.

\begin{theorem}
    \label{s<2_thrm}
 Let $\widehat{N}_{0,Q} (\boldsymbol{\theta})$ be the operator defined by~\emph{\eqref{N0Q_invar_repr}}.
Suppose that $\widehat{N}_{0,Q} (\boldsymbol{\theta}_0) \ne 0$ for some $\boldsymbol{\theta}_0 \in \mathbb{S}^{d-1}$.

\noindent
$1^\circ$. Let $0 \ne \tau \in \mathbb{R}$ and $0 \le s < 2$. Then there does not exist a constant $\mathcal{C} (\tau) > 0$ such that
the estimate
    \begin{equation}
        \label{sndw_s<2_est_imp}
        \| J_1(\mathbf{k},\varepsilon^{-1}\tau)
\mathcal{R}(\mathbf{k}, \varepsilon)^{s/2}\|_{L_2(\Omega) \to L_2 (\Omega) }  \le \mathcal{C} (\tau) \varepsilon
    \end{equation}
   holds for almost every $\mathbf{k} = t \boldsymbol{\theta} \in \widetilde{\Omega}$ and sufficiently small $\varepsilon > 0$.

\noindent
$2^\circ$. Let $0 \ne \tau \in \mathbb{R}$ and $0 \le s < 1$. Then there does not exist a constant $\mathcal{C} (\tau) > 0$ such that
the estimate
    \begin{equation}
\label{9.4a}
        \| {J}_2(\mathbf{k},\varepsilon^{-1}\tau)
\mathcal{R}(\mathbf{k}, \varepsilon)^{s/2}\|_{L_2(\Omega) \to L_2 (\Omega) }  \le \mathcal{C} (\tau)
   \end{equation}
   holds for almost every $\mathbf{k} = t \boldsymbol{\theta} \in \widetilde{\Omega}$ and sufficiently small $\varepsilon > 0$.
\end{theorem}

The following lemma is an analog of  Lemma~\ref{Lipschitz_lemma}. We omit the proof.

\begin{lemma}
    \label{sndw_Lipschitz_lemma}
 Let $\delta$ and $t^0$ be given by~\emph{\eqref{delta_fixation}} and~\emph{\eqref{t0_fixation}}, respectively.
Let $F (\mathbf{k}) = F (t, \boldsymbol{\theta})$~be the spectral projection of the operator $\mathcal{A}(\mathbf{k})$ for the interval $[0, \delta]$. Then for $|\mathbf{k}| \le t^0$ and $|\mathbf{k}_0| \le t^0$ we have
    \begin{gather*}
          \| \cos(\tau \mathcal{A}(\mathbf{k})^{1/2}) F (\mathbf{k}) - \cos(\tau \mathcal{A}(\mathbf{k}_0)^{1/2}) F (\mathbf{k}_0)\|_{L_2(\Omega) \to L_2 (\Omega) }
\le C'(\tau) | \mathbf{k} - \mathbf{k}_0|,
\\
          \| \mathcal{A}(\mathbf{k})^{-1/2} \sin(\tau \mathcal{A}(\mathbf{k})^{1/2}) F (\mathbf{k}) -
          \mathcal{A}(\mathbf{k}_0)^{-1/2} \sin(\tau \mathcal{A}(\mathbf{k}_0)^{1/2}) F (\mathbf{k}_0)\|_{L_2(\Omega) \to L_2 (\Omega) }
\le C''(\tau) | \mathbf{k} - \mathbf{k}_0|.
    \end{gather*}
\end{lemma}

\noindent \textbf{Proof of Theorem~\ref{s<2_thrm}}.
Let us check assertion $1^\circ$.
It suffices to consider $1 \le s < 2$. We prove by contradiction. Fix $\tau \ne 0$.  Suppose that for some $1 \le s < 2$ there exists a constant   $\mathcal{C}(\tau) > 0$  such that estimate~\eqref{sndw_s<2_est_imp} holds for almost every $\mathbf{k} \in \widetilde{\Omega}$ and sufficiently small $\varepsilon > 0$. Then, by~\eqref{R_P},
\begin{equation}
    \label{s<2_proof_f1}
    \|  {J}_1(\mathbf{k},\varepsilon^{-1}\tau)
\widehat{P}\|_{L_2(\Omega) \to L_2 (\Omega) } \varepsilon^s (|\mathbf{k}|^2 + \varepsilon^2)^{-s/2}  \le {\mathcal{C}}(\tau) \varepsilon
\end{equation}
for almost every  $\mathbf{k} \in \widetilde{\Omega}$  and sufficiently small $\varepsilon > 0$.

By~\eqref{abstr_P_and_P_hat_relation}, we have $f^{-1} \widehat{P} = P f^* \overline{Q}$.
Then the operator under the norm sign in~\eqref{s<2_proof_f1} can be written as  $ f \cos (\varepsilon^{-1} \tau {\mathcal{A}}(\mathbf{k})^{1/2}) P f^* \overline{Q} -
f_0 \cos (\varepsilon^{-1} \tau {\mathcal{A}}^0(\mathbf{k})^{1/2}) f_0^{-1} \widehat{P}$.

Now, let $|\mathbf{k}| \le t^0$. By \eqref{abstr_F(t)_threshold_1},
\begin{equation}
    \label{s<2_proof_f2}
    \| F (\mathbf{k}) - P \|_{L_2 (\Omega) \to L_2 (\Omega)} \le C_1 |\mathbf{k}|, \quad |\mathbf{k}| \le t^0.
\end{equation}
By~\eqref{s<2_proof_f1} and~\eqref{s<2_proof_f2},
 there exists a constant $\widetilde{\mathcal C}(\tau)$ such that
\begin{equation}
    \label{s<2_proof_f3}
    \|  f \cos (\varepsilon^{-1} \tau \mathcal{A}(\mathbf{k})^{1/2}) F (\mathbf{k}) f^* \overline{Q} - f_0 \cos (\varepsilon^{-1} \tau \mathcal{A}^0(\mathbf{k})^{1/2}) f_0^{-1} \widehat{P}\|_{L_2(\Omega) \to L_2 (\Omega) }
\varepsilon^s (|\mathbf{k}|^2 + \varepsilon^2)^{-s/2}  \le \widetilde{\mathcal{C}}(\tau) \varepsilon
\end{equation}
for almost every $\mathbf{k} $ in the ball $|\mathbf{k}| \le t^0$ and sufficiently small $\varepsilon > 0$.

Observe that $\widehat{P}$  is the spectral projection of the operator $\mathcal{A}^0 (\mathbf{k})$ for the interval $[0, \delta]$.
Therefore, Lemma~\ref{sndw_Lipschitz_lemma} (applied to $\mathcal{A} (\mathbf{k})$ and $\mathcal{A}^0 (\mathbf{k})$)
 implies that for fixed $\tau$ and $\varepsilon$  the operator under the norm sign in~\eqref{s<2_proof_f3}
 is continuous with respect to $\mathbf{k}$ in the ball $|\mathbf{k}| \le t^0$. Hence, estimate~\eqref{s<2_proof_f3} holds for all $\mathbf{k}$
in this ball, in particular, for $\mathbf{k} = t\boldsymbol{\theta}_0$ if $t \le t^0$.
Applying once more the inequality~\eqref{s<2_proof_f2} and the identity $P f^* \overline{Q} = f^{-1} \widehat{P}$,
we see that
\begin{equation}
    \label{s<2_proof_f4}
    \| {J}_1(t\boldsymbol{\theta}_0,\varepsilon^{-1}\tau)
  \widehat{P}\|_{L_2(\Omega) \to L_2 (\Omega) } \varepsilon^s (t^2 + \varepsilon^2)^{-s/2}  \le \check{\mathcal{C}}(\tau) \varepsilon
\end{equation}
for all $t \le t^0$  and sufficiently small $\varepsilon$.
In abstract terms, estimate~\eqref{s<2_proof_f4} corresponds to the inequality~\eqref{abstr_sndwchd_s<2_est_imp}.
Since $\widehat{N}_{0,Q} (\boldsymbol{\theta}_0) \ne 0$, applying Theorem~\ref{abstr_sndwchd_s<2_general_thrm}($1^\circ$),
 we arrive at a contradiction.

Let us check assertion $2^\circ$.
 Suppose that for some $\tau \ne 0$ and $1 \le s < 1$ there exists a constant   $\mathcal{C}(\tau) > 0$  such that estimate~\eqref{9.4a} holds for almost every $\mathbf{k} \in \widetilde{\Omega}$ and sufficiently small $\varepsilon > 0$. Then
\begin{equation}
    \label{s<1_proof_f1}
    \|  {J}_2(\mathbf{k},\varepsilon^{-1}\tau)
\widehat{P}\|_{L_2(\Omega) \to L_2 (\Omega) } \varepsilon^s (|\mathbf{k}|^2 + \varepsilon^2)^{-s/2}  \le {\mathcal{C}}(\tau) 
\end{equation}
for almost every  $\mathbf{k} \in \widetilde{\Omega}$  and sufficiently small $\varepsilon > 0$.
Now, let $|\mathbf{k}| \le t^0$. Note that 
\begin{equation}
    \label{s<1_proof_f2}
    \|\mathcal{A} (\mathbf{k})^{-1/2} F(\mathbf{k})^\perp  \|_{L_2(\Omega) \to L_2 (\Omega) }   \le \delta^{-1/2}.  
\end{equation}
Hence, there exists a constant $\widetilde{\mathcal{C}}(\tau) >0$ such that 
\begin{equation}
    \label{s<1_proof_f3}
\begin{split}    
\|  f \mathcal{A}(\mathbf{k})^{-1/2} \sin (\varepsilon^{-1} \tau \mathcal{A}(\mathbf{k})^{1/2}) F (\mathbf{k}) f^* \widehat{P}  
- f_0 \mathcal{A}^0(\mathbf{k})^{-1/2} \sin (\varepsilon^{-1} \tau \mathcal{A}^0(\mathbf{k})^{1/2}) f_0 \widehat{P}\|_{L_2(\Omega) \to L_2 (\Omega) } 
\\
\times \varepsilon^s (|\mathbf{k}|^2 + \varepsilon^2)^{-s/2}  \le \widetilde{\mathcal{C}}(\tau) 
\end{split}
\end{equation}
for almost every $\mathbf{k} $ in the ball $|\mathbf{k}| \le t^0$ and sufficiently small $\varepsilon > 0$.

By  Lemma~\ref{sndw_Lipschitz_lemma} (applied to $\mathcal{A} (\mathbf{k})$ and $\mathcal{A}^0 (\mathbf{k})$),
the operator  under the norm sign in~\eqref{s<1_proof_f3}
 is continuous in $\mathbf{k}$ in the ball $|\mathbf{k}| \le t^0$. Hence, estimate~\eqref{s<1_proof_f3} holds for all $\mathbf{k}$
in this ball, in particular, for $\mathbf{k} = t\boldsymbol{\theta}_0$ with $t \le t^0$.
Applying~\eqref{s<1_proof_f2} again, we see that
\begin{equation}
    \label{s<1_proof_f4}
    \| {J}_2(t\boldsymbol{\theta}_0,\varepsilon^{-1}\tau)
  \widehat{P}\|_{L_2(\Omega) \to L_2 (\Omega) } \varepsilon^s (t^2 + \varepsilon^2)^{-s/2}  \le \check{\mathcal{C}}(\tau) 
\end{equation}
for all $t \le t^0$  and sufficiently small $\varepsilon$.
In abstract terms, estimate~\eqref{s<1_proof_f4} corresponds to the inequality~\eqref{4.23}.
Since $\widehat{N}_{0,Q} (\boldsymbol{\theta}_0) \ne 0$, applying Theorem~\ref{abstr_sndwchd_s<2_general_thrm}($2^\circ$),
 we arrive at a contradiction.
 $\square$

\section{Approximation of the operators $\cos(\varepsilon^{-1} \tau \mathcal{A}^{1/2})$ and ${\mathcal{A}}^{-1/2} \sin(\varepsilon^{-1}\tau {\mathcal{A}}^{1/2})$}

\subsection{Approximation of the operators $\cos(\varepsilon^{-1} \tau \widehat{\mathcal{A}}^{1/2})$
and $\widehat{\mathcal{A}}^{-1/2} \sin(\varepsilon^{-1}\tau \widehat{\mathcal{A}}^{1/2})$}

In $L_2 (\mathbb{R}^d; \mathbb{C}^n)$, we consider the operator~\eqref{hatA}. Let $\widehat{\mathcal{A}}^0$~ be the effective operator~\eqref{hatA0}. Denote
\begin{align*}
\widehat{J}_1(\tau) &:= \cos(\tau \widehat{\mathcal{A}}^{1/2}) - \cos(\tau (\widehat{\mathcal{A}}^0)^{1/2}),
\\
\widehat{J}_2(\tau) &:= \widehat{\mathcal{A}}^{-1/2} \sin(\tau \widehat{\mathcal{A}}^{1/2}) - (\widehat{\mathcal{A}}^0)^{-1/2} \sin(\tau (\widehat{\mathcal{A}}^0)^{1/2}).
\end{align*}

Recall the notation $\mathcal{H}_0 = - \Delta$ and put
\begin{equation}
\label{R(epsilon)}
\mathcal{R} (\varepsilon) := \varepsilon^2 (\mathcal{H}_0 + \varepsilon^2 I)^{-1}.
\end{equation}
The operator $\mathcal{R} (\varepsilon)$  expands in the direct integral of the operators~\eqref{R(k, epsilon)}:
\begin{equation*}
\mathcal{R} (\varepsilon) = \mathcal{U}^{-1} \left( \int_{\widetilde{\Omega}} \oplus  \mathcal{R} (\mathbf{k}, \varepsilon) \, d \mathbf{k}  \right) \mathcal{U}.
\end{equation*}

Recall the notation \eqref{J(k,tau)} and \eqref{J2(k,tau)}.
By expansion~\eqref{decompose} for $\widehat{\mathcal{A}}$ and $\widehat{\mathcal{A}}^0$, we see that the operator
$\widehat{J}_l(\varepsilon^{-1} \tau) \mathcal{R} (\varepsilon)^{s/2}$ expands in the direct integral of the operators
$\widehat{J}_l ({\mathbf k}, \varepsilon^{-1} \tau) {\mathcal R} ({\mathbf k},\varepsilon)^{s/2}$, $l=1,2$.
 Hence,
\begin{equation}
\label{norms_and_Gelfand_transf}
\| \widehat{J}_l (\varepsilon^{-1} \tau) {\mathcal R} (\varepsilon)^{s/2} \|_{L_2(\mathbb{R}^d) \to L_2(\mathbb{R}^d)}
 = \underset{\mathbf{k} \in \widetilde{\Omega}}{\esssup}
\| \widehat{J}_l ({\mathbf k}, \varepsilon^{-1} \tau)
\mathcal{R} (\mathbf{k},\varepsilon)^{s/2} \|_{L_2(\Omega) \to L_2(\Omega)},\quad l=1,2.
\end{equation}
Therefore, we deduce the following results from Theorems~\ref{cos_general_thrm},~\ref{cos_enchanced_thrm_11}, and~\ref{cos_enchanced_thrm_2}.

\begin{theorem}
    \label{cos_thrm_1}
    For $\tau \in \mathbb{R}$ and $\varepsilon > 0$ we have
    \begin{align*}
    &\| \widehat{J}_1 (\varepsilon^{-1} \tau) {\mathcal R} (\varepsilon)
    \|_{L_2(\mathbb{R}^d) \to L_2(\mathbb{R}^d)} \le \widehat{\mathcal{C}}_1(1 + |\tau|) \varepsilon,
    \\
    &\| \widehat{J}_2 (\varepsilon^{-1} \tau) {\mathcal R} (\varepsilon)^{1/2}
    \|_{L_2(\mathbb{R}^d) \to L_2(\mathbb{R}^d)} \le \widehat{\mathcal{C}}_2(1 + |\tau|).
    \end{align*}
    The constants $\widehat{\mathcal{C}}_1$ and $\widehat{\mathcal{C}}_2$
    depend only on $\alpha_0$, $\alpha_1$, $\|g\|_{L_\infty}$, $\|g^{-1}\|_{L_\infty}$, and $r_0$.
\end{theorem}

\begin{theorem}
     \label{cos_thrm_2}
Let $\widehat{N}(\boldsymbol{\theta})$
be the operator defined by~\emph{\eqref{N(theta)}}. Suppose that $\widehat{N}(\boldsymbol{\theta}) = 0$ for any $\boldsymbol{\theta} \in \mathbb{S}^{d-1}$. Then for $\tau \in \mathbb{R}$ and $\varepsilon > 0$ we have
    \begin{align*}
    &\| \widehat{J}_1 (\varepsilon^{-1} \tau) {\mathcal R} (\varepsilon)^{3/4}
    \|_{L_2(\mathbb{R}^d) \to L_2(\mathbb{R}^d)} \le \widehat{\mathcal{C}}_3(1 + |\tau|) \varepsilon,
    \\
    &\| \widehat{J}_2 (\varepsilon^{-1} \tau) {\mathcal R} (\varepsilon)^{1/4}
    \|_{L_2(\mathbb{R}^d) \to L_2(\mathbb{R}^d)} \le \widehat{\mathcal{C}}_4(1 + |\tau|).
    \end{align*}
The constants $\widehat{\mathcal{C}}_3$ and $\widehat{\mathcal{C}}_4$
    depend only on $\alpha_0$, $\alpha_1$, $\|g\|_{L_\infty}$, $\|g^{-1}\|_{L_\infty}$, and $r_0$.
    \end{theorem}

\begin{theorem}
    \label{cos_thrm_3}
   Suppose  that Condition~\emph{\ref{cond9}} \emph{(}or more restrictive Condition~\emph{\ref{cond99}}\emph{)} is satisfied. Then for $\tau \in \mathbb{R}$ and $\varepsilon > 0$ we have
    \begin{align*}
    &\| \widehat{J}_1 (\varepsilon^{-1} \tau) {\mathcal R} (\varepsilon)^{3/4}
    \|_{L_2(\mathbb{R}^d) \to L_2(\mathbb{R}^d)} \le \widehat{\mathcal{C}}_5 (1 + |\tau|) \varepsilon,
    \\
    &\| \widehat{J}_2 (\varepsilon^{-1} \tau) {\mathcal R} (\varepsilon)^{1/4}
    \|_{L_2(\mathbb{R}^d) \to L_2(\mathbb{R}^d)} \le \widehat{\mathcal{C}}_{6}(1 + |\tau|).
    \end{align*}
The constants $\widehat{\mathcal{C}}_5$ and $\widehat{\mathcal{C}}_{6}$
    depend only on $\alpha_0$, $\alpha_1$, $\|g\|_{L_\infty}$, $\|g^{-1}\|_{L_\infty}$, $r_0$, $n$, and $\widehat{c}^\circ$.
    \end{theorem}

Theorem~\ref{cos_thrm_1} was known before (see \cite[Theorem 9.2]{BSu5} and \cite[Section 9]{M}).

Applying Theorem~\ref{hat_s<2_thrm},  we confirm the sharpness of the result of Theorem~\ref{cos_thrm_1}.

\begin{theorem}
     \label{s<2_cos_thrm}
Let $\widehat{N}_0 (\boldsymbol{\theta})$
be the operator defined by~\emph{(\ref{N0_invar_repr})}. Suppose that $\widehat{N}_0 (\boldsymbol{\theta}_0) \ne 0$ for some  $\boldsymbol{\theta}_0 \in \mathbb{S}^{d-1}$.

\noindent
$1^\circ$. Let $0 \ne \tau \in \mathbb{R}$ and $0 \le s < 2$.
Then there does not exist a constant $\mathcal{C}(\tau) > 0$ such that the estimate
    \begin{equation}
    \label{s<2_cos_est}
    \| \widehat{J}_1 (\varepsilon^{-1} \tau)\mathcal{R}(\varepsilon)^{s/2}\|_{L_2(\mathbb{R}^d) \to L_2 (\mathbb{R}^d) }  \le \mathcal{C}(\tau) \varepsilon
    \end{equation}
  holds for all sufficiently small $\varepsilon > 0$.

  \noindent
$2^\circ$. Let $0 \ne \tau \in \mathbb{R}$ and $0 \le s < 1$.
Then there does not exist a constant $\mathcal{C}(\tau) > 0$ such that the estimate
    \begin{equation*}
    \| \widehat{J}_2 (\varepsilon^{-1} \tau)\mathcal{R}(\varepsilon)^{s/2}\|_{L_2(\mathbb{R}^d) \to L_2 (\mathbb{R}^d) }  \le \mathcal{C}(\tau)
    \end{equation*}
  holds for all sufficiently small $\varepsilon > 0$.
\end{theorem}

\begin{proof}
Let us check assertion $1^\circ$. We prove by contradiction.
Suppose that for some $\tau \ne 0$ and $0 \le s < 2$  there exists a constant $\mathcal{C}(\tau) > 0$ such that \eqref{s<2_cos_est}
 holds for all sufficiently small $\varepsilon > 0$. By~\eqref{norms_and_Gelfand_transf},  this means that for almost every $\mathbf{k} \in \widetilde{\Omega}$ and sufficiently small $\varepsilon$ estimate~\eqref{hat_s<2_est_imp} holds.
But this contradicts the statement of Theorem~\ref{hat_s<2_thrm}($1^\circ$).

Assertion $2^\circ$ is deduced from Theorem~\ref{hat_s<2_thrm}($2^\circ$) similarly.
\end{proof}

\subsection{ Approximation of the sandwiched operators $\cos(\varepsilon^{-1} \tau \mathcal{A}^{1/2})$ and ${\mathcal{A}}^{-1/2} \sin(\varepsilon^{-1}\tau {\mathcal{A}}^{1/2})$}
In $L_2 (\mathbb{R}^d; \mathbb{C}^n)$, consider the operator~\eqref{A}.
 Let $f_0$~be the matrix~\eqref{f_0}, and let $\mathcal{A}^0$ be the operator~\eqref{A0}. Denote
\begin{align*}
{J}_1(\tau) &:= f \cos(\tau \mathcal{A}^{1/2}) f^{-1} - f_0 \cos(\tau (\mathcal{A}^0)^{1/2}) f_0^{-1},
\\
{J}_2(\tau) &:= f {\mathcal{A}}^{-1/2} \sin(\tau {\mathcal{A}}^{1/2}) f^* - f_0 ({\mathcal{A}}^0)^{-1/2} \sin(\tau ({\mathcal{A}}^0)^{1/2}) f_0.
\end{align*}

Similarly to \eqref{norms_and_Gelfand_transf}, by the direct integral expansions, we obtain
\begin{equation*}
\| {J}_l (\varepsilon^{-1} \tau) {\mathcal R} (\varepsilon)^{s/2} \|_{L_2(\mathbb{R}^d) \to L_2(\mathbb{R}^d)}
 = \underset{\mathbf{k} \in \widetilde{\Omega}}{\esssup}
\| {J}_l ({\mathbf k}, \varepsilon^{-1} \tau)
\mathcal{R} (\mathbf{k},\varepsilon)^{s/2} \|_{L_2(\Omega) \to L_2(\Omega)},\quad l=1,2.
\end{equation*}

Therefore, Theorems~\ref{sndw_cos_general_thrm},~\ref{sndw_cos_enchanced_thrm_1}, and~\ref{sndw_cos_enchanced_thrm_2}
imply the following results.

\begin{theorem}
    \label{sndw_cos_thrm_1}
    For $\tau \in \mathbb{R}$ and $\varepsilon > 0$ we have
  \begin{align*}
    &\|   {J}_1 (\varepsilon^{-1} \tau) \mathcal{R} (\varepsilon) \|_{L_2(\mathbb{R}^d) \to L_2(\mathbb{R}^d)}
 \le \mathcal{C}_1 (1+ |\tau|) \varepsilon,
 \\
    &\| {J}_2 (\varepsilon^{-1} \tau) \mathcal{R} (\varepsilon)^{1/2} \|_{L_2(\mathbb{R}^d) \to L_2(\mathbb{R}^d)}
 \le \mathcal{C}_2(1 + |\tau|).
    \end{align*}
  The constants $\mathcal{C}_1$ and $\mathcal{C}_2$
depend only on $\alpha_0$, $\alpha_1$,  $\|g\|_{L_\infty}$, $\|g^{-1}\|_{L_\infty}$,  $\|f\|_{L_\infty}$, $\|f^{-1}\|_{L_\infty}$, and $r_0$.
\end{theorem}

\begin{theorem}
   \label{sndw_cos_thrm_2}
Let $\widehat{N}_Q(\boldsymbol{\theta})$
be the operator defined by~\emph{\eqref{N_Q(theta)}}. Suppose that $\widehat{N}_Q(\boldsymbol{\theta}) = 0$ for any
$\boldsymbol{\theta} \in \mathbb{S}^{d-1}$. Then for $\tau \in \mathbb{R}$ and $\varepsilon > 0$ we have
\begin{align*}
    &\|   {J}_1 (\varepsilon^{-1} \tau) \mathcal{R} (\varepsilon)^{3/4} \|_{L_2(\mathbb{R}^d) \to L_2(\mathbb{R}^d)}
 \le \mathcal{C}_3 (1+  |\tau|) \varepsilon,
 \\
    &\|   {J}_2 (\varepsilon^{-1} \tau) \mathcal{R} (\varepsilon)^{1/4} \|_{L_2(\mathbb{R}^d) \to L_2(\mathbb{R}^d)}
 \le \mathcal{C}_4 (1 + |\tau|).
    \end{align*}
  The constants $\mathcal{C}_3$ and $\mathcal{C}_4$
depend only on $\alpha_0$, $\alpha_1$,  $\|g\|_{L_\infty}$, $\|g^{-1}\|_{L_\infty}$,  $\|f\|_{L_\infty}$, $\|f^{-1}\|_{L_\infty}$, and $r_0$.
\end{theorem}

\begin{theorem}
    \label{sndw_cos_thrm_3}
Suppose that Condition~\emph{\ref{sndw_cond1}} \emph{(}or more restrictive Condition~\emph{\ref{sndw_cond2}}\emph{)} is satisfied. Then for $\tau \in \mathbb{R}$ and $\varepsilon > 0$ we have
\begin{align*}
    &\|   {J}_1 (\varepsilon^{-1} \tau) \mathcal{R} (\varepsilon)^{3/4} \|_{L_2(\mathbb{R}^d) \to L_2(\mathbb{R}^d)}
 \le \mathcal{C}_5(1 +  |\tau|) \varepsilon,
 \\
    &\|   {J}_2 (\varepsilon^{-1} \tau) \mathcal{R} (\varepsilon)^{1/4} \|_{L_2(\mathbb{R}^d) \to L_2(\mathbb{R}^d)}
 \le \mathcal{C}_{6}(1 + |\tau|).
    \end{align*}
  The constants $\mathcal{C}_5$ and $\mathcal{C}_{6}$
depend only on $\alpha_0$, $\alpha_1$,  $\|g\|_{L_\infty}$, $\|g^{-1}\|_{L_\infty}$,  $\|f\|_{L_\infty}$, $\|f^{-1}\|_{L_\infty}$, $r_0$, $n$, and on the number $c^{\circ}$
defined by~\emph{\eqref{c^circ}}.
\end{theorem}

Theorem \ref{sndw_cos_thrm_1} was known before (see \cite[Theorem 9.2]{BSu5} and \cite[Section 10]{M}).

By analogy with the proof of Theorem~\ref{s<2_cos_thrm},   we deduce the following result from Theorem~\ref{s<2_thrm};
this confirms that the result of Theorem~\ref{sndw_cos_thrm_1} is sharp.

\begin{theorem}
    \label{sndw_s<2_cos_thrm}
 Let $\widehat{N}_{0,Q}(\boldsymbol{\theta})$  be the operator defined by~\emph{\eqref{N0Q_invar_repr}}. Suppose that $\widehat{N}_{0,Q}(\boldsymbol{\theta}_0) \ne 0$ for some $\boldsymbol{\theta}_0 \in \mathbb{S}^{d-1}$.

 \noindent
 $1^\circ$. Let $0 \ne \tau \in \mathbb{R}$ and $0 \le s < 2$. Then there does not exist a constant $\mathcal{C}(\tau) > 0$ such that the estimate
    \begin{equation*}
    \| {J}_1 (\varepsilon^{-1} \tau) \mathcal{R} (\varepsilon)^{s/2}\|_{L_2(\mathbb{R}^d) \to L_2(\mathbb{R}^d)}
 \le \mathcal{C}(\tau) \varepsilon
    \end{equation*}
holds  for all sufficiently small $\varepsilon > 0$.

 \noindent
 $2^\circ$. Let $0 \ne \tau \in \mathbb{R}$ and $0 \le s < 1$. Then there does not exist a constant $\mathcal{C}(\tau) > 0$ such that the estimate
    \begin{equation*}
        \| {J}_2 (\varepsilon^{-1} \tau) \mathcal{R} (\varepsilon)^{s/2}\|_{L_2(\mathbb{R}^d) \to L_2(\mathbb{R}^d)}
 \le \mathcal{C}(\tau)
    \end{equation*}
holds  for all sufficiently small $\varepsilon > 0$.
\end{theorem}

\section*{Chapter 3. Homogenization problems for hyperbolic equations}

\section{Approximation of the operators $\cos(\tau \mathcal{A}_\varepsilon^{1/2})$ and $\mathcal{A}_\varepsilon^{-1/2} \sin(\tau \mathcal{A}_\varepsilon^{1/2})$}

\subsection{The operators $\widehat{\mathcal{A}}_\varepsilon$ and $\mathcal{A}_\varepsilon$. The scaling transformation}

If $\psi(\mathbf{x})$~is a $\Gamma$-periodic measurable function in  $\mathbb{R}^d$, we  denote $\psi^{\varepsilon}(\mathbf{x}) := \psi(\varepsilon^{-1} \mathbf{x}), \; \varepsilon > 0$. \emph{Our main objects} are  the operators $\widehat{\mathcal{A}}_\varepsilon$ and $\mathcal{A}_\varepsilon$ acting in $L_2 (\mathbb{R}^d; \mathbb{C}^n)$ and formally given by
\begin{align}
\label{Ahat_eps}
\widehat{\mathcal{A}}_\varepsilon &:= b(\mathbf{D})^* g^{\varepsilon}(\mathbf{x}) b(\mathbf{D}), \\
\label{A_eps}
\mathcal{A}_\varepsilon &:= (f^{\varepsilon}(\mathbf{x}))^* b(\mathbf{D})^* g^{\varepsilon}(\mathbf{x}) b(\mathbf{D}) f^{\varepsilon}(\mathbf{x}).
\end{align}
The precise definitions are given in terms of the corresponding quadratic forms (cf. Subsection~\ref{A_oper_subsect}).

Let $T_{\varepsilon}$~be the \emph{unitary scaling transformation in $L_2 (\mathbb{R}^d; \mathbb{C}^n)$} defined by
$(T_{\varepsilon} \mathbf{u})(\mathbf{x}) = \varepsilon^{d/2} \mathbf{u} (\varepsilon \mathbf{x})$,  $\varepsilon > 0$.
Then $\mathcal{A}_\varepsilon = \varepsilon^{-2}T_{\varepsilon}^* \mathcal{A} T_{\varepsilon}$. Hence,
\begin{equation}
\label{cos_and_scale_transform}
\cos(\tau \mathcal{A}_\varepsilon^{1/2}) = T_{\varepsilon}^* \cos(\varepsilon^{-1} \tau \mathcal{A}^{1/2}) T_{\varepsilon},
\quad
\mathcal{A}_\varepsilon^{-1/2} \sin(\tau \mathcal{A}_\varepsilon^{1/2}) = \varepsilon \, T_{\varepsilon}^* \mathcal{A}^{-1/2} \sin(\varepsilon^{-1} \tau \mathcal{A}^{1/2}) T_{\varepsilon}.
\end{equation}
The operator $\widehat{\mathcal{A}}_\varepsilon$ satisfies similar relations.
Next, applying the scaling transformation to the resolvent of the operator $\mathcal{H}_0 = - \Delta$, we obtain
\begin{equation}
\label{H0_resolv_and_scale_transform}
(\mathcal{H}_0 + I)^{-1} = \varepsilon^2 T_\varepsilon^* (\mathcal{H}_0 + \varepsilon^2 I)^{-1} T_\varepsilon = T_\varepsilon^*
\mathcal{R} (\varepsilon) T_\varepsilon,
\end{equation}
where $\mathcal{R} (\varepsilon)$ is given by~\eqref{R(epsilon)}.
Finally, if $\psi(\mathbf{x})$~is a $\Gamma$-periodic function, then
\begin{equation}
\label{mult_op_and_scale_transform}
[\psi^{\varepsilon}] = T_\varepsilon^* [\psi] T_\varepsilon.
\end{equation}

\subsection{Approximation of the operators $\cos( \tau \widehat{\mathcal{A}}_\varepsilon^{1/2})$ and
$\widehat{\mathcal{A}}_\varepsilon^{-1/2}\sin( \tau \widehat{\mathcal{A}}_\varepsilon^{1/2})$}

We start with the simpler operator \eqref{Ahat_eps}. Let $\widehat{\mathcal{A}}^0$~be the effective operator~\eqref{hatA0}. Using relations of the form~\eqref{cos_and_scale_transform} (for the operators $\widehat{\mathcal{A}}_\varepsilon$ and $\widehat{\mathcal{A}}^0$) and identity~\eqref{H0_resolv_and_scale_transform}, we obtain
\begin{align}
\label{cos-cos_and_scale_transform}
&\bigl(\cos( \tau \widehat{\mathcal{A}}_\varepsilon^{1/2}) - \cos( \tau (\widehat{\mathcal{A}}^0)^{1/2}) \bigr) (\mathcal{H}_0 + I)^{-s/2} =
T_{\varepsilon}^*  \widehat{J}_1 (\varepsilon^{-1} \tau) \mathcal{R} (\varepsilon)^{s/2} T_{\varepsilon},\quad \varepsilon > 0,
\\
\label{sin-sin_and_scale_transform}
&\bigl( \widehat{\mathcal{A}}_\varepsilon^{-1/2} \sin( \tau \widehat{\mathcal{A}}_\varepsilon^{1/2}) -
(\widehat{\mathcal{A}}^0)^{-1/2} \sin( \tau (\widehat{\mathcal{A}}^0)^{1/2}) \bigr) (\mathcal{H}_0 + I)^{-s/2}
 = \varepsilon\, T_{\varepsilon}^*  \widehat{J}_2 (\varepsilon^{-1} \tau) \mathcal{R} (\varepsilon)^{s/2} T_{\varepsilon},\quad \varepsilon > 0.
\end{align}

In the general case, the following result holds.

\begin{theorem}
    \label{cos_sin_thrm1_H^s_L2}
Let $\widehat{\mathcal{A}}_{\varepsilon}$~be the operator~\emph{\eqref{Ahat_eps}} and let  $\widehat{\mathcal{A}}^0$~be
the effective operator~\emph{\eqref{hatA0}}. Then for $0 \le s \le 2$, $0 \le r \le 1$, $\tau \in \mathbb{R}$, and $\varepsilon > 0$ we have
    \begin{gather}
    \label{cos_thrm1_H^s_L2_est}
    \| \cos( \tau \widehat{\mathcal{A}}_{\varepsilon}^{1/2}) - \cos( \tau (\widehat{\mathcal{A}}^0)^{1/2}) \|_{H^s (\mathbb{R}^d) \to L_2 (\mathbb{R}^d)}
     \le \widehat{\mathfrak{C}}_1(s) (1+ |\tau|)^{s/2} \varepsilon^{s/2},
    \\
    \label{sin_thrm1_H^s_L2_est}
    \|\widehat{\mathcal{A}}_\varepsilon^{-1/2} \sin( \tau \widehat{\mathcal{A}}_\varepsilon^{1/2}) - 
(\widehat{\mathcal{A}}^0)^{-1/2} \sin( \tau (\widehat{\mathcal{A}}^0)^{1/2}) \|_{H^r (\mathbb{R}^d) \to L_2 (\mathbb{R}^d)} \le \widehat{\mathfrak{C}}_2 (r) (1+ |\tau|)  \varepsilon^{r},
    \end{gather}
   where $\widehat{\mathfrak{C}}_1 (s) = 2^{1-s/2} \widehat{\mathcal{C}}_1^{s/2}$ and
    $\widehat{\mathfrak{C}}_2 (r) =  2^{1-r} \widehat{\mathcal{C}}_2^{r}$.
     The constants $\widehat{\mathcal{C}}_1$ and $\widehat{\mathcal{C}}_2$
depend only on $\alpha_0$, $\alpha_1$, $\|g\|_{L_\infty}$, $\|g^{-1}\|_{L_\infty}$, and $r_0$.
\end{theorem}

Theorem~\ref{cos_sin_thrm1_H^s_L2} can be deduced from Theorem~\ref{cos_thrm_1} and relations \eqref{cos-cos_and_scale_transform}, \eqref{sin-sin_and_scale_transform}.
This result was known before (see~\cite[Theorem~13.2]{BSu5} and \cite[Theorem 11.2]{M}).

Now, we show that the result can be improved under the additional assumptions.

\begin{theorem}
    \label{cos_sin_thrm2_H^s_L2}
   Suppose that the assumptions of Theorem~\emph{\ref{cos_sin_thrm1_H^s_L2}} are satisfied.
Let $\widehat{N}(\boldsymbol{\theta})$ be the operator defined by~\emph{\eqref{N(theta)}}. Suppose that $\widehat{N}(\boldsymbol{\theta}) = 0$ for any $\boldsymbol{\theta} \in \mathbb{S}^{d-1}$. Then for $0 \le s \le 3/2$, $0\le r \le 1/2$,
$\tau \in \mathbb{R}$, and $\varepsilon > 0$ we have
    \begin{align}
    \label{ccc}
    &\| \cos( \tau \widehat{\mathcal{A}}_{\varepsilon}^{1/2}) - \cos( \tau (\widehat{\mathcal{A}}^0)^{1/2}) \|_{H^s (\mathbb{R}^d) \to L_2 (\mathbb{R}^d)} \le
    \widehat{\mathfrak{C}}_3 (s) (1+|\tau|)^{2s/3} \varepsilon^{2s/3},
    \\
    \label{sss}
    &\|\widehat{\mathcal{A}}_\varepsilon^{-1/2} \sin( \tau \widehat{\mathcal{A}}_\varepsilon^{1/2}) - (\widehat{\mathcal{A}}^0)^{-1/2} \sin( \tau (\widehat{\mathcal{A}}^0)^{1/2}) \|_{H^r (\mathbb{R}^d) \to L_2 (\mathbb{R}^d)} \le \widehat{\mathfrak{C}}_4 (r) (1+|\tau|)  \varepsilon^{2r},
    \end{align}
where $\widehat{\mathfrak{C}}_3 (s) = 2^{1-2s/3} \widehat{\mathcal{C}}_3^{2s/3}$ and $\widehat{\mathfrak{C}}_4 (r) =  2^{1-2r} \widehat{\mathcal{C}}_4^{2r}$.
The constants $\widehat{\mathcal{C}}_3$ and $\widehat{\mathcal{C}}_4$
depend only on $\alpha_0$, $\alpha_1$,  $\|g\|_{L_\infty}$, $\|g^{-1}\|_{L_\infty}$, and $r_0$.
\end{theorem}

\begin{proof}
Since $T_{\varepsilon}$  is unitary, it follows from Theorem~\ref{cos_thrm_2} and
\eqref{cos-cos_and_scale_transform}, \eqref{sin-sin_and_scale_transform} that
\begin{align}
\label{13.12}
&\| \bigl( \cos( \tau \widehat{\mathcal{A}}_\varepsilon^{1/2}) - \cos( \tau (\widehat{\mathcal{A}}^0)^{1/2}) \bigr) (\mathcal{H}_0 + I)^{-3/4}
\|_{L_2 (\mathbb{R}^d) \to L_2 (\mathbb{R}^d)} \le \widehat{\mathcal{C}}_3(1 + |\tau|)\varepsilon,
\\
\label{13.12a}
&\| \bigl( \widehat{\mathcal{A}}_\varepsilon^{-1/2} \sin( \tau \widehat{\mathcal{A}}_\varepsilon^{1/2}) -
(\widehat{\mathcal{A}}^0)^{-1/2} \sin( \tau (\widehat{\mathcal{A}}^0)^{1/2}) \bigr) (\mathcal{H}_0 + I)^{-1/4}
\|_{L_2 (\mathbb{R}^d) \to L_2 (\mathbb{R}^d)} \le \widehat{\mathcal{C}}_4(1 + |\tau|)\varepsilon,
\end{align}
for $\tau \in \mathbb{R}$ and $\varepsilon >0$.
Obviously,
\begin{align}
\label{cos-cos_le_2}
&\| \cos( \tau \widehat{\mathcal{A}}_{\varepsilon}^{1/2}) - \cos( \tau (\widehat{\mathcal{A}}^0)^{1/2}) \|_{L_2 (\mathbb{R}^d) \to L_2 (\mathbb{R}^d)} \le 2,
\\
\label{sin-sin_le_2}
&\| \widehat{\mathcal{A}}_\varepsilon^{-1/2} \sin( \tau \widehat{\mathcal{A}}_\varepsilon^{1/2}) -
(\widehat{\mathcal{A}}^0)^{-1/2} \sin( \tau (\widehat{\mathcal{A}}^0)^{1/2})
\|_{L_2 (\mathbb{R}^d) \to L_2 (\mathbb{R}^d)} \le 2 |\tau| \le 2(1 + |\tau|),
\end{align}
for $\tau \in \mathbb{R}$ and $\varepsilon >0$.
Interpolating between~\eqref{cos-cos_le_2} and~\eqref{13.12}, and between~\eqref{sin-sin_le_2} and~\eqref{13.12a},
for $ 0 \le s \le 3/2$ and $0\le r\le 1/2$ we obtain
\begin{align}
\label{cos_interp_1}
&\| \bigl( \cos( \tau \widehat{\mathcal{A}}_{\varepsilon}^{1/2}) - \cos( \tau (\widehat{\mathcal{A}}^0)^{1/2}) \bigr)  (\mathcal{H}_0 + I)^{-s/2}
\|_{L_2 (\mathbb{R}^d) \to L_2 (\mathbb{R}^d)}
\le \widehat{\mathfrak{C}}_3 (s) (1+ |\tau|)^{2s/3}  \varepsilon^{2s/3},
\\
\label{sin_interp_1}
&\| \bigl( \widehat{\mathcal{A}}_{\varepsilon}^{-1/2} \sin( \tau \widehat{\mathcal{A}}_{\varepsilon}^{1/2})
- (\widehat{\mathcal{A}}^0)^{-1/2} \sin( \tau (\widehat{\mathcal{A}}^0)^{1/2}) \bigr)  (\mathcal{H}_0 + I)^{-r/2}
\|_{L_2 (\mathbb{R}^d) \to L_2 (\mathbb{R}^d)}
\le \widehat{\mathfrak{C}}_4 (r) (1+ |\tau|) \varepsilon^{2r},
\end{align}
for $\tau \in \mathbb{R}$ and $\varepsilon > 0$.
The operator $(\mathcal{H}_0 + I)^{s/2}$  is an isometric isomorphism of $H^s(\mathbb{R}^d; \mathbb{C}^n)$ onto  $L_2(\mathbb{R}^d; \mathbb{C}^n)$. Therefore, \eqref{cos_interp_1} is equivalent to \eqref{ccc}, and \eqref{sin_interp_1} is equivalent to \eqref{sss}.
\end{proof}

Recall that some sufficient conditions for $\widehat{N}(\boldsymbol{\theta}) \equiv 0$ are given in
Proposition~\ref{N=0_proposit}.

Finally, Theorem~\ref{cos_thrm_3} together with~\eqref{cos-cos_and_scale_transform}, \eqref{sin-sin_and_scale_transform}, \eqref{cos-cos_le_2}, and \eqref{sin-sin_le_2} leads to the following result.

\begin{theorem}
    \label{cos_sin_thrm3_H^s_L2}
  Suppose that the assumptions of Theorem~\emph{\ref{cos_sin_thrm1_H^s_L2}} are satisfied.
 Suppose that Condition~\emph{\ref{cond9}}
\emph{(}or more restrictive Condition~\emph{\ref{cond99})} is satisfied.
Then for $0 \le s \le 3/2$, $0 \le r \le 1/2$, $\tau \in \mathbb{R}$,
and $\varepsilon > 0$ we have
 \begin{align*}
    &\| \cos( \tau \widehat{\mathcal{A}}_{\varepsilon}^{1/2}) - \cos( \tau (\widehat{\mathcal{A}}^0)^{1/2}) \|_{H^s (\mathbb{R}^d) \to L_2 (\mathbb{R}^d)} \le
    \widehat{\mathfrak{C}}_5 (s) (1+|\tau|)^{2s/3} \varepsilon^{2s/3},
    \\
    &\|\widehat{\mathcal{A}}_\varepsilon^{-1/2} \sin( \tau \widehat{\mathcal{A}}_\varepsilon^{1/2}) - (\widehat{\mathcal{A}}^0)^{-1/2} \sin( \tau (\widehat{\mathcal{A}}^0)^{1/2}) \|_{H^r (\mathbb{R}^d) \to L_2 (\mathbb{R}^d)} \le \widehat{\mathfrak{C}}_6 (r) (1+|\tau|)  \varepsilon^{2r},
    \end{align*}
where $\widehat{\mathfrak{C}}_5 (s) = 2^{1-2s/3} \widehat{\mathcal{C}}_5^{2s/3}$ and
      $\widehat{\mathfrak{C}}_6 (r) =  2^{1-2r} \widehat{\mathcal{C}}_6^{2r}$.
The constants $\widehat{\mathcal{C}}_5$ and $\widehat{\mathcal{C}}_6$
depend only on $\alpha_0$, $\alpha_1$,  $\|g\|_{L_\infty}$, $\|g^{-1}\|_{L_\infty}$, $r_0$, $n$, and on the number $\widehat{c}^{\circ}$
given by~\emph{\eqref{hatc^circ}}.
\end{theorem}

Recall that some sufficient conditions ensuring that Condition~\ref{cond99} is valid, are given in
Remark~\ref{rem9.5}.

\begin{remark}
\label{remS11}
Theorems \emph{\ref{cos_sin_thrm1_H^s_L2}}, \emph{\ref{cos_sin_thrm2_H^s_L2}}, and \emph{\ref{cos_sin_thrm3_H^s_L2}}
allow us to obtain qualified estimates for small $\varepsilon$ and large $|\tau|$, which is of independent interest.
For instance, under the assumptions of Theorems \emph{\ref{cos_sin_thrm1_H^s_L2}},
if $0< \varepsilon \le 1$ and $|\tau|=\varepsilon^{-\alpha}$ with $0 < \alpha <1$,
then the left-hand side of  \eqref{cos_thrm1_H^s_L2_est} is of order $O(\varepsilon^{s(1-\alpha)/2})$\emph{;}
if $0< \varepsilon \le 1$ and $|\tau|=\varepsilon^{-\alpha}$ with $0 < \alpha <r$,
then the left-hand side of  \eqref{sin_thrm1_H^s_L2_est} is of order $O(\varepsilon^{r -\alpha})$.
\end{remark}

Applying Theorem~\ref{s<2_cos_thrm},  we confirm the sharpness of the result of Theorem~\ref{cos_sin_thrm1_H^s_L2}
in the general case.

\begin{theorem}
    \label{s<2_cos_thrm_Rd}
 Suppose that the assumptions of Theorem~\emph{\ref{cos_sin_thrm1_H^s_L2}} are satisfied. Let $\widehat{N}_0 (\boldsymbol{\theta})$
 be the operator defined by~\emph{\eqref{N0_invar_repr}}. Suppose that $\widehat{N}_0 (\boldsymbol{\theta}_0) \ne 0$ for some $\boldsymbol{\theta}_0 \in \mathbb{S}^{d-1}$.

 \noindent $1^\circ$. Let $0 \ne \tau \in \mathbb{R}$ and $0 \le s < 2$.
Then there does not exist a constant $\mathcal{C}(\tau) > 0$ such that the estimate
    \begin{equation}
    \label{s<2_cos_est1}
    \| \bigl( \cos (\tau \widehat{\mathcal{A}}_\varepsilon^{1/2})  - \cos (\tau (\widehat{\mathcal{A}}^0)^{1/2}) \bigr)(\mathcal{H}_0  + I)^{-s/2}\|_{L_2(\mathbb{R}^d) \to L_2 (\mathbb{R}^d) }  \le \mathcal{C}(\tau) \varepsilon
    \end{equation}
holds for all sufficiently small $\varepsilon > 0$.

 \noindent $2^\circ$. Let $0 \ne \tau \in \mathbb{R}$ and $0 \le s < 1$.
Then there does not exist a constant $\mathcal{C}(\tau) > 0$ such that the estimate
    \begin{equation*}
    \| \bigl( \widehat{\mathcal{A}}_\varepsilon^{-1/2} \sin (\tau \widehat{\mathcal{A}}_\varepsilon^{1/2})
     - (\widehat{\mathcal{A}}^0)^{-1/2} \sin (\tau (\widehat{\mathcal{A}}^0)^{1/2}) \bigr)(\mathcal{H}_0  + I)^{-s/2}\|_{L_2(\mathbb{R}^d) \to L_2 (\mathbb{R}^d) }  \le \mathcal{C}(\tau) \varepsilon
    \end{equation*}
holds for all sufficiently small $\varepsilon > 0$.
\end{theorem}

\begin{proof} Let us check assertion $1^\circ$. We prove by contradiction. Suppose that for some $\tau \ne 0$ and $0 \le s < 2$ there exists a constant
   $\mathcal{C}(\tau) > 0$ such that estimate~\eqref{s<2_cos_est1}  holds for all sufficiently small $\varepsilon > 0$.
   Then, by~\eqref{cos-cos_and_scale_transform}, for sufficiently small $\varepsilon$ estimate~\eqref{s<2_cos_est} is also valid.
   But this contradicts the statement of Theorem~\ref{s<2_cos_thrm}($1^\circ$).
   Assertion $2^\circ$ is proved similarly with the help of Theorem~\ref{s<2_cos_thrm}($2^\circ$).
    \end{proof}

\subsection{Approximation of the sandwiched operators $\cos(\tau \mathcal{A}_\varepsilon^{1/2})$
and $\mathcal{A}_\varepsilon^{-1/2} \sin(\tau \mathcal{A}_\varepsilon^{1/2})$}
Now we proceed to the operator $\mathcal{A}_\varepsilon$ (see~\eqref{A_eps}).
Let $\mathcal{A}^0$~be defined by~\eqref{A0}. Using relations of the form~\eqref{cos_and_scale_transform}
for the operators $\mathcal{A}_\varepsilon$ and $\mathcal{A}^0$, and taking~\eqref{H0_resolv_and_scale_transform} and \eqref{mult_op_and_scale_transform} into account, we obtain
\begin{align}
\label{sndw_cos-cos_and_scale_transform}
&\bigl(f^\varepsilon \cos( \tau \mathcal{A}_\varepsilon^{1/2}) (f^\varepsilon)^{-1} - f_0 \cos( \tau (\mathcal{A}^0)^{1/2}) f_0^{-1} \bigr) (\mathcal{H}_0 + I)^{-s/2} =
 T_{\varepsilon}^* J_1 (\varepsilon^{-1} \tau)  \mathcal{R} (\varepsilon)^{s/2} T_{\varepsilon},
\\
\label{sndw_sin-sin_and_scale_transform}
&\bigl(f^\varepsilon \mathcal{A}_\varepsilon^{-1/2} \sin( \tau \mathcal{A}_\varepsilon^{1/2}) (f^\varepsilon)^* -
f_0 (\mathcal{A}^0)^{-1/2} \sin( \tau (\mathcal{A}^0)^{1/2}) f_0 \bigr) (\mathcal{H}_0 + I)^{-s/2} =
\varepsilon T_{\varepsilon}^* J_2 (\varepsilon^{-1} \tau)  \mathcal{R} (\varepsilon)^{s/2} T_{\varepsilon}.
\end{align}

In the general case the following result holds.

\begin{theorem}
    \label{sndw_cos_sin_thrm1_H^s_L2}
    Let $\mathcal{A}_{\varepsilon}$~and~$\mathcal{A}^0$ be the operators defined by~\emph{\eqref{A_eps}} and~\emph{\eqref{A0}}, respectively.
 Then for $0 \le s \le 2$,   $0\le r \le 1$, $\tau \in \mathbb{R}$, and $\varepsilon > 0$ we have
    \begin{align*}
    &\| f^\varepsilon \cos( \tau \mathcal{A}_\varepsilon^{1/2})  (f^\varepsilon)^{-1}  - f_0 \cos( \tau (\mathcal{A}^0)^{1/2}) f_0^{-1}  \|_{H^s (\mathbb{R}^d) \to L_2 (\mathbb{R}^d)} \le \mathfrak{C}_1 (s) (1+ |\tau|)^{s/2} \varepsilon^{s/2},
    \\
    &\| f^\varepsilon \mathcal{A}_\varepsilon^{-1/2} \sin( \tau \mathcal{A}_\varepsilon^{1/2}) (f^\varepsilon)^*
    -  f_0 (\mathcal{A}^0)^{-1/2} \sin( \tau (\mathcal{A}^0)^{1/2}) f_0 \|_{H^r (\mathbb{R}^d) \to L_2 (\mathbb{R}^d)}
    \le \mathfrak{C}_2 (r) (1+ |\tau|)  \varepsilon^{r},
    \end{align*}
    where $\mathfrak{C}_1 (s) = (2 \| f \|_{L_\infty} \| f^{-1} \|_{L_\infty})^{1-s/2} \mathcal{C}_1^{s/2}$ and
    $\mathfrak{C}_2 (r) =  (2 \| f \|^2_{L_\infty})^{1- r} \mathcal{C}_2^{r}$.
   The constants $\mathcal{C}_1$ and $\mathcal{C}_2$
 depend only on $\alpha_0$, $\alpha_1$, $\|g\|_{L_\infty}$, $\|g^{-1}\|_{L_\infty}$, $\| f \|_{L_\infty}$, $\| f^{-1} \|_{L_\infty}$, and $r_0$.
\end{theorem}

Theorem \ref{sndw_cos_sin_thrm1_H^s_L2} can be deduced from Theorem~\ref{sndw_cos_thrm_1}
and relations \eqref{sndw_cos-cos_and_scale_transform}, \eqref{sndw_sin-sin_and_scale_transform}.
This result was known before (see~\cite[Theorem~13.4]{BSu5} and \cite[Theorem 12.2]{M}).

Now we improve the results under the additional assumptions.
 Theorem~\ref{sndw_cos_thrm_2}, relations \eqref{sndw_cos-cos_and_scale_transform}, \eqref{sndw_sin-sin_and_scale_transform},
and obvious estimates
    \begin{align}
\label{cl2}
    &\| f^\varepsilon \cos( \tau \mathcal{A}_\varepsilon^{1/2})  (f^\varepsilon)^{-1}  - f_0 \cos( \tau (\mathcal{A}^0)^{1/2}) f_0^{-1}  \|_{L_2 (\mathbb{R}^d) \to L_2 (\mathbb{R}^d)} \le 2 \|f\|_{L_\infty} \|f^{-1}\|_{L_\infty},
    \\
\label{sl2}
    &\| f^\varepsilon \mathcal{A}_\varepsilon^{-1/2} \sin( \tau \mathcal{A}_\varepsilon^{1/2}) (f^\varepsilon)^*
    - f_0 (\mathcal{A}^0)^{-1/2} \sin( \tau (\mathcal{A}^0)^{1/2}) f_0 \|_{L_2 (\mathbb{R}^d) \to L_2 (\mathbb{R}^d)}
       \le 2 \|f\|_{L_\infty}^2 |\tau|
    \end{align}
imply the following result.

\begin{theorem}
    \label{sndw_cos_sin_thrm2_H^s_L2}
    Suppose that the assumptions of Theorem~\emph{\ref{sndw_cos_sin_thrm1_H^s_L2}} are satisfied.
Let  $\widehat{N}_Q (\boldsymbol{\theta})$  be the operator defined by~\emph{\eqref{N_Q(theta)}}. Suppose that
    $\widehat{N}_Q (\boldsymbol{\theta}) = 0$ for any $\boldsymbol{\theta} \in \mathbb{S}^{d-1}$.
Then for $0 \le s \le 3/2$, $0\le r \le 1/2$, $\tau \in \mathbb{R}$, and $\varepsilon > 0$ we have
    \begin{align*}
    &\| f^\varepsilon \cos( \tau \mathcal{A}_\varepsilon^{1/2})  (f^\varepsilon)^{-1}  - f_0 \cos( \tau (\mathcal{A}^0)^{1/2}) f_0^{-1}  \|_{H^s (\mathbb{R}^d) \to L_2 (\mathbb{R}^d)} \le \mathfrak{C}_3 (s) (1+ |\tau|)^{2s/3} \varepsilon^{2s/3},
    \\
    &\| f^\varepsilon \mathcal{A}_\varepsilon^{-1/2} \sin( \tau \mathcal{A}_\varepsilon^{1/2}) (f^\varepsilon)^* - f_0 (\mathcal{A}^0)^{-1/2} \sin( \tau (\mathcal{A}^0)^{1/2}) f_0 \|_{H^r (\mathbb{R}^d) \to L_2 (\mathbb{R}^d)}
       \le \mathfrak{C}_4 (r) (1+| \tau|)  \varepsilon^{2r},
    \end{align*}
    where
    $\mathfrak{C}_3 (s) = (2 \| f \|_{L_\infty} \| f^{-1} \|_{L_\infty})^{1-2s/3} \mathcal{C}_3^{2s/3}$ and
    $\mathfrak{C}_4 (r) = (2 \| f \|_{L_\infty}^2 )^{1-2r} \mathcal{C}_4^{2r}$.
    The constants $\mathcal{C}_3$ and $\mathcal{C}_4$ depend only on $\alpha_0$, $\alpha_1$, $\|g\|_{L_\infty}$, $\|g^{-1}\|_{L_\infty}$, $\|f\|_{L_\infty}$, $\|f^{-1}\|_{L_\infty}$, and $r_0$.
\end{theorem}

Recall that some sufficient conditions for $\widehat{N}_Q (\boldsymbol{\theta}) \equiv 0$ are given by
Proposition~\ref{N_Q=0_proposit}.

Finally, Theorem~\ref{sndw_cos_thrm_3} and relations \eqref{sndw_cos-cos_and_scale_transform}--\eqref{sl2}  imply the following result.

\begin{theorem}
    \label{sndw_cos_sin_thrm3_H^s_L2}
    Suppose that the assumptions of Theorem~\emph{\ref{sndw_cos_sin_thrm1_H^s_L2}} are satisfied.
Suppose that Condition~\emph{\ref{sndw_cond1}} \emph{(}or more restrictive Condition~\emph{\ref{sndw_cond2})} is satisfied.
Then for $0 \le s \le 3/2$, $0\le r \le 1/2$, $\tau \in \mathbb{R}$, and $\varepsilon > 0$ we have
   \begin{align*}
   &\| f^\varepsilon \cos( \tau \mathcal{A}_\varepsilon^{1/2})  (f^\varepsilon)^{-1}  - f_0 \cos( \tau (\mathcal{A}^0)^{1/2}) f_0^{-1}  \|_{H^s (\mathbb{R}^d) \to L_2 (\mathbb{R}^d)} \le \mathfrak{C}_5 (s) (1+|\tau|)^{2s/3} \varepsilon^{2s/3},
   \\
   &\| f^\varepsilon \mathcal{A}_\varepsilon^{-1/2} \sin( \tau \mathcal{A}_\varepsilon^{1/2}) (f^\varepsilon)^*
   - f_0 (\mathcal{A}^0)^{-1/2} \sin( \tau (\mathcal{A}^0)^{1/2}) f_0 \|_{H^r (\mathbb{R}^d) \to L_2 (\mathbb{R}^d)}
   \le \mathfrak{C}_6 (r) (1+ |\tau|)  \varepsilon^{2r},
   \end{align*}
   where
   $\mathfrak{C}_5 (s) = (2 \| f \|_{L_\infty} \| f^{-1} \|_{L_\infty})^{1-2s/3} \mathcal{C}_5^{2s/3}$ and
   $\mathfrak{C}_6 (r) =  (2 \| f \|_{L_\infty}^2)^{1-2r} \mathcal{C}_5^{2r}$.
    The constants $\mathcal{C}_5$ and $\mathcal{C}_6$
    depend only on $\alpha_0$, $\alpha_1$,  $\|g\|_{L_\infty}$, $\|g^{-1}\|_{L_\infty}$,  $\|f\|_{L_\infty}$, $\|f^{-1}\|_{L_\infty}$, $r_0$, $n$,
    and on the number $c^{\circ}$ defined by~\emph{\eqref{c^circ}}.
\end{theorem}

Recall that some sufficient conditions ensuring that Condition~\ref{sndw_cond2} holds are given by
Remark~\ref{sndw_simple_spec_remark}.

\begin{remark}
Theorems \emph{\ref{sndw_cos_sin_thrm1_H^s_L2}}, \emph{\ref{sndw_cos_sin_thrm2_H^s_L2}}, and \emph{\ref{sndw_cos_sin_thrm3_H^s_L2}}
allow us to obtain qualified estimates for small $\varepsilon$ and large $|\tau|$. Cf. Remark \emph{\ref{remS11}}.
\end{remark}

Applying Theorem~\ref{sndw_s<2_cos_thrm},  we confirm the sharpness of the result of Theorem~\ref{sndw_cos_sin_thrm1_H^s_L2} in the general case.

\begin{theorem}
    Suppose that the assumptions of Theorem~\emph{\ref{sndw_cos_sin_thrm1_H^s_L2}} are satisfied. Let $\widehat{N}_{0,Q} (\boldsymbol{\theta})$
     be the operator defined by~\emph{\eqref{N0Q_invar_repr}}. Suppose that \hbox{$\widehat{N}_{0,Q} (\boldsymbol{\theta}_0) \ne 0$} for some  $\boldsymbol{\theta}_0 \in \mathbb{S}^{d-1}$.

\noindent $1^\circ$. Let $0 \ne \tau \in \mathbb{R}$ and $0 \le s < 2$.
     Then there does not exist a constant $\mathcal{C}(\tau) > 0$ such that the estimate
    \begin{equation*}
    \|  \bigl( f^\varepsilon \cos( \tau \mathcal{A}_\varepsilon^{1/2})  (f^\varepsilon)^{-1}  - f_0 \cos( \tau (\mathcal{A}^0)^{1/2}) f_0^{-1} \bigr) (\mathcal{H}_0  + I)^{-s/2}\|_{L_2(\mathbb{R}^d) \to L_2 (\mathbb{R}^d) }
    \le \mathcal{C}(\tau) \varepsilon
    \end{equation*}
    holds for all sufficiently small $\varepsilon > 0$.

\noindent $2^\circ$. Let $0 \ne \tau \in \mathbb{R}$ and $0 \le s < 1$.
     Then there does not exist a constant $\mathcal{C}(\tau) > 0$ such that the estimate
    \begin{equation*}
    \|  \bigl( f^\varepsilon \mathcal{A}_\varepsilon^{-1/2} \sin( \tau \mathcal{A}_\varepsilon^{1/2})  (f^\varepsilon)^*  -
f_0 (\mathcal{A}^0)^{-1/2} \sin( \tau (\mathcal{A}^0)^{1/2}) f_0 \bigr) (\mathcal{H}_0  + I)^{-s/2}\|_{L_2(\mathbb{R}^d) \to L_2 (\mathbb{R}^d) }
    \le \mathcal{C}(\tau) \varepsilon
    \end{equation*}
    holds for all sufficiently small $\varepsilon > 0$.
\end{theorem}

\section{Homogenization of the Cauchy problem for  hyperbolic equations}

\subsection{The Cauchy problem for the  equation with the operator $\widehat{\mathcal{A}}_\varepsilon$}

Let $\mathbf{v}_\varepsilon (\mathbf{x}, \tau)$, $\mathbf{x} \in \mathbb{R}^d$, $\tau \in \mathbb{R}$,
be the solution of the Cauchy problem
\begin{equation}
\label{nonhomog_Cauchy_hatA_eps}
\left\{
\begin{aligned}
&\frac{\partial^2 \mathbf{v}_\varepsilon (\mathbf{x}, \tau)}{\partial \tau^2} = - b(\mathbf{D})^* g^\varepsilon (\mathbf{x}) b(\mathbf{D}) \mathbf{v}_\varepsilon (\mathbf{x}, \tau) + \mathbf{F} (\mathbf{x}, \tau), \\
& \mathbf{v}_\varepsilon (\mathbf{x}, 0) = \boldsymbol{\phi} (\mathbf{x}), \quad \frac{\partial \mathbf{v}_\varepsilon }{\partial \tau} (\mathbf{x}, 0) = \boldsymbol{\psi}(\mathbf{x}),
\end{aligned}
\right.
\end{equation}
where $\boldsymbol{\phi}, \boldsymbol{\psi} \in L_2 (\mathbb{R}^d; \mathbb{C}^n)$ and $\mathbf{F} \in L_{1, \mathrm{loc}} (\mathbb{R}; L_2 (\mathbb{R}^d; \mathbb{C}^n) )$.
 The solution can be represented as
 \begin{equation}
\label{14.2}
\mathbf{v}_\varepsilon (\cdot, \tau) = \cos(\tau \widehat{\mathcal{A}}_\varepsilon^{1/2}) \boldsymbol{\phi} + \widehat{\mathcal{A}}_\varepsilon^{-1/2} \sin(\tau \widehat{\mathcal{A}}_\varepsilon^{1/2}) \boldsymbol{\psi}
+ \int_{0}^{\tau} \widehat{\mathcal{A}}_\varepsilon^{-1/2} \sin((\tau - \tilde{\tau}) \widehat{\mathcal{A}}_\varepsilon^{1/2}) \mathbf{F} (\cdot, \tilde{\tau}) \, d \tilde{\tau}.
\end{equation}
Let $\mathbf{v}_0 (\mathbf{x}, \tau)$~be the solution of the homogenized problem
\begin{equation}
\label{nonhomog_Cauchy_hatA_0}
\left\{
\begin{aligned}
&\frac{\partial^2 \mathbf{v}_0 (\mathbf{x}, \tau)}{\partial \tau^2} = - b(\mathbf{D})^* g^0 b(\mathbf{D}) \mathbf{v}_0 (\mathbf{x}, \tau) + \mathbf{F} (\mathbf{x}, \tau) , \\
& \mathbf{v}_0 (\mathbf{x}, 0) = \boldsymbol{\phi} (\mathbf{x}), \quad \frac{\partial \mathbf{v}_0 }{\partial \tau} (\mathbf{x}, 0) = \boldsymbol{\psi}(\mathbf{x}).
\end{aligned}
\right.
\end{equation}
Then
\begin{equation}
\label{14.4}
\mathbf{v}_0 (\cdot, \tau) = \cos(\tau (\widehat{\mathcal{A}}^0)^{1/2}) \boldsymbol{\phi} + (\widehat{\mathcal{A}}^0)^{-1/2} \sin(\tau (\widehat{\mathcal{A}}^0)^{1/2}) \boldsymbol{\psi}
+ \int_{0}^{\tau} (\widehat{\mathcal{A}}^0)^{-1/2} \sin((\tau - \tilde{\tau}) (\widehat{\mathcal{A}}^0)^{1/2}) \mathbf{F} (\cdot, \tilde{\tau}) \, d \tilde{\tau}.
\end{equation}

In the general case the following result holds.

\begin{theorem}
    \label{nonhomog_Cauchy_hatA_eps_thrm}
    Let $\mathbf{v}_\varepsilon$~be the solution of problem~\emph{\eqref{nonhomog_Cauchy_hatA_eps}}, and let $\mathbf{v}_0$~be the solution of problem~\emph{\eqref{nonhomog_Cauchy_hatA_0}}.

  \noindent  $1^{\circ}$.
         If $\boldsymbol{\phi} \in H^s (\mathbb{R}^d; \mathbb{C}^n)$,  $\boldsymbol{\psi} \in H^r (\mathbb{R}^d; \mathbb{C}^n)$, and
$\mathbf{F} \in L_{1, \mathrm{loc}} (\mathbb{R}; H^r (\mathbb{R}^d; \mathbb{C}^n))$, where $0 \le s \le 2$, $0\le r \le 1$,
then for $\tau \in \mathbb{R}$ and $\varepsilon > 0$ we have
        \begin{multline}
        \label{nonhomog_Cauchy_hatA_eps_est1}
        \| \mathbf{v}_\varepsilon(\cdot, \tau) - \mathbf{v}_0 (\cdot, \tau) \|_{L_2 (\mathbb{R}^d)} \le
        \widehat{\mathfrak{C}}_1 (s) (1+ |\tau|)^{s/2} \varepsilon^{s/2}
\| \boldsymbol{\phi} \|_{H^s(\mathbb{R}^d)} \\+ \widehat{\mathfrak{C}}_2 (r) (1+ |\tau|) \varepsilon^r
\left( \| \boldsymbol{\psi} \|_{H^r(\mathbb{R}^d)} +  \|\mathbf{F} \|_{L_1((0,\tau);H^r(\mathbb{R}^d))} \right).
\end{multline}

     \noindent   $2^{\circ}$.  If $\boldsymbol{\phi}, \boldsymbol{\psi} \in L_2 (\mathbb{R}^d; \mathbb{C}^n)$ and $\mathbf{F} \in L_{1, \mathrm{loc}} (\mathbb{R}; L_2 (\mathbb{R}^d; \mathbb{C}^n) )$, then
        $\lim\limits_{\varepsilon \to 0} \| \mathbf{v}_\varepsilon (\cdot, \tau) - \mathbf{v}_0 (\cdot, \tau) \|_{L_2(\mathbb{R}^d)} = 0$, $\tau \in \mathbb{R}$.
\end{theorem}

Theorem \ref{nonhomog_Cauchy_hatA_eps_thrm} follows from Theorem~\ref{cos_sin_thrm1_H^s_L2} and representations
\eqref{14.2}, \eqref{14.4}.

Statement $1^\circ$ of Theorem~\ref{nonhomog_Cauchy_hatA_eps_thrm} can be refined under the additional assumptions.
Theorems~\ref{cos_sin_thrm2_H^s_L2} and ~\ref{cos_sin_thrm3_H^s_L2}  imply the following results.

\begin{theorem}
    \label{nonhomog_Cauchy_hatA_eps_ench_thrm_1}
    Suppose that the assumptions of Theorem~\emph{\ref{nonhomog_Cauchy_hatA_eps_thrm}} are satisfied. Let $\widehat{N}(\boldsymbol{\theta})$ be the operator defined by~\emph{\eqref{N(theta)}}. Suppose that $\widehat{N}(\boldsymbol{\theta}) = 0$ for any $\boldsymbol{\theta} \in \mathbb{S}^{d-1}$. If $\boldsymbol{\phi} \in H^s (\mathbb{R}^d; \mathbb{C}^n)$, $\boldsymbol{\psi} \in H^r (\mathbb{R}^d; \mathbb{C}^n)$, and  $\mathbf{F} \in L_{1, \mathrm{loc}} (\mathbb{R}; H^r (\mathbb{R}^d; \mathbb{C}^n))$, where $0 \le s \le 3/2$, $0\le r \le 1/2$, then for $\tau \in \mathbb{R}$ and $\varepsilon > 0$ we have
    \begin{multline*}
    \| \mathbf{v}_\varepsilon (\cdot, \tau) - \mathbf{v}_0 (\cdot, \tau) \|_{L_2(\mathbb{R}^d)} \le
    \widehat{\mathfrak{C}}_3 (s) (1+|\tau|)^{2s/3} \varepsilon^{2s/3} \| \boldsymbol{\phi} \|_{H^s(\mathbb{R}^d)}
\\
+ \widehat{\mathfrak{C}}_4 (r) (1+ |\tau|) \varepsilon^{2r} \left(\| \boldsymbol{\psi} \|_{H^r(\mathbb{R}^d)}
+ \|\mathbf{F} \|_{L_1((0,\tau);H^r(\mathbb{R}^d))}\right).
    \end{multline*}
      \end{theorem}

\begin{theorem}
    Suppose that the assumptions of Theorem~\emph{\ref{nonhomog_Cauchy_hatA_eps_thrm}} are satisfied. Suppose also that
    Condition~\emph{\ref{cond9}} \emph{(}or more restrictive Condition~\emph{\ref{cond99})} is satisfied.
If $\boldsymbol{\phi} \in H^s (\mathbb{R}^d; \mathbb{C}^n)$, $\boldsymbol{\psi} \in H^r (\mathbb{R}^d; \mathbb{C}^n)$, and  $\mathbf{F} \in L_{1, \mathrm{loc}} (\mathbb{R}; H^r (\mathbb{R}^d; \mathbb{C}^n))$, where $0 \le s \le 3/2$, $0\le r \le 1/2$, then for $\tau \in \mathbb{R}$ and $\varepsilon > 0$ we have
    \begin{multline*}
    \| \mathbf{v}_\varepsilon (\cdot, \tau) - \mathbf{v}_0 (\cdot, \tau) \|_{L_2(\mathbb{R}^d)}  \le
\widehat{\mathfrak{C}}_5 (s) (1+| \tau|)^{2s/3} \varepsilon^{2s/3} \| \boldsymbol{\phi} \|_{H^s(\mathbb{R}^d)}
\\
+ \widehat{\mathfrak{C}}_6 (r) (1+ |\tau|) \varepsilon^{2r} \left( \| \boldsymbol{\psi} \|_{H^r(\mathbb{R}^d)}
+ \|\mathbf{F} \|_{L_1((0,\tau);H^r(\mathbb{R}^d))} \right).
    \end{multline*}
  \end{theorem}

\subsection{More general problem}

Let $Q(\mathbf{x})$ be a $\Gamma$-periodic Hermitian $(n\times n)$-matrix-valued function such that
$Q, Q^{-1} \in L_\infty$; $Q(\mathbf{x})>0$. The matrix 
 $Q(\mathbf{x})^{-1}$ can be written in a factorized form $Q(\mathbf{x})^{-1} = f(\mathbf{x})f(\mathbf{x})^*$,
where $f(\mathbf{x})$ is a $\Gamma$-periodic $(n\times n)$-matrix-valued function such that $f, f^{-1} \in L_\infty$.

Let $\mathbf{z}_\varepsilon(\mathbf{x},\tau)$,  $\mathbf{x}\in \mathbb{R}^d$, $\tau\in \mathbb{R}$, be the solution of the problem
\begin{equation}
\label{nonhomog_Cauchy_wQ}
\left\{
\begin{aligned}
& Q^\varepsilon (\mathbf{x}) \frac{\partial^2 \mathbf{z}_\varepsilon (\mathbf{x}, \tau)}{\partial \tau^2} = - b(\mathbf{D})^* g^\varepsilon (\mathbf{x}) b(\mathbf{D}) \mathbf{z}_\varepsilon (\mathbf{x}, \tau) + \mathbf{F} (\mathbf{x}, \tau), \\
& \mathbf{z}_\varepsilon (\mathbf{x}, 0) = \boldsymbol{\phi} (\mathbf{x}), \quad Q^\varepsilon (\mathbf{x})\frac{\partial \mathbf{z}_\varepsilon }{\partial \tau} (\mathbf{x}, 0) = \boldsymbol{\psi}(\mathbf{x}),
\end{aligned}
\right.
\end{equation}
where $\boldsymbol{\phi}, \boldsymbol{\psi} \in L_2 (\mathbb{R}^d; \mathbb{C}^n)$ and
$\mathbf{F} \in L_{1, \mathrm{loc}} (\mathbb{R}; L_2 (\mathbb{R}^d; \mathbb{C}^n) )$.
Let ${\mathcal A}_\varepsilon$~be the operator~\eqref{A_eps}. Then
\begin{multline}\label{14.8}
    \mathbf{z}_\varepsilon (\cdot, \tau) = f^\varepsilon \cos(\tau \mathcal{A}_\varepsilon^{1/2}) (f^\varepsilon)^{-1} \boldsymbol{\phi} + f^\varepsilon \mathcal{A}_\varepsilon^{-1/2} \sin(\tau \mathcal{A}_\varepsilon^{1/2}) (f^\varepsilon)^* \boldsymbol{\psi}
\\
+ \int_{0}^{\tau} f^\varepsilon \mathcal{A}_\varepsilon^{-1/2} \sin((\tau - \tilde{\tau}) \mathcal{A}_\varepsilon^{1/2}) (f^\varepsilon)^* \mathbf{F} (\cdot, \tilde{\tau}) \, d \tilde{\tau} .
\end{multline}

Let  $\mathbf{z}_0(\mathbf{x},\tau)$ be the solution of the homogenized problem
\begin{equation}
\label{nonhomog_Cauchy_wQ_eff}
\left\{
\begin{aligned}
& \overline{Q} \frac{\partial^2 \mathbf{z}_0 (\mathbf{x}, \tau)}{\partial \tau^2} = - b(\mathbf{D})^* g^0 b(\mathbf{D}) \mathbf{z}_0 (\mathbf{x}, \tau) + \mathbf{F} (\mathbf{x}, \tau), \\
& \mathbf{z}_0 (\mathbf{x}, 0) = \boldsymbol{\phi} (\mathbf{x}), \quad \overline{Q}\frac{\partial \mathbf{z}_0 }{\partial \tau} (\mathbf{x}, 0) = \boldsymbol{\psi}(\mathbf{x}).
\end{aligned}
\right.
\end{equation}
Let $f_0$ be the matrix~\eqref{f_0} and ${\mathcal A}^0$ be the operator~\eqref{A0}.
Then
\begin{multline}\label{14.10}
    \mathbf{z}_0 (\cdot, \tau) = f_0 \cos(\tau (\mathcal{A}^0)^{1/2}) f_0^{-1} \boldsymbol{\phi} + f_0 (\mathcal{A}^0)^{-1/2} \sin(\tau (\mathcal{A}^0)^{1/2}) f_0 \boldsymbol{\psi}
\\
+ \int_{0}^{\tau} f_0 (\mathcal{A}^0)^{-1/2} \sin((\tau - \tilde{\tau}) (\mathcal{A}^0)^{1/2}) f_0 \mathbf{F} (\cdot, \tilde{\tau}) \, d \tilde{\tau}.
\end{multline}

In the general case the following result holds.

\begin{theorem}
    \label{nonhomog_Cauchy_wQ_thrm_1}
    Let $\mathbf{z}_\varepsilon$~be the solution of problem~\emph{\eqref{nonhomog_Cauchy_wQ}}, and let $\mathbf{z}_0$~be the
    solution of problem~\emph{\eqref{nonhomog_Cauchy_wQ_eff}}.

    \noindent $1^{\circ}.$
         If $\boldsymbol{\phi} \in H^s (\mathbb{R}^d;\mathbb{C}^n)$, $\boldsymbol{\psi} \in H^r (\mathbb{R}^d;\mathbb{C}^n)$, and
$\mathbf{F} \in L_{1, \mathrm{loc}} (\mathbb{R}; H^r (\mathbb{R}^d; \mathbb{C}^n))$, where $0 \le s \le 2$, $0\le r \le 1$,
then for $\tau \in \mathbb{R}$ and $\varepsilon > 0$ we have
        \begin{multline*}
        \| \mathbf{z}_\varepsilon(\cdot, \tau) - \mathbf{z}_0 (\cdot, \tau) \|_{L_2 (\mathbb{R}^d)}
\le  \mathfrak{C}_1 (s) (1+ |\tau|)^{s/2} \varepsilon^{s/2} \| \boldsymbol{\phi} \|_{H^s(\mathbb{R}^d)}
\\
+ \mathfrak{C}_2 (r) (1+ |\tau|) \varepsilon^{r} \left( \| \boldsymbol{\psi} \|_{H^r(\mathbb{R}^d)}
 +  \|\mathbf{F} \|_{L_1((0,\tau);H^r(\mathbb{R}^d))} \right).
        \end{multline*}

     \noindent $2^{\circ}.$
        If $\boldsymbol{\phi}, \boldsymbol{\psi} \in L_2 (\mathbb{R}^d; \mathbb{C}^n)$ and $\mathbf{F} \in L_{1, \mathrm{loc}} (\mathbb{R}; L_2 (\mathbb{R}^d; \mathbb{C}^n) )$, then
        $\lim\limits_{\varepsilon \to 0} \| \mathbf{z}_\varepsilon (\cdot, \tau) -  \mathbf{z}_0 (\cdot, \tau) \|_{L_2(\mathbb{R}^d)} = 0$, $\tau \in \mathbb{R}$.
 \end{theorem}

Theorem \ref{nonhomog_Cauchy_wQ_thrm_1}
is a consequence of Theorem~\ref{sndw_cos_sin_thrm1_H^s_L2}
and representations~\eqref{14.8}, \eqref{14.10}.

Statement $1^\circ$ of Theorem~\ref{nonhomog_Cauchy_wQ_thrm_1}  can be refined under the additional assumptions.
Theorems~\ref{sndw_cos_sin_thrm2_H^s_L2} and~\ref{sndw_cos_sin_thrm3_H^s_L2} and relations~\eqref{14.8}, \eqref{14.10} imply the following results.

\begin{theorem}
    \label{nonhomog_Cauchy_wQ_thrm_2}
    Suppose that the assumptions of Theorem~\emph{\ref{nonhomog_Cauchy_wQ_thrm_1}} are satisfied. Let $\widehat{N}_Q (\boldsymbol{\theta})$ be
    the operator defined by~\emph{\eqref{N_Q(theta)}}. Suppose that $\widehat{N}_Q (\boldsymbol{\theta}) = 0$ for any $\boldsymbol{\theta} \in \mathbb{S}^{d-1}$. If $\boldsymbol{\phi} \in H^s (\mathbb{R}^d; \mathbb{C}^n)$, $\boldsymbol{\psi} \in H^r (\mathbb{R}^d; \mathbb{C}^n)$, and  $\mathbf{F} \in L_{1, \mathrm{loc}} (\mathbb{R}; H^r (\mathbb{R}^d; \mathbb{C}^n))$, where $0 \le s \le 3/2$, $0\le r \le 1/2$,
then for $\tau \in \mathbb{R}$ and $\varepsilon > 0$ we have
    \begin{multline*}
    \| \mathbf{z}_\varepsilon (\cdot, \tau) - \mathbf{z}_0 (\cdot, \tau) \|_{L_2(\mathbb{R}^d)}
\le \mathfrak{C}_3 (s) (1+ |\tau|)^{2s/3} \varepsilon^{2s/3}\| \boldsymbol{\phi} \|_{H^s(\mathbb{R}^d)}
\\
+ \mathfrak{C}_4 (r) (1+|\tau|) \varepsilon^{2r} \left( \| \boldsymbol{\psi} \|_{H^r(\mathbb{R}^d)}
+ \|\mathbf{F} \|_{L_1((0,\tau);H^r(\mathbb{R}^d))} \right).
    \end{multline*}
    \end{theorem}

\begin{theorem}
    Suppose that the assumptions of Theorem~\emph{\ref{nonhomog_Cauchy_wQ_thrm_1}} are satisfied.
    Suppose also that Condition~\emph{\ref{sndw_cond1}} \emph{(}or more restrictive Condition~\emph{\ref{sndw_cond2})} is satisfied.
 If $\boldsymbol{\phi} \in H^s (\mathbb{R}^d; \mathbb{C}^n)$, $\boldsymbol{\psi} \in H^r (\mathbb{R}^d; \mathbb{C}^n)$, and  $\mathbf{F} \in L_{1, \mathrm{loc}} (\mathbb{R}; H^r (\mathbb{R}^d; \mathbb{C}^n))$, where $0 \le s \le 3/2$, $0\le r \le 1/2$,
then for $\tau \in \mathbb{R}$ and $\varepsilon > 0$ we have
    \begin{multline*}
    \| \mathbf{z}_\varepsilon (\cdot, \tau) - \mathbf{z}_0 (\cdot, \tau) \|_{L_2(\mathbb{R}^d)} \le \mathfrak{C}_5 (s) (1+|\tau|)^{2s/3} \varepsilon^{2s/3}
    \| \boldsymbol{\phi} \|_{H^s(\mathbb{R}^d)} \\
+ \mathfrak{C}_6 (r) (1+| \tau|) \varepsilon^{2r} \left( \| \boldsymbol{\psi} \|_{H^r(\mathbb{R}^d)}
+  \|\mathbf{F} \|_{L_1((0,\tau);H^r(\mathbb{R}^d))} \right).
    \end{multline*}
\end{theorem}

\section{Application of the general results: the acoustics equation}

\subsection{The model example\label{model}}

In $L_2 (\mathbb{R}^d)$, $d \ge 1$, we consider the operator
\begin{equation}
\label{model_operator}
\widehat{\mathcal{A}} = \mathbf{D}^* g(\mathbf{x}) \mathbf{D} =  - \operatorname{div} g(\mathbf{x}) \nabla.
\end{equation}
Here $g(\mathbf{x})$~is a $\Gamma$-periodic Hermitian $(d \times d)$-matrix-valued function such that
  $g(\mathbf{x}) > 0$ and $g, g^{-1} \in L_\infty$.
The operator~\eqref{model_operator} is a particular case of the operator~\eqref{hatA}.
Now we have $n=1$, $m=d$, and $b(\mathbf{D}) = \mathbf{D}$.
Obviously, condition~\eqref{rank_alpha_ineq} is satisfied with $\alpha_0 = \alpha_1 = 1$.
By~\eqref{hatA0}, the effective operator for the operator~\eqref{model_operator} is given by
$\widehat{\mathcal{A}}^0 = \mathbf{D}^* g^0 \mathbf{D} =  - \operatorname{div} g^0 \nabla$.
According to~\eqref{equation_for_Lambda}--\eqref{g_tilde}, the effective matrix is defined as follows.
Let $\mathbf{e}_1, \ldots, \mathbf{e}_d$~be the standard orthonormal basis in $\mathbb{R}^d$.
Let $\Phi_j \in \widetilde{H}^1 (\Omega)$~be the weak  $\Gamma$-periodic solution of the problem
\begin{equation*}
\operatorname{div} g(\mathbf{x}) (\nabla \Phi_j(\mathbf{x}) + \mathbf{e}_j) = 0, \quad \int_{\Omega} \Phi_j(\mathbf{x}) \, d \mathbf{x} = 0.
\end{equation*}
Then $g^0$~is the $(d \times d)$-matrix with the columns
\begin{equation*}
\mathbf{g}_j^0 = |\Omega|^{-1} \int_{\Omega} g(\mathbf{x}) (\nabla \Phi_j(\mathbf{x}) + \mathbf{e}_j) \, d \mathbf{x}, \quad j = 1, \ldots, d.
\end{equation*}
If $d=1$, then $m = n = 1$, whence $g^0 = \underline{g}$.

If $g(\mathbf{x})$~is a symmetric matrix with real entries, then, by Proposition~\ref{N=0_proposit}($1^\circ$),
  $\widehat{N} (\boldsymbol{\theta}) = 0$ for any $\boldsymbol{\theta} \in \mathbb{S}^{d-1}$. If $g(\mathbf{x})$~is a Hermitian matrix with complex entries, then, in general, $\widehat{N} (\boldsymbol{\theta})$ is not zero. Since $n=1$, then $\widehat{N} (\boldsymbol{\theta}) = \widehat{N}_0 (\boldsymbol{\theta})$ is the operator of multiplication by $\widehat{\mu}(\boldsymbol{\theta})$, where  $\widehat{\mu}(\boldsymbol{\theta})$~is the coefficient at $t^3$ in the expansion  for the first eigenvalue
$
    \widehat{\lambda}(t, \boldsymbol{\theta}) = \widehat{\gamma} (\boldsymbol{\theta}) t^2 + \widehat{\mu} (\boldsymbol{\theta}) t^3 + \ldots
$
of the operator $\widehat{\mathcal{A}} (\mathbf{k})$. A calculation (see~\cite[Subsection~10.3]{BSu3}) shows that
\begin{align*}
    \widehat{N} (\boldsymbol{\theta}) &= \widehat{\mu} (\boldsymbol{\theta}) = -i \sum_{j,l,k=1}^{d} (a_{jlk} - a_{jlk}^*) \theta_j \theta_l \theta_k, \quad \boldsymbol{\theta} \in \mathbb{S}^{d-1}, \\
    a_{jlk} &:= |\Omega|^{-1} \int_{\Omega} \Phi_j (\mathbf{x})^* \left\langle g(\mathbf{x}) (\nabla \Phi_l (\mathbf{x}) + \mathbf{e}_l), \mathbf{e}_k \right\rangle \, d \mathbf{x}, \quad j, l, k = 1, \ldots, d.
\end{align*}

The following example is borrowed from~\cite[Subsection~10.4]{BSu3}.

\begin{example}[\cite{BSu3}]
    \label{model_exmpl}
    Let $d=2$ and $\Gamma = (2 \pi \mathbb{Z})^2$. Suppose that the matrix $g(\mathbf{x})$ is given by
   \begin{equation*}
    g(\mathbf{x}) = \begin{pmatrix}
    1 & i \beta'(x_1) \\
    - i \beta' (x_1) & 1
    \end{pmatrix},
   \end{equation*}
    where $\beta(x_1)$~is a smooth $(2 \pi)$-periodic real-valued function such that \hbox{$1 - (\beta'(x_1))^2 > 0$} and
    $\int_{0}^{2 \pi} \beta(x_1)\, d x_1 = 0$. Then $\widehat{N} (\boldsymbol{\theta}) = - \alpha \pi^{-1} \theta_2^3$, where
    $\alpha = \int_{0}^{2 \pi} \beta (x_1) (\beta' (x_1))^2 dx_1$.
    It is easy to give a concrete example where \hbox{$\alpha \ne 0$}{\rm :} if $\beta (x_1) = c (\sin x_1 + \cos 2x_1)$ with $0 < c < 1/3$, then
    \hbox{$\alpha = - (3 \pi/2) c^3 \ne 0$}. In this example, $\widehat{N} (\boldsymbol{\theta}) = \widehat{\mu} (\boldsymbol{\theta}) \ne 0$ for all $\boldsymbol{\theta} \in \mathbb{S}^1$ except for the points $(\pm 1 ,0)$.
\end{example}

Consider the Cauchy problem
\begin{equation}
\label{model_Cauchy}
\left\{
\begin{aligned}
& \frac{\partial^2 v_\varepsilon (\mathbf{x}, \tau)}{\partial \tau^2} = - \mathbf{D}^* g^\varepsilon (\mathbf{x}) \mathbf{D} v_\varepsilon (\mathbf{x}, \tau) + F(\mathbf{x}, \tau), \\
& v_\varepsilon (\mathbf{x}, 0) = \phi (\mathbf{x}), \quad  \frac{\partial v_\varepsilon }{\partial \tau} (\mathbf{x}, 0) = \psi (\mathbf{x}),
\end{aligned}
\right.
\end{equation}
where $\phi, \psi \in L_2({\mathbb R}^d)$ and $F\in L_{1, \mathrm{loc}} (\mathbb{R}; L_2 (\mathbb{R}^d))$. Let
 $v_0 (\mathbf{x}, \tau)$ be the solution of the homogenized problem
\begin{equation}\label{model_eff}
\left\{
\begin{aligned}
& \frac{\partial^2 v_0 (\mathbf{x}, \tau)}{\partial \tau^2} = - \mathbf{D}^* g^0  \mathbf{D} v_0 (\mathbf{x}, \tau) + F(\mathbf{x}, \tau), \\
& v_0 (\mathbf{x}, 0) = \phi (\mathbf{x}), \quad  \frac{\partial v_0 }{\partial \tau} (\mathbf{x}, 0) = \psi (\mathbf{x}).
\end{aligned}
\right.
\end{equation}
Applying Theorem~\ref{nonhomog_Cauchy_hatA_eps_thrm} in the general case and
Theorem~\ref{nonhomog_Cauchy_hatA_eps_ench_thrm_1}
in the ``real'' case, we arrive at the following statement.

\begin{proposition}
Under the assumptions of Subsection~\emph{\ref{model}}, let $v_\varepsilon$ be the solution of problem~\eqref{model_Cauchy},
and let $v_0$ be the solution of problem~\eqref{model_eff}.

\noindent
$1^\circ$. If $\phi \in H^s(\mathbb{R}^d)$, $\psi \in H^r(\mathbb{R}^d)$,  and $F\in L_{1, \mathrm{loc}} (\mathbb{R}; H^r (\mathbb{R}^d))$
with $0 \le s \le 2$, $0\le r \le 1$, then for $\tau \in \mathbb{R}$ and $\varepsilon >0$ we have
 \begin{multline*}
        \| {v}_\varepsilon(\cdot, \tau) - {v}_0 (\cdot, \tau) \|_{L_2 (\mathbb{R}^d)} \le
        \widehat{\mathfrak{C}}_1 (s) (1+|\tau|)^{s/2} \varepsilon^{s/2} \| {\phi} \|_{H^s(\mathbb{R}^d)}
\\
+ \widehat{\mathfrak{C}}_2 (r) (1+|\tau|) \varepsilon^r \left( \| {\psi} \|_{H^r(\mathbb{R}^d)}
+ \| {F} \|_{L_1((0,\tau);H^r(\mathbb{R}^d))} \right).
\end{multline*}
where $\widehat{\mathfrak{C}}_1 (s)$ depends on $s$, $\|g\|_{L_\infty}$, $\|g^{-1}\|_{L_\infty}$, and $r_0$;
$\widehat{\mathfrak{C}}_2 (r)$
depends on $r$, $\|g\|_{L_\infty}$, $\|g^{-1}\|_{L_\infty}$, and $r_0$.
If \hbox{$\phi, \psi \in L_2(\mathbb{R}^d)$} and $F\in L_{1, \mathrm{loc}} (\mathbb{R}; L_2 (\mathbb{R}^d))$, then
$\| v_\varepsilon(\cdot,\tau) - v_0(\cdot,\tau)\|_{L_2(\mathbb{R}^d)} \to 0$ as $\varepsilon \to 0$, $\tau \in \mathbb{R}$.

\noindent
$2^\circ$. Let $g(\mathbf{x})$ be a symmetric matrix with real entries.
If $\phi \in H^s(\mathbb{R}^d)$, $\psi \in H^r(\mathbb{R}^d)$,  and $F\in L_{1, \mathrm{loc}} (\mathbb{R}; H^r (\mathbb{R}^d))$
with $0 \le s \le 3/2$, $0\le r \le 1/2$, then for $\tau \in \mathbb{R}$ and $\varepsilon >0$ we have
\begin{multline*}
        \| {v}_\varepsilon(\cdot, \tau) - {v}_0 (\cdot, \tau) \|_{L_2 (\mathbb{R}^d)} \le
         \widehat{\mathfrak{C}}_3 (s) (1+ |\tau|)^{2s/3} \varepsilon^{2 s/3}
\| {\phi} \|_{H^s(\mathbb{R}^d)}
\\
+ \widehat{\mathfrak{C}}_4 (r) (1+ |\tau|) \varepsilon^{2r} \left(\| {\psi} \|_{H^r(\mathbb{R}^d)} +  \| {F} \|_{L_1((0,\tau);H^r(\mathbb{R}^d))} \right).
\end{multline*}
Here $\widehat{\mathfrak{C}}_3 (s)$ depends on $s$, $\|g\|_{L_\infty}$, $\|g^{-1}\|_{L_\infty}$, and $r_0${\rm ;}
$\widehat{\mathfrak{C}}_4 (r)$ depends on $r$, $\|g\|_{L_\infty}$, $\|g^{-1}\|_{L_\infty}$, and $r_0$.
\end{proposition}

\subsection{The acoustics equation\label{acoustic}}
Now, under the assumptions of Subsection~\ref{model}, suppose that
$g(\mathbf{x})$ is a symmetric matrix with real entries.
The matrix  $g(\mathbf{x})$ characterizes the parameters of the acoustical (in general,
anisotropic) medium under study.
 Let $Q (\mathbf{x})$~be a $\Gamma$-periodic function in $\mathbb{R}^d$ such that
$Q(\mathbf{x}) > 0$; $Q, Q^{-1} \in L_\infty$.
The function $Q(\mathbf{x})$ stands for the density of the medium.

Consider the Cauchy problem for the acoustics equation in a rapidly oscillating medium:
\begin{equation}
    \label{acoustics_Cauchy}
    \left\{
    \begin{aligned}
        & Q^\varepsilon (\mathbf{x}) \frac{\partial^2 z_\varepsilon (\mathbf{x}, \tau)}{\partial \tau^2} =
        - \mathbf{D}^* g^\varepsilon (\mathbf{x}) \mathbf{D} z_\varepsilon (\mathbf{x}, \tau), \\
        & z_\varepsilon (\mathbf{x}, 0) = \phi (\mathbf{x}), \quad  Q^\varepsilon (\mathbf{x})\frac{\partial z_\varepsilon }{\partial \tau} (\mathbf{x}, 0) = \psi (\mathbf{x}),
    \end{aligned}
    \right.
\end{equation}
where $\phi, \psi \in L_2(\mathbb{R}^d)$~are given functions.
(For simplicity, we consider the homogeneous equation.)
Then the homogenized problem takes the form
\begin{equation}\label{ac_eff}
        \left\{
    \begin{aligned}
        & \overline{Q} \frac{\partial^2 z_0 (\mathbf{x}, \tau)}{\partial \tau^2} = - \mathbf{D}^* g^0 \mathbf{D} z_0 (\mathbf{x}, \tau), \\
        & z_0 (\mathbf{x}, 0) = \phi (\mathbf{x}), \quad  \overline{Q} \frac{\partial z_0 }{\partial \tau} (\mathbf{x}, 0) = \psi (\mathbf{x}).
    \end{aligned}
    \right.
\end{equation}

By Proposition~\ref{N_Q=0_proposit}($1^\circ$), we have $\widehat{N}_Q (\boldsymbol{\theta}) = 0$ for any $\boldsymbol{\theta} \in \mathbb{S}^{d-1}$. (We put $f = Q^{-1/2}$.)
Applying Theorem~\ref{nonhomog_Cauchy_wQ_thrm_2}, we arrive at the following result.

\begin{proposition}
Under the assumptions of Subsection~\emph{\ref{acoustic}}, let $z_\varepsilon$ be the solution of problem~\eqref{acoustics_Cauchy},
and let $z_0$ be the solution of problem~\eqref{ac_eff}.
If $\phi\in H^s(\mathbb{R}^d)$ and $\psi \in H^r(\mathbb{R}^d)$ with some $0 \le s \le 3/2$, $0 \le r \le 1/2$,
then for $\tau \in \mathbb{R}$ and $\varepsilon >0$ we have
 \begin{equation*}
        \| {z}_\varepsilon(\cdot, \tau) - {z}_0 (\cdot, \tau) \|_{L_2 (\mathbb{R}^d)} \le
        {\mathfrak{C}}_3 (s) (1+ |\tau| )^{2s/3} \varepsilon^{2s/3} \| {\phi} \|_{H^s(\mathbb{R}^d)} +
 {\mathfrak{C}}_4 (r) (1+|\tau|) \varepsilon^{2r} \| {\psi} \|_{H^r(\mathbb{R}^d)},
\end{equation*}
where ${\mathfrak{C}}_3 (s)$  depends on $s$, $\|g\|_{L_\infty}$, $\|g^{-1}\|_{L_\infty}$,
$\|Q\|_{L_\infty}$, $\|Q^{-1}\|_{L_\infty}$, and $r_0${\rm ;}
${\mathfrak{C}}_4 (r)$  depends on $r$, $\|g\|_{L_\infty}$, $\|g^{-1}\|_{L_\infty}$,
$\|Q\|_{L_\infty}$, $\|Q^{-1}\|_{L_\infty}$, and $r_0$.
If $\phi, \psi \in L_2(\mathbb{R}^d)$, then
$\| z_\varepsilon(\cdot,\tau) - z_0(\cdot,\tau)\|_{L_2(\mathbb{R}^d)} \to 0$ as $\varepsilon \to 0$, $\tau \in \mathbb{R}$.
\end{proposition}

\section{Application of the general results:
the system of elasticity theory}

\subsection{The operator of elasticity theory\label{elast}}
Let $d \ge 2$. We represent the operator of elasticity theory as in~\cite[Chapter~5, Section~2]{BSu1}.
Let $\zeta$~be an orthogonal second rank tensor in $\mathbb{R}^d$; in the standard orthonormal basis in $\mathbb{R}^d$,
 it can be represented by a matrix $\zeta = \{\zeta_{jl}\}_{j,l = 1}^d$.
 We consider \emph{symmetric} tensors $\zeta$, which are identified with vectors $\zeta_* \in \mathbb{C}^m$, $2m = d(d+1)$,
 by the following rule. The vector $\zeta_*$ is formed by all components $\zeta_{jl}$,  $j \le l$,
 and the pairs $(j, l)$ are put in order in some fixed way.

For the \emph{displacement vector} $\mathbf{u} \in H^1 (\mathbb{R}^d; \mathbb{C}^n)$, we introduce the deformation tensor
$e (\mathbf{u}) = \frac{1}{2} \left\lbrace \frac{\partial u_j}{\partial x_l} + \frac{\partial u_l}{\partial x_j} \right\rbrace$.
Let $e_* (\mathbf{u})$~be the vector corresponding to the tensor
$e (\mathbf{u})$. The relation $b(\mathbf{D}) \mathbf{u} = -i e_* (\mathbf{u})$
determines an $(m \times d)$-matrix homogeneous DO $b(\mathbf{D})$ uniquely; the symbol of this DO
is a matrix with real entries. 
For instance, with an appropriate ordering, we have
\begin{equation*}
b(\boldsymbol{\xi}) = \begin{pmatrix}
\xi_1 & 0 \\
\frac{1}{2} \xi_2 & \frac{1}{2} \xi_1 \\
0 & \xi_2
\end{pmatrix}, \quad d=2; \quad
b(\boldsymbol{\xi}) = \begin{pmatrix}
\xi_1 & 0 & 0 \\
\frac{1}{2} \xi_2 & \frac{1}{2} \xi_1 & 0 \\
0 & \xi_2 & 0 \\
0 & \frac{1}{2} \xi_3 & \frac{1}{2} \xi_2 \\
0 & 0 & \xi_3 \\
\frac{1}{2} \xi_3 & 0 & \frac{1}{2} \xi_1
\end{pmatrix}, \quad d=3.
\end{equation*}
The constants $\alpha_0, \alpha_1$ (see \eqref{rank_alpha_ineq}) depend only on $d$.

Let $\sigma(\mathbf{u})$~be the \emph{stress tensor}, and let $\sigma_*(\mathbf{u})$~be the corresponding vector.
In the accepted way of writing, the \emph{Hooke law} about proportionality of stresses and deformations can be expressed by
the relation
$\sigma_*(\mathbf{u}) = g (\mathbf{x}) e_* (\mathbf{u})$,
where $g (\mathbf{x})$~is an $(m \times m)$-matrix-valued function with real entries
(which gives a \textquotedblleft concise\textquotedblright \ description of the Hooke tensor).
The matrix $g (\mathbf{x})$ characterizes the parameters of the elastic (in general, anisotropic) medium under study.
We assume that the matrix-valued function $g (\mathbf{x})$~is $\Gamma$-periodic and such that
$g (\mathbf{x}) > 0$, and $g, g^{-1} \in L_\infty$.

The energy of elastic deformations is given by the quadratic form
\begin{equation}
\label{elast_energy}
w[\mathbf{u},\mathbf{u}] = \frac{1}{2} \int_{\mathbb{R}^d} \left\langle \sigma_*(\mathbf{u}), e_* (\mathbf{u}) \right\rangle_{\mathbb{C}^m} d \mathbf{x} = \frac{1}{2} \int_{\mathbb{R}^d} \left\langle g(\mathbf{x}) b(\mathbf{D}) \mathbf{u}, b(\mathbf{D}) \mathbf{u} \right\rangle_{\mathbb{C}^m} d \mathbf{x},
\quad
\mathbf{u} \in H^1 (\mathbb{R}^d; \mathbb{C}^d).
\end{equation}
The operator $\mathcal{W}$ generated by this form is
the \emph{operator of elasticity theory}. Thus, the operator
$2 \mathcal{W} = b(\mathbf{D})^* g b(\mathbf{D}) = \widehat{\mathcal{A}}$
is of the form~\eqref{hatA} with $n = d$ and $m = d(d + 1)/2$.

In the case of isotropic medium, the matrix $g(\mathbf{x})$ depends only on two functional \emph{Lame parameters} $\lambda(\mathbf{x})$ and $\mu(\mathbf{x})$. The parameter $\mu(\mathbf{x})$~is the \emph{shear modulus}.
Often, another parameter $K(\mathbf{x})$ is introduced instead of $\lambda(\mathbf{x})$;
$K(\mathbf{x})$ is called the  \emph{modulus of volume compression}. We need yet another modulus $\beta(\mathbf{x})$.
Here are the relations:
$K(\mathbf{x}) = \lambda(\mathbf{x}) + \frac{2 \mu(\mathbf{x})}{d}$,  $\beta(\mathbf{x}) = \mu(\mathbf{x}) + \frac{\lambda(\mathbf{x})}{2}$.
The modulus $\lambda(\mathbf{x})$ may be negative. In the isotropic case, the conditions that ensure the positive definiteness of the matrix  $g(\mathbf{x})$ are as follows: $\mu(\mathbf{x}) \ge \mu_0 > 0$, $K(\mathbf{x}) \ge K_0 > 0$.
As an example, we write down the matrix $g$ in the isotropic case for $d = 2, 3$:
\begin{gather*}
g_{\mu, K} (\mathbf{x}) = \begin{pmatrix}
K(\mathbf{x}) + \mu(\mathbf{x}) & 0 & K(\mathbf{x}) - \mu(\mathbf{x}) \\
0 & 4 \mu(\mathbf{x}) & 0 \\
K(\mathbf{x}) - \mu(\mathbf{x}) & 0 & K(\mathbf{x}) + \mu(\mathbf{x})
\end{pmatrix}, \quad d = 2, \\
g_{\mu, K} (\mathbf{x}) = \begin{pmatrix}
K(\mathbf{x}) + \frac{4}{3}\mu(\mathbf{x}) & 0 & K(\mathbf{x}) - \frac{2}{3}\mu(\mathbf{x}) & 0 & K(\mathbf{x}) - \frac{2}{3}\mu(\mathbf{x}) & 0 \\
0 & 4\mu(\mathbf{x}) & 0 & 0 & 0 & 0 \\
K(\mathbf{x}) - \frac{2}{3} \mu(\mathbf{x}) & 0 & K(\mathbf{x}) + \frac{4}{3}\mu(\mathbf{x}) & 0 & K(\mathbf{x}) - \frac{2}{3}\mu(\mathbf{x}) & 0 \\
0 & 0 & 0 & 4 \mu(\mathbf{x}) & 0 & 0 \\
K(\mathbf{x}) - \frac{2}{3}\mu(\mathbf{x}) & 0 & K(\mathbf{x}) - \frac{2}{3} \mu(\mathbf{x}) & 0 & K(\mathbf{x}) + \frac{4}{3}\mu(\mathbf{x}) & 0 \\
0 & 0 & 0 & 0 & 0 & 4 \mu(\mathbf{x})
\end{pmatrix}, \quad d = 3.
\end{gather*}

\subsection{Homogenization of the Cauchy problem for the elasticity system}
\label{elast_homog_section}

Consider the operator $\mathcal{W}_\varepsilon = \frac{1}{2} \widehat{\mathcal{A}}_\varepsilon$ with rapidly oscillating
 coefficients. The effective matrix $g^0$ and the effective operator $\mathcal{W}^0 = \frac{1}{2} \widehat{\mathcal{A}}^0$
 are defined by the general rules (see~\eqref{g0}, \eqref{g_tilde},  \eqref{hatA0}).

Let $Q(\mathbf{x})$~be a $\Gamma$-periodic symmetric $(d \times d)$-matrix-valued function with real entries such that
$Q(\mathbf{x}) > 0$; $Q, Q^{-1} \in L_\infty$.
Usually $Q(\mathbf{x})$~is a scalar function (the density of the medium).
We assume that $Q(\mathbf{x})$ is a matrix-valued function.
Consider the following \emph{Cauchy problem for the elasticity system}:
\begin{equation}
\label{elasticity_Cauchy}
\left\{
\begin{aligned}
& Q^\varepsilon (\mathbf{x}) \frac{\partial^2 \mathbf{u}_\varepsilon (\mathbf{x}, \tau)}{\partial \tau^2} = - \mathcal{W}_\varepsilon \mathbf{u}_\varepsilon (\mathbf{x}, \tau), \\
& \mathbf{u}_\varepsilon (\mathbf{x}, 0) = \boldsymbol{\phi} (\mathbf{x}), \quad  Q^\varepsilon (\mathbf{x}) \frac{\partial \mathbf{u}_\varepsilon }{\partial \tau} (\mathbf{x}, 0) = \boldsymbol{\psi}(\mathbf{x}),
\end{aligned}
\right.
\end{equation}
where $\boldsymbol{\phi}, \boldsymbol{\psi} \in L_2 (\mathbb{R}^d; \mathbb{C}^d)$~are given functions.
(For simplicity, we consider the homogeneous equation.)
The homogenized problem takes the form
\begin{equation}
\label{elasticity_Cauchy_eff}
\left\{
\begin{aligned}
& \overline{Q} \frac{\partial^2 \mathbf{u}_0 (\mathbf{x}, \tau)}{\partial \tau^2} = - \mathcal{W}^0 \mathbf{u}_0 (\mathbf{x}, \tau), \\
& \mathbf{u}_0 (\mathbf{x}, 0) = \boldsymbol{\phi} (\mathbf{x}), \quad  \overline{Q}\frac{\partial \mathbf{u}_0 }{\partial \tau} (\mathbf{x}, 0) = \boldsymbol{\psi}(\mathbf{x}).
\end{aligned}
\right.
\end{equation}

Applying Theorem~\ref{nonhomog_Cauchy_wQ_thrm_1}, we obtain the following result.

\begin{proposition}\label{pr161}
Under the assumptions of Subsection~\emph{\ref{elast}}, let ${\mathbf u}_\varepsilon$ be the solution of problem~\eqref{elasticity_Cauchy},
and let ${\mathbf u}_0$ be the solution of problem~\eqref{elasticity_Cauchy_eff}.
If $\boldsymbol{\phi} \in H^s(\mathbb{R}^d;\mathbb{C}^d)$ and $\boldsymbol{\psi} \in H^r(\mathbb{R}^d;\mathbb{C}^d)$ with some $0 \le s \le 2$, $0\le r \le 1$, then for $\tau \in \mathbb{R}$ and $\varepsilon >0$ we have
 \begin{equation}
\label{elast_general_est}
        \| {\mathbf u}_\varepsilon(\cdot, \tau) - {\mathbf u}_0 (\cdot, \tau) \|_{L_2 (\mathbb{R}^d)} \le
        {\mathfrak{C}}_1 (s) (1+| \tau|)^{s/2} \varepsilon^{s/2} \| \boldsymbol{\phi} \|_{H^s(\mathbb{R}^d)} +
 {\mathfrak{C}}_2 (r) (1+ |\tau|) \varepsilon^r \| \boldsymbol{\psi} \|_{H^r(\mathbb{R}^d)},
\end{equation}
where ${\mathfrak{C}}_1 (s)$ depends on $s$, $d$, $\|g\|_{L_\infty}$, $\|g^{-1}\|_{L_\infty}$,
$\|Q\|_{L_\infty}$, $\|Q^{-1}\|_{L_\infty}$, and $r_0${\rm ;}
${\mathfrak{C}}_2 (r)$ depends on $r$, $d$, $\|g\|_{L_\infty}$, $\|g^{-1}\|_{L_\infty}$,
$\|Q\|_{L_\infty}$, $\|Q^{-1}\|_{L_\infty}$, and $r_0$.
If $\boldsymbol{\phi}, \boldsymbol{\psi} \in L_2(\mathbb{R}^d;\mathbb{C}^d)$, then
$\| {\mathbf u}_\varepsilon(\cdot,\tau) - {\mathbf u}_0(\cdot,\tau)\|_{L_2(\mathbb{R}^d)} \to 0$ as $\varepsilon \to 0$, $\tau \in \mathbb{R}$.
\end{proposition}

Even in the isotropic case, in general estimate~\eqref{elast_general_est} cannot be improved (see Subsection~\ref{example}).
Recall also Example~\ref{elast_exmpl_N0_ne_0} which confirms the sharpness of the result in the anisotropic case.

 \subsection{Example}
\label{example}
    Consider the system of isotropic elasticity in the twodimensional case assuming that $K$ and $\mu$ are periodic and depend only on $x_1$.
    Now $d=2$, $m=3$, $n=2$, $\Gamma = (2 \pi \mathbb{Z})^2$, and $Q(\mathbf{x}) = \mathbf{1}$.

    The periodic solution $\Lambda(\mathbf{x})$ of equation~\eqref{equation_for_Lambda} is given by
    \begin{equation*}
    \Lambda (\mathbf{x}) = \begin{pmatrix}
    \Lambda_{11} (x_1) & 0 & \Lambda_{13} (x_1) \\
    0 & \Lambda_{22} (x_1) & 0
    \end{pmatrix},
    \end{equation*}
where $\Lambda_{11}$, $\Lambda_{22}$, and $\Lambda_{13}$ are $(2\pi {\mathbb Z})$-periodic solutions of the following equations
    \begin{align}
    \notag
    D_1 (K(x_1) + \mu(x_1)) (D_1 \Lambda_{11}(x_1) + 1) &= 0, \qquad \int_0^{2\pi} \Lambda_{11}(x_1) \, d x_1 = 0, \\
    \label{isotr_elast_exmp_Lambda22_eq}
    D_1 \left(2 \mu(x_1) \left(\frac{1}{2} D_1 \Lambda_{22}(x_1) + 1 \right)  \right) &= 0, \qquad \int_0^{2\pi} \Lambda_{22}(x_1) \, d x_1 = 0,  \\
    \notag
    D_1 \left( (K(x_1) + \mu(x_1)) D_1 \Lambda_{13}(x_1) + K(x_1) - \mu(x_1) \right) &= 0,
    \qquad \int_0^{2\pi} \Lambda_{13}(x_1) \, d x_1 = 0.
    \end{align}
    Hence,
    \begin{gather*}
    (K + \mu) (D_1 \Lambda_{11} + 1) = \underline{(K + \mu)},
    \quad \mu \left( \frac{1}{2} D_1 \Lambda_{22} + 1 \right) = \underline{\mu}, \\
    (K + \mu) D_1 \Lambda_{13} + K - \mu = \underline{(K + \mu)} \overline{\left(\frac{K - \mu}{K + \mu} \right) }.
    \end{gather*}
        Then the matrix $\widetilde{g} (\mathbf{x})$ (see~\eqref{g_tilde}) is given by
    \begin{equation*}
    \widetilde{g} (\mathbf{x}) = \begin{pmatrix}
    \underline{(K + \mu)} & 0 & \underline{(K + \mu)} \overline{\left(\frac{K - \mu}{K + \mu} \right) } \\
    0 & 4 \underline{\mu} & 0 \\
    \underline{(K + \mu)} \left(\frac{K - \mu}{K + \mu} \right) & 0 & \frac{4 K \mu}{K + \mu} + \frac{K - \mu}{K + \mu} \overline{\left(\frac{K - \mu}{K + \mu} \right) }  \underline{(K + \mu)}
    \end{pmatrix}.
    \end{equation*}
    The effective matrix $g^0$~(see~\eqref{g0}) takes the form
    \begin{align*}
    &g^0 = \begin{pmatrix}
    A & 0 & B \\
    0 & C & 0 \\
    B & 0 & E
    \end{pmatrix},
  \\
    A = \underline{(K + \mu)}, \quad  B = \underline{(K + \mu)}  &\overline{\left(\frac{K - \mu}{K + \mu} \right) }, \quad   C = 4 \underline{\mu},
\quad
 E = 4 \overline{\left( \frac{K \mu}{K + \mu} \right) } + \left(\overline{\frac{K - \mu}{K + \mu}} \right)^2 \underline{(K + \mu)}.
    \end{align*}
    The spectral germ is given by
    \begin{equation}
    \label{isotr_elast_exmp_germ}
    \widehat{S}(\boldsymbol{\theta}) = \begin{pmatrix}
    A\theta_1^2  + \frac{1}{4}C\theta_2^2   &   \left( B + \frac{1}{4}C \right)\theta_1 \theta_2  \\
     \left( B + \frac{1}{4}C \right)\theta_1 \theta_2  & E\theta_2^2  + \frac{1}{4}C\theta_1^2
    \end{pmatrix}.
    \end{equation}
    In order to construct an example where $B + \frac{1}{4}C = 0$, we put
    \begin{equation*}
    \mu(x_1) = 1 + 624 \cdot \cos^2 x_1, \quad K(x_1)
= a + 100 \cdot
\begin{cases}
-1,  & \text{if} \  x_1 < \frac{\pi}{2}\\
 1, &  \text{if} \  x_1 \ge \frac{\pi}{2}.
\end{cases}
    \end{equation*}
    Then
    \begin{equation*}
    A^{-1} =  \frac{1}{4 q} + \frac{3}{4 \rho},
\quad
    B =  \frac{6 (a + b) q + 2(a - b) \rho - 4 q \rho}{\rho + 3 q },
\quad
    C = 4\sqrt{c+1},\quad
    E = \frac{6 b \rho - 6 b q  - 12 b^2 + 4 q \rho}{\rho  + 3q }.
    \end{equation*}
    Here $b = 100$, $c=624$, $q= \sqrt{(a-b+c+1)(a-b+1)}$,
and $\rho = \sqrt{(a+b+c+1)(a+b+1)}$. Thus, relation $B + \frac{1}{4} C = 0$ holds if $a$ satisfies the equation
    \begin{equation*}
    \frac{6 (a + b) q  + 2(a - b) \rho - 4 q \rho}{ \rho + 3 q}  = -25.
    \end{equation*}
    The root exists, since the left-hand side is continuous in $a$ and, by calculations,
    is (approximately) equal to $-34.4$ for $a= 130$, and is equal to $-22.1$ for $a=150$.
    The approximate value of the root is $a \approx 145.6581$.
    For such $a$ the eigenvalues of the germ~\eqref{isotr_elast_exmp_germ} coincide if
    $A \theta_1^2 + \frac{1}{4}C \theta_2^2 = E \theta_2^2 + \frac{1}{4} C \theta_1^2$.
    Since $ \theta_1^2 +  \theta_2^2 = 1$, this is valid for
    \begin{equation}
    \label{isotr_elast_exmp_theta}
    \theta_1^2 = \frac{E - \frac{1}{4}C}{A + E - \frac{1}{2}C} \approx 0.5394.
    \end{equation}

    Next, we calculate $L(\boldsymbol{\theta})$ and $\widehat{N} (\boldsymbol{\theta})$ (see~\eqref{N(theta)}, \eqref{L(theta)}):
    \begin{equation*}
    L(\boldsymbol{\theta}) =
    \begin{pmatrix}
    0 & S \theta_2 & 0 \\
    S^* \theta_2 & 0 & T^* \theta_2 \\
    0 & T \theta_2 & 0
    \end{pmatrix},
    \quad \widehat{N} (\boldsymbol{\theta}) = \frac{1}{2}\begin{pmatrix}
    0 &  S \theta_1^2 \theta_2 + T^* \theta_2^3 \\
    S^* \theta_1^2 \theta_2 + T \theta_2^3 & 0
    \end{pmatrix},
    \end{equation*}
    where
    \begin{equation*}
    S = \underline{K + \mu} \overline{\left(\frac{K - \mu}{K + \mu} \right) \Lambda_{22}}, \quad T = \overline{\left( \frac{4 K \mu}{K + \mu} +  \frac{K - \mu}{K + \mu} \overline{\left(\frac{K - \mu}{K + \mu} \right)} \underline{K + \mu} \right) \Lambda_{22} },
    \end{equation*}
    and $\Lambda_{22}$ is the solution of equation~\eqref{isotr_elast_exmp_Lambda22_eq}:
    \begin{equation*}
    \Lambda_{22}(x_1) =
\begin{cases}
2i \arctan\left(\frac{1}{25} \tan(x_1) \right) - 2 i x_1 , & \text{if} \; 0 \le x_1 < \frac{\pi}{2}\\
    2i \arctan\left(\frac{1}{25} \tan(x_1) \right) - 2 i x_1 + 2 \pi i , &  \text{if} \; \frac{\pi}{2} \le x_1 < \frac{3\pi}{2}\\
    2i \arctan\left(\frac{1}{25} \tan(x_1) \right) - 2 i x_1 + 4 \pi i , &  \text{if} \; \frac{3 \pi}{2} \le  x_1 \le 2 \pi
\end{cases}.
    \end{equation*}
    Approximately, we have $S \approx 65.6650i$, $T \approx 76.2833i$.

    For the points $\boldsymbol{\theta}^{(j)}, j=1,2,3,4$, satisfying~\eqref{isotr_elast_exmp_theta}, we have $\widehat{\gamma}_1 (\boldsymbol{\theta}^{(j)}) = \widehat{\gamma}_2 (\boldsymbol{\theta}^{(j)})$ and $\widehat{N} (\boldsymbol{\theta}^{(j)}) = \widehat{N}_0 (\boldsymbol{\theta}^{(j)}) \ne 0$. The numbers $\pm \widehat{\mu}$, where $\widehat{\mu}$ is approximately equal to $0.09850$, are the
eigenvalues of the operator  $\widehat{N} (\boldsymbol{\theta}^{(j)})$.
    Applying Theorem~\ref{s<2_cos_thrm_Rd}, we confirm that the result of Proposition~\ref{pr161}  is sharp.

\subsection{The Hill body}
In mechanics (see,~e.g.,~\cite{ZhKO}), the elastic isotropic medium with constant shear modulus
$\mu(\mathbf{x}) = \mu_0 = \mathrm{const}$ is called the Hill body.
In this case, a simpler factorization for the operator $\mathcal{W}$ is possible (see~\cite[Chapter~5, Subsection~2.3]{BSu1}).
The form~\eqref{elast_energy} can be written as
\begin{equation}
\label{elast_energy_Hill}
w[\mathbf{u},\mathbf{u}] = \int_{\mathbb{R}^d} \left\langle g_\wedge(\mathbf{x}) b_\wedge(\mathbf{D}) \mathbf{u}, b_\wedge(\mathbf{D}) \mathbf{u} \right\rangle_{\mathbb{C}^{m_\wedge}} d \mathbf{x}.
\end{equation}
Here $m_\wedge = 1 + d(d-1)/2$. The $(m_\wedge \times d)$-matrix $b_\wedge(\boldsymbol{\xi})$ can be described as follows.
The first row of $b_\wedge(\boldsymbol{\xi})$ is $(\xi_1, \xi_2, \ldots, \xi_d)$.
The other rows correspond to pairs of indices $(j, l), 1 \le j <l \le d$.
The element in the $(j, l)$th row and the $j$th column is $\xi_l$, and
the element in the $(j, l)$th row and the $l$th column is $(-\xi_j)$; all other elements
of the $(j, l)$th row are equal to zero. The order of the rows is irrelevant. Finally,
$g_\wedge (\mathbf{x}) = \mathrm{diag} \{\beta(\mathbf{x}), \mu_0/2, \mu_0/2, \ldots, \mu_0/2\}$.
Thus, by~\eqref{elast_energy_Hill},
$\mathcal{W} = b_\wedge(\mathbf{D})^* g_\wedge (\mathbf{x}) b_\wedge(\mathbf{D})$.
The effective matrix $g_\wedge^0$ coincides with $\underline{g_\wedge}$, i.~e.,
$g_\wedge^0 = \underline{g_\wedge} = \mathrm{diag} \{\underline{\beta}, \mu_0/2, \mu_0/2, \ldots, \mu_0/2\}$; see~\cite[Chapter~5, Subsection~2.3]{BSu1}.
The effective operator is given by
\begin{equation}
\label{elast_W^0_Hill}
\mathcal{W}^0 = b_\wedge(\mathbf{D})^* g_\wedge^0 b_\wedge(\mathbf{D}).
\end{equation}
For the new factorization of the operator $\mathcal{W}_\varepsilon = b_\wedge(\mathbf{D})^* g_\wedge^\varepsilon (\mathbf{x}) b_\wedge(\mathbf{D})$, the statement of the Cauchy problem~\eqref{elasticity_Cauchy} remains the same.
In the homogenized problem~\eqref{elasticity_Cauchy_eff}, the effective operator $\mathcal{W}^0$ is given by~\eqref{elast_W^0_Hill}.
 For the Hill body, the solutions of problem~\eqref{elasticity_Cauchy}
 satisfy all the statements of Subsection~\ref{elast_homog_section}.

Suppose that $Q(\mathbf{x}) = \mathbf{1}$. By Proposition~\ref{N=0_proposit}($3^\circ$),
we see that $\widehat{N}(\boldsymbol{\theta}) = 0$ for any $\boldsymbol{\theta} \in \mathbb{S}^{d-1}$.
Then, by Theorem~\ref{nonhomog_Cauchy_hatA_eps_ench_thrm_1}, for
$\boldsymbol{\phi}\in H^s (\mathbb{R}^d; \mathbb{C}^d)$ and $\boldsymbol{\psi} \in H^r (\mathbb{R}^d; \mathbb{C}^d)$,
where $0 \le s \le 3/2$, $0 \le r \le 1/2$, we have
\begin{equation*}
\| \mathbf{u}_\varepsilon (\cdot, \tau) - \mathbf{u}_0 (\cdot, \tau) \|_{L_2(\mathbb{R}^d)} \le
\widehat{\mathfrak{C}}_3 (s) (1+ |\tau|)^{2s/3} \varepsilon^{2s/3} \| \boldsymbol{\phi} \|_{H^s(\mathbb{R}^d)} +
 \widehat{\mathfrak{C}}_4 (r) (1+ |\tau|) \varepsilon^{2r} \| \boldsymbol{\psi} \|_{H^r(\mathbb{R}^d)},
\end{equation*}
where $\widehat{\mathfrak{C}}_3 (s)$ depends
on $s$, $d$, $\|\beta\|_{L_\infty}$, $\|\beta^{-1}\|_{L_\infty}$, $\mu_0$, and $r_0${\rm ;}
$\widehat{\mathfrak{C}}_4 (r)$ depends
on $r$, $d$, $\|\beta\|_{L_\infty}$, $\|\beta^{-1}\|_{L_\infty}$, $\mu_0$, and $r_0$.


\begin{thebibliography}{XXXX}


\bibitem[BaPa]{BaPa}Bakhvalov N.~S., Panasenko G.~P., \textit{Homogenization: Averaging processes
in periodic media. Mathematical problems in mechanics of composite materials}, "Nauka", Moscow, 1984;
English transl., Math. Appl. (Soviet Ser.), vol. 36, Kluwer Acad. Publ. Group, Dordrecht, 1989.


\bibitem[BeLP]{BeLP} Bensoussan~A., Lions~J.-L., Papanicolaou~G.,
\emph{Asymptotic analysis for periodic structures}, Stud. Math. Appl., vol. 5,
North-Holland Publishing Co., Amsterdam-New York, 1978.


\bibitem[BSu1]{BSu1}Birman M.~Sh., Suslina T.~A., \textit{Second order periodic
differential operators. Threshold properties and homogenization}, Algebra i Analiz
{\bf 15} (2003), no. 5, 1-108; English transl., St.~Petersburg Math. J. {\bf 15} (2004),
no. 5, 639--714.

\bibitem[BSu2]{BSu2}Birman M.~Sh., Suslina T.~A., \textit{Threshold approximations with corrector for the resolvent
of a factorized operator family},
Algebra i Analiz \textbf{17} (2005), no. 5, 69--90; English transl., St.~Petersburg Math. J. {\bf 17} (2006), no. 5, 745--762.


\bibitem[BSu3]{BSu3}Birman M.~Sh., Suslina T.~A.,
\textit{Homogenization with corrector term for periodic elliptic differential
operators}, Algebra i Analiz {\bf 17} (2005), no.~6, 1--104;
English transl., St.~Petersburg Math. J. {\bf 17} (2006), no.~6, 897--973.

\bibitem[BSu4]{BSu4}Birman M.~Sh., Suslina T.~A.,
\textit{Homogenization with corrector term for periodic differential
operators. Approximation of solutions in the Sobolev class  $H^1(\mathbb{R}^d)$},
Algebra i Analiz  {\bf 18} (2006), no.~6, 1--130; English transl.,
St.~Petersburg  Math. J. {\bf 18} (2007), no.~6, 857--955.


\bibitem[BSu5]{BSu5}Birman M.~Sh., Suslina T.~A., \textit{Operator error estimates in the homogenization problem
for nonstationary periodic equations},
Algebra i Analiz \textbf{20} (2008), no. 6, 30--107; English transl.,
St.~Petersburg  Math. J. {\bf 20} (2009), no.~6, 873--928.

\bibitem [COrVa] {COrVa}  Conca~C., Orive~R., Vanninathan~M., \textit{Bloch approximation in homogenization and applications}, SIAM J. Math. Anal.~\textbf{33} (2002), no.~5, 1166--1198.

\bibitem[DSu]{DSu} Dorodnyi~M.~A., Suslina~T.~A., \textit{Homogenization of the hyperbolic equations},
Funktsional. Analiz i ego Prilozhen. {\bf 50} (2016), no. 4, 91--96; English transl.,
Funct. Anal. Appl. {\bf 50} (2016), no.~4, 319--324.


\bibitem[Ka]{Ka} Kato T., \textit{Perturbation theory for linear operators}, Springer-Verlag, Berlin, 1995.

    \bibitem [Kr] {Kr} Vilenkin~N.~Ya., Krein~S.~G. et al., \textit{Functional Analysis}, Groningen (Netherlands): Wolters-Noordhoff Publishing, 1972.

\bibitem[M]{M} Meshkova~Yu.~M., \textit{On operator error estimates for homogenization of hyperbolic systems with periodic
coefficients}, Preprint (2017), arXiv:1705.02531v3.


\bibitem[Se]{Se} Sevost'yanova~E.~V.,
{\em Asymptotic expansion of the solution of a second-order elliptic equation  with periodic rapidly oscillating coefficients}, Matem. Sbornik \textbf{115} (1981), no.~2, 204--222; English transl., Math. USSR-Sb. \textbf{43}(1982), no.~2, ~181--198.


\bibitem[Su1]{Su1} Suslina~T.~A., \textit{On homogenization of  periodic parabolic systems},
Funktsional. Analiz i ego Prilozhen. {\bf 38} (2004), no. 4, 86--90; English transl., Funct. Anal. Appl.
{\bf 38} (2004), no. 4, 309--312.

\bibitem[Su2]{Su2} Suslina~T.~A., \textit{Homogenization of a periodic parabolic Cauchy problem},
Nonlinear Equations and Spectral Theory, Amer. Math. Soc. Transl. Ser. 2, vol. 220, Amer. Math. Soc., Providence, RI, 2007, pp. 201--233.

\bibitem[Su3]{Su3} Suslina~T.~A., \textit{Homogenization of a periodic parabolic Cauchy problem in the Sobolev space $H^1(\mathbb{R}^d)$},
Math. Model. Nat. Phenom. {\bf 5} (2005), no.~4, 390--447.

\bibitem[Su4]{Su4} Suslina~T.~A., \textit{Spectral approach to homogenization of nonstationary Schr\"odinger-type equations}, J.~Math.~Anal.~Appl. {\bf 446} (2017), no. 2, 1466--1523.

\bibitem[Su5]{Su5} Suslina~T.~A., \textit{Homogenization of the Schr\"odinger-type equations},
Funktsional. Analiz i ego Prilozhen. {\bf 50} (2016), no. 3, 90--96; English transl., Funct. Anal. Appl.
{\bf 50} (2016), no. 3, 241--246.



\bibitem[V]{V} Vasilevskaya~E.~S., \textit{A periodic parabolic Cauchy problem:
Homogenization with corrector}, Algebra i Analiz {\bf 21} (2009) no.~1, 3--60; English transl.,
St.~Petersburg  Math. J. {\bf 21} (2010), no.~1, 1--41.

\bibitem[VSu]{VSu} Vasilevskaya~E.~S.,  Suslina~T.~A., \textit{Homogenization of parabolic and elliptic periodic operators in
$L_2({\mathbb R}^d)$ with the first and second correctors taken into account},
Algebra i Analiz {\bf 24} (2012), no.~2, 1--103; English transl.,
St.~Petersburg  Math. J. {\bf 24} (2013), no.~2, 185--261.

\bibitem[Zh]{Zh}Zhikov V.~V., \textit{On some estimates of homogenization theory},
Dokl. Ros. Akad. Nauk \textbf{406} (2006), no. 5, 597-601; English transl., Dokl. Math. \textbf{73}
(2006), 96--99.


\bibitem[ZhKO]{ZhKO}Zhikov V.~V., Kozlov S.~M., Olejnik O.~A., \textit{Homogenization of
differential operators}, "Nauka", Moscow, 1993; English transl., Springer-Verlag, Berlin, 1994.


\bibitem[ZhPas1]{ZhPas1} Zhikov~V.~V., Pastukhova~S.~E.,
\textit{On operator estimates for some problems in homogenization theory},
Russ. J. Math. Phys. \textbf{12} (2005), no. 4, 515-524.

\bibitem[ZhPas2]{ZhPas2} Zhikov~V.~V., Pastukhova~S.~E.,
\textit{Estimates of homogenization for a parabolic equation with periodic
coefficients},  Russ. J. Math. Phys. \textbf{13} (2006), no.~2, 224--237.

\bibitem[ZhPas3]{ZhPas3} Zhikov~V.~V., Pastukhova~S.~E.,
\textit{Operator estimates in  homogenization theory}, Uspekhi  Matem. Nauk \textbf{71} (2016), no.~3, 27--122; English transl.,
Russ. Math. Surveys \textbf{71} (2016), no.~3, 417--511.



\end{thebibliography}
\end{document}